\documentclass[11pt,oneside]{amsart}

\usepackage{amsmath,ifthen, amsfonts, amssymb,
srcltx, amsopn, color, enumerate, hyperref, morefloats}
\usepackage[cmtip,arrow]{xy}
\usepackage{pb-diagram, pb-xy}
\usepackage{overpic}

\usepackage[T1]{fontenc}

\newcommand{\showcomments}{yes}
\renewcommand{\showcomments}{no}

\newsavebox{\commentbox}
{\ifthenelse{\equal{\showcomments}{yes}}%
{\footnotemark
        \begin{lrbox}{\commentbox}
        \begin{minipage}[t]{1.25in}\raggedright\sffamily\tiny
        \footnotemark[\arabic{footnote}]}
{\begin{lrbox}{\commentbox}}}%
{\ifthenelse{\equal{\showcomments}{yes}}%
{\end{minipage}\end{lrbox}\marginpar{\usebox{\commentbox}}}
{\end{lrbox}}}

\newcommand{\visual}{\partial_{vis}}

 \usepackage{mathabx}
\newcommand{\propnest}{\sqsubsetneq}

\newcounter{intronum}
\newcounter{claimnum}
\newcounter{ax}
\newtheorem{thm}{Theorem}[section]

\newtheorem{lem}[thm]{Lemma}

\newtheorem{cor}[thm]{Corollary}

\newtheorem{prop}[thm]{Proposition}
\newtheorem{propi}[intronum]{Proposition}

\newtheorem{question}[intronum]{Question}
\newtheorem{defnn}[intronum]{Definition}
\newtheorem{thmn}[intronum]{Theorem}

\theoremstyle{definition}
\newtheorem{defn}[thm]{Definition}
\newtheorem{rem}[thm]{Remark}

\newtheorem{exmp}[thm]{Example}

\newtheorem{notation}[thm]{Notation}

\newtheorem{claim}[claimnum]{Claim}

\newtheorem{claim*}{Claim}

\DeclareMathOperator{\kernel}{ker}
\DeclareMathOperator{\image}{im}

\DeclareMathOperator{\Aut}{Aut}

\DeclareMathOperator{\stabilizer}{Stab}
\DeclareMathOperator{\diam}{diam}

\DeclareMathOperator{\BIG}{Big}
\DeclareMathOperator{\Fix}{Fix}

\newcommand{\neb}{\mathcal N}
\newcommand{\TT}{\mathcal{T}}
\newcommand{\MCG}{\mathcal{MCG}}

\newcommand{\field}[1]{\mathbb{#1}}
\newcommand{\integers}{\ensuremath{\field{Z}}}

\newcommand{\naturals}{\ensuremath{\field{N}}}
\newcommand{\reals}{\ensuremath{\field{R}}}
\newcommand{\Euclidean}{\ensuremath{\field{E}}}

\newcommand{\boundary}{{\ensuremath \partial}}

\newcommand{\interior} [1] {{\ensuremath \text{\rm Int}(#1) }}

\makeatletter

\newcommand{\Rmnum}[1]{\mathbf{{\expandafter\@slowromancap\romannumeral #1@}}}

\newcommand{\simp}{\ensuremath{\partial_{_{\triangle}}}}

\newcommand{\contact}[1]{\ensuremath{\mathcal C#1}}

\makeatother

\newcommand{\tup}[1]{\vec{\mathrm{#1}}}

\let\oldmarginpar\marginpar
\renewcommand\marginpar[1]{\-\oldmarginpar[\raggedleft\footnotesize #1]%
{\raggedright\footnotesize #1}}

\newcommand{\ignore}[2]{\big\{\{{#1}\}\big\}_{#2}}

\DeclareMathOperator{\supp}{\mathrm{Supp}}

\newcounter{enumitemp}

\newcommand{\dist}{\textup{\textsf{d}}}

\newcommand{\cuco}[1]{\mathcal{#1}}

\newcommand{\fontact}{\mathcal C}

\newcommand{\gate}{\mathfrak g}
\newcommand{\nest}{\sqsubseteq}
\newcommand{\orth}{\perp}
\newcommand{\north}{\not\perp}
\newcommand{\transverse}{\pitchfork}

\newcommand{\support}{\mathrm{Supp}}
\newcommand{\remote}[2]{\partial_{#1}^{\mathrm{rem}}#2}
\newcommand{\id}{\mathrm{id}}

\newcommand{\Teich}{\mathcal T}

\newcommand{\PML}{\mathbb{P}\mathcal{ML}}

\newcommand{\Active}{\mathfrak A}
\newcommand{\mma}{\mathfrak M}

\long\def\Restate#1#2#3#4{
\medskip\par\noindent
{\bf #1 \ref{#2} #3} {\it #4}\par\medskip }

 \def\notorth{\:\ensuremath{\reflectbox{\rotatebox[origin=c]{90}{$\nvdash$}}}}

\begin{document}
\title[Boundaries of HHS]{Boundaries and automorphisms of hierarchically hyperbolic spaces}

\author[M.G. Durham]{Matthew G. Durham}
\address{U. Michigan, Ann Arbor, Michigan, USA}
\email{durhamma@umich.edu}
\thanks{\flushleft {Durham was supported by the National Science Foundation under Grant Number NSF 1045119.}}

\author[M.F. Hagen]{Mark F. Hagen}
\address{DPMMS, University of Cambridge, Cambridge, UK}
\email{markfhagen@gmail.com}
\thanks{\flushleft {Hagen was supported by NSF Grant Number 1045119 and by the EPSRC grant of Henry Wilton.}}

\author[A. Sisto]{Alessandro Sisto}
\address{ETH, Z\"{u}rich, Switzerland}
\email{sisto@math.ethz.ch}
\thanks{\flushleft{Sisto was supported by the Swiss National Science Foundation project 144373.}}

\setcounter{tocdepth}{1}

\maketitle

\begin{abstract}
Hierarchically hyperbolic spaces provide a common framework for studying mapping class groups of finite type surfaces, Teichm\"uller space, right-angled Artin groups, and many other cubical groups.  Given such a space $\cuco X$, we build a bordification of $\cuco X$ compatible with its hierarchically hyperbolic structure.  If $\cuco X$ is proper, e.g. a hierarchically hyperbolic group such as the mapping class group, we get a compactification of $\cuco X$; we also prove that our construction generalizes the Gromov boundary of a hyperbolic space.  In our first main set of applications, we introduce a notion of geometrical finiteness for hierarchically hyperbolic subgroups of hierarchically hyperbolic groups in terms of boundary embeddings.  As primary examples of geometrical finiteness, we prove that the natural inclusions of finitely generated Veech groups and the Leininger-Reid combination subgroups extend to continuous embeddings of their Gromov boundaries into the boundary of the mapping class group, both of which fail to happen with the Thurston compactification of Teichm\"uller space.  Our second main set of applications are dynamical and structural, built upon our classification of automorphisms of hierarchically hyperbolic spaces and analysis of how the various types of automorphisms act on the boundary.  We prove a generalization of the Handel-Mosher ``omnibus subgroup theorem'' for mapping class groups to all hierarchically hyperbolic groups, obtain a new proof of the Caprace-Sageev rank-rigidity theorem for many CAT(0) cube complexes, and identify the boundary of a hierarchically hyperbolic group as its Poisson boundary; these results rely on a theorem detecting \emph{irreducible axial} elements of a group acting on a hierarchically hyperbolic space (which generalize pseudo-Anosov elements of the mapping class group and rank-one isometries of a cube complex not virtually stabilizing a hyperplane).
\end{abstract}

\tableofcontents

\section*{Introduction}\label{sec:intro}
The class of hierarchically hyperbolic spaces (HHS) was introduced in~\cite{BehrstockHagenSisto:HHS_I}, and given a streamlined definition in~\cite{BehrstockHagenSisto:HHS_II}, to provide a common framework for studying cubical groups and mapping class groups of surfaces.  The definition was motivated by the observation that, under natural hypotheses, a CAT(0) cube complex is equipped with a collection of projections to hyperbolic spaces obeying rules reminiscent of the hierarchical structure of mapping class groups and projections to curve graphs introduced by Masur and Minsky in~\cite{MasurMinsky:I,MasurMinsky:II}.  The class of HHS includes the aforementioned spaces (mapping class groups and many CAT(0) cube complexes, including all universal covers of compact special cube complexes), along with Gromov-hyperbolic spaces, Teichm\"uller space with any of the usual metrics, and many others; see~\cite{BehrstockHagenSisto:HHS_I,BehrstockHagenSisto:HHS_II,BehrstockHagenSisto:HHS_III} for an account of the current scope of the theory. 

Much of the utility of HHS comes from the fact that many features of Gromov-hyperbolic spaces have natural generalizations in the HHS world.  Since one of the most useful objects associated to a hyperbolic space is its Gromov boundary, we provide here a generalization of the Gromov boundary to hierarchically hyperbolic spaces.  The boundary of a hierarchically hyperbolic space is inspired by various boundaries associated to the salient examples of HHS, e.g. the simplicial boundary of a CAT(0) cube complex and the Thurston compactification of Teichm\"uller space, projective measured lamination space $\PML(S)$.  

Just as the Gromov boundary does for hyperbolic spaces and groups, the HHS boundary provides considerable information about the geometry of an HHS and the dynamics of its automorphisms; our aim in this paper is to explore some of these properties.

\subsection*{Introduction to HHS}\label{subsec:intro_HHS}
We first briefly and softly recall the HHS theory.  A \emph{hierarchically hyperbolic space} is a pair $(\cuco X,\mathfrak S)$ equipped with some additional data: $\cuco X$ is a quasigeodesic metric space and $\mathfrak S$ is an index set equipped with a partial order $\nest$, called \emph{nesting}, with a unique maximal element $S$.  There is also an \emph{orthogonality} relation on $\mathfrak S$; when $\mathfrak S$ is the set of essential subsurfaces of a surface $S$, up to isotopy, orthogonality is just disjointness.  We often call elements of $\mathfrak S$ \emph{domains}.

Each $U\in\mathfrak S$ is equipped with a uniformly hyperbolic space $\fontact U$ and a coarse map $\pi_U:\cuco X\to\fontact U$.  There are also \emph{relative projections} $\rho^U_V$, which are coarse maps $\fontact U\to\fontact V$ defined unless $U,V$ are orthogonal.  In the case where $\cuco X$ is the marking complex of the surface $S$ and $\mathfrak S$ is the set of subsurfaces of $S$, then the associated hyperbolic spaces are the curve graphs of these subsurfaces and the projections are subsurface projections.  We impose other rules reminiscent of the hierarchical structure of the mapping class group; see Definition~\ref{defn:space_with_distance_formula}. 

The \emph{distance formula} is crucial: for any $x,y\in\cuco X$, the distance $\dist_{\cuco X}(x,y)$ differs, up to bounded multiplicative and additive error, from the sum of the distances $\dist_{\fontact U}(\pi_U(x),\pi_U(y))$ as $U\in\mathfrak S$ varies over those domains where that distance exceeds some predefined threshold~\cite{BehrstockHagenSisto:HHS_II}.

Just as quasiconvexity is vital to the study of hyperbolic spaces, \emph{hierarchical quasiconvexity} is central in the study of HHS.  Roughly, $\cuco Y\subseteq\cuco X$ is hierarchically quasiconvex if $\pi_U(\cuco Y)$ is uniformly quasiconvex for each $U\in\mathfrak S$, and any point in $\cuco X$ projecting under $\pi_U$ close to $\pi_U(\cuco Y)$ for each $U$ must lie close (in $\cuco X$) to $\cuco Y$.  The fundamental example of a hierarchically quasiconvex subspaces is the \emph{standard product region} $P_U$ associated to each $U\in\mathfrak S$.  Roughly, the subspace $P_U$ consists of those points $x\in\cuco X$ where $\pi_V(x)$ is close to $\rho^U_V$ for any $V\in\mathfrak S$ that is not orthogonal to, or nested in, $V$.  The factor of $P_U$ obtained by fixing, in addition, the projections to domains orthogonal to $U$ (and allowing movement in domains nested in $U$) is denoted $F_U$, and the other factor is $E_U$.  A familiar example here is the region of Teichm\"uller space with the Teichm\"uller metric where the boundary curves of some subsurface $U$ are short: Minsky \cite{Min96} proved that these so-called thin parts are quasiisometric to products of the Teichm\"uller spaces of the complementary subsurfaces, one of which is $U$. 

\subsubsection*{What's needed from~\cite{BehrstockHagenSisto:HHS_I,BehrstockHagenSisto:HHS_II}}
The paper~\cite{BehrstockHagenSisto:HHS_II} is the main foundational paper in the theory of HHS.  In the current paper, we use most of the background material developed in~\cite{BehrstockHagenSisto:HHS_II}, with the notable exception of the combination theorems.  In particular, we use the main definition of HHS (which is equivalent to, but much simpler than, the original definition from~\cite{BehrstockHagenSisto:HHS_II}), the realization theorem, the distance formula, and the existence of hierarchy paths.  The fact that mapping class groups are HHG, which is crucial for our applications to Veech and Leininger-Reid subgroups in Section~\ref{sec:extending} could be deduced from~\cite{MasurMinsky:I,MasurMinsky:II,BKMM:consistency,Behrstock:asymptotic}, but is also given a streamlined proof in~\cite[Section~11]{BehrstockHagenSisto:HHS_II}.  

From~\cite{BehrstockHagenSisto:HHS_I}, we need the acylindricity result (Theorem~14.3) and, for the purposes of Section~\ref{sec:cube_complex}, the HHS structure on CAT(0) cube complexes.  We note that the acylindricity result from~\cite{BehrstockHagenSisto:HHS_I} is independent of the other HHS results in that paper.

Finally, the recent paper~\cite{BehrstockHagenSisto:HHS_III} is completely independent of this one.\footnote{The picture at \url{http://www.wescac.net/HHS_infographic.pdf} shows the current state of the theory, indicating the main concepts and results and their interdependencies.}

\subsection*{The boundary}\label{subsec:boundary}
Consider an HHS $(\cuco X,\mathfrak S)$.  Since any two points of $\cuco X$ are joined by a \emph{hierarchy path} -- a uniform quasigeodesic projecting to a uniform unparametrized quasigeodesic in $\fontact U$ for each $U\in\mathfrak S$ (see~\cite{BehrstockHagenSisto:HHS_II}) -- a natural approach to constructing a boundary is to imitate the construction of the Gromov boundary, or the visual boundary of a CAT(0) space: boundary points would be asymptotic classes of ``hierarchy rays'' emanating from a fixed basepoint, and one might imagine topologizing this set by defining two boundary points to be close if the corresponding rays stay close ``for a long time''.  

The boundary construction is motivated by this intuition.  Given a hierarchy ray $\gamma:\naturals\to\cuco X$, one first observes that the set of $U\in\mathfrak S$ for which $\pi_U\circ\gamma$ is unbounded is a pairwise-orthogonal collection -- $\gamma$ either spends a bounded amount of time in each standard product region, or $\gamma$ wanders permanently into the (coarse) intersection of several standard product regions.  Accordingly, the underlying set of the boundary $\boundary(\cuco X,\mathfrak S)$ is the set of formal linear combinations $p=\sum_{U\in\mathfrak U}a_Up_U$, where $\mathfrak U\subset\mathfrak S$ (the \emph{support} of $p$) is a pairwise-orthogonal set, each $p_U$ is a point in the Gromov boundary of $\fontact U$, each $a_U\in(0,1]$, and $\sum_Ua_U=1$.  

Regarding each $\boundary\fontact U$ as a discrete set, the above construction yields a (highly disconnected, locally infinite) simplicial complex.  The ``rank-one hierarchy rays'' -- i.e. the points of $\boundary\fontact S$ -- correspond to isolated $0$--simplices, while the standard product regions contribute boundary subcomplexes isomorphic to simplicial joins.  This complex is a kind of ``Tits boundary'' for $(\cuco X,\mathfrak S)$.  The actual boundary we define is related to this complex in much the same way that the visual boundary of a CAT(0) space is related to the Tits boundary; we define the boundary $\boundary(\cuco X,\mathfrak S)$ by imposing a coarser topology, described in Section~\ref{sec:boundary_definition}.  (When the context is clear, we denote $\boundary(\cuco X,\mathfrak S)$ by $\boundary\cuco X$, being mindful that this space depends, as far as we know, on the particular HHS structure $\mathfrak S$.)  

The resulting space $\overline{\cuco X}=\cuco X\cup\boundary{\cuco X}$ is Hausdorff and separable; $\boundary\cuco X$ is a closed subset and $\cuco X$ is dense (Proposition~\ref{prop:properties}).  Moreover, the Gromov boundary $\boundary\fontact U$ embeds in $\boundary (\cuco X,\mathfrak S)$, in the obvious way, for each $U\in\mathfrak S$, by Theorem~\ref{thm:hyperbolic}.  Crucially:

\Restate{Theorem}{thm:cpt}{(Compactness)}{Let $(\cuco X,\mathfrak S)$ be a hierarchically hyperbolic space with $\cuco X$ proper.  Then $\overline{\cuco X}$ is compact.}

The definition of $\boundary(\cuco X,\mathfrak S)$ is given strictly in terms of $\mathfrak S$ and the accompanying hyperbolic spaces and projections; the standing assumption that $(\cuco X,\mathfrak S)$ is \emph{normalized} -- each $\pi_U$ is coarsely surjective -- connects the boundary to the space $\cuco X$ by ensuring that $\cuco X$ is dense in $\overline{\cuco X}$.  Even so, it is not clear whether the homeomorphism type of $\boundary(\cuco X,\mathfrak S)$ depends on the particular choice of HHS structure:

\begin{question}\label{question:boundary_invariance}
Let $(\cuco X,\mathfrak S)$ be a hierarchically hyperbolic space and let $(\cuco X,\mathfrak S')$ be a different hierarchically hyperbolic structure on the same space.  Does the identity $\cuco X\to\cuco X$ extend to a map $\cuco X\cup\boundary(\cuco X,\mathfrak S)\to\cuco X\cup\boundary(\cuco X,\mathfrak S')$ which restricts to a homeomorphism of boundaries?
\end{question}

A positive answer to Question~\ref{question:boundary_invariance} would stand in contrast to the situation for CAT(0) spaces.  For example, the right-angled Artin group $A$, presented by a path of length $3$, famously has the property that the universal cover $\widetilde X$ of the Salvetti complex can be endowed with different CAT(0) metrics (obtained by perturbing angles in the $2$--cells) with non-homeomorphic visual boundaries~\cite{CrokeKleiner}.  On the other hand, $\widetilde X$ admits a hierarchically hyperbolic structure $(\widetilde X,\mathfrak S)$ coming from the cubical structure of $\widetilde X$ (with no dependence on the CAT(0) metric).  Perturbing the CAT(0) metric within its quasiisometry type does not change the HHS structure (and hence the HHS boundary), so the HHS boundary is in a sense more ``canonical'' than the visual boundary in this example (and indeed for all CAT(0) cube complexes with \emph{factor systems}, which we discuss in more detail below).

\subsection*{Automorphisms and their actions on the boundary}\label{subsec:autos_intro}
An \emph{automorphism} of $(\cuco X,\mathfrak S)$ is a bijection $g:\mathfrak S\to\mathfrak S$ along with an isometry $\fontact U\to\fontact g(U)$ for each $U\in\mathfrak S$ which satisfy certain compatibility conditions.  The distance formula ensures that automorphisms induce uniform quasi-isometries of $\cuco X$, so the group $\Aut(\mathfrak S)$ of automorphisms uniformly quasi-acts by (uniform) quasi-isometries on $\cuco X$.  The (quasi-)action of $\Aut(\mathfrak S)$ on $\cuco X$ extends to an action on $\overline{\cuco X}$ restricting to an action by homeomorphisms on $\boundary\cuco X$ (Corollary~\ref{cor:auto extend}).

In one of the main cases of interest, $\cuco X$ is a Cayley graph of a finitely-generated group $G$, and the action of $G$ on itself by left multiplication corresponds to an action on $(G,\mathfrak S)$ by HHS automorphisms.  In this situation, if the action on $\mathfrak S$ is cofinite, then $(G,\mathfrak S)$ is a \emph{hierarchically hyperbolic group structure}; if a group $G$ admits a hierarchically hyperbolic group structure, then $G$ is a \emph{hierarchically hyperbolic group}.  The archetypal hierarchically hyperbolic group is the mapping class group of a connected, oriented surface of finite type~\cite[Section 11]{BehrstockHagenSisto:HHS_II}.  Other examples include many cubical groups~\cite{BehrstockHagenSisto:HHS_I}, many graphs of hierarchically hyperbolic groups~\cite{BehrstockHagenSisto:HHS_II}, and certain quotients of hierarchically hyperbolic groups~\cite{BehrstockHagenSisto:HHS_III}.  If $(G,\mathfrak S)$ is a hierarchically hyperbolic group, then the isometric action of $G$ on itself by left multiplication extends to an action by homeomorphisms on $\overline{G}$ (Corollary~\ref{cor:hhg_boundary}).  We describe in detail below our results regarding the dynamics and structure of groups of automorphisms.

\subsection*{Embeddings of subspace boundaries and geometrical finiteness}\label{subsec:embeddings between boundaries}

A desirable property of a boundary is that inclusions of subspaces that are ``convex'' in an appropriate sense induce embeddings of boundaries with closed images.  In Section~\ref{sec:extending}, we show that hierarchically quasiconvex subspaces of $\cuco X$, which admit their own natural HHS structures~\cite{BehrstockHagenSisto:HHS_II}, have this property: if $\cuco Y\subset\cuco X$ is hierarchically quasiconvex, then $\cuco Y$ has a limit set in $\boundary\cuco X$ which is homeomorphic to $\boundary\cuco Y$ with the HHS structure inherited from $\cuco X$.  In fact, Theorem~\ref{thm:extension to boundary maps} provides more, by giving natural conditions on maps between HHS ensuring that they extend continuously to the HHS boundary.  This motivates the following definition:

\begin{defnn}[Geometrical finiteness]\label{defn:gf}
We say a hierarchically hyperbolic subgroup $H$ of a hierarchically hyperbolic group $G$ is \emph{geometrically finite} if the natural inclusion $\iota:H \hookrightarrow G$ extends continuously to an $H$-equivariant embedding $\partial \iota:\partial H \hookrightarrow \partial G$.   
\end{defnn}

In what follows, we will be interested in developing this notion and establishing examples in the context of the mapping class group of a finite type surface.

\subsection*{Comparison of the mapping class group boundary with $\PML(S)$}
The archetypal hierarchically hyperbolic group is the mapping class group $\MCG(S)$ of a connected, oriented surface $S$ of finite type.  The hierarchically hyperbolic structure is provided by results of~\cite{MasurMinsky:I,MasurMinsky:II,BKMM:consistency,Aougab:uniform,HenselPrzytyckiWebb:unicorn,Bowditch:uniform,PrzytyckiSisto:universe,ClayRafiSchleimer:uniform,Behrstock:asymptotic,Mangahas:UU,Webb:BGI} and is discussed in detail in Section~11 of~\cite{BehrstockHagenSisto:HHS_II}.  Roughly, $\mathfrak S$ is the set of essential subsurfaces of $S$, up to isotopy, $\fontact U$ is the curve graph of $U$ for each $U\in\mathfrak S$, and projections are usual subsurface projections.

Traditionally, $\MCG(S)$ has been studied via its action on Teichm\"uller space $\TT(S)$ with its Thurston compactification by $\PML(S)$.  This approach has been fruitful especially when considering subgroups of $\MCG(S)$ defined via flat or hyperbolic geometry.  Nonetheless, the $\MCG(S)$ action on $\TT(S)$ is not cocompact and the orbits of many subgroups (in fact, any with Dehn twists) are distorted in $\TT(S)$, which make $\TT(S)$ imperfect for studying the coarse geometry of $\MCG(S)$ and its subgroups.

The situation is further complicated when one attempts to extend the $\MCG(S)$ action on $\TT(S)$ to its various boundaries.  Teichm\"uller geodesics are unique and thus geodesic rays based at a point form a natural visual compactification of $\TT(S)$, but Kerckhoff \cite{Ker80} proved that it is basepoint dependent and thus the $\MCG(S)$ action fails to extend continuously.  While Thurston \cite{Thur88} defined a compactification via $\PML(S)$ to which the $\MCG(S)$ action does extend continuously, Thurston's compactification is defined via hyperbolic geometry and the Teichm\"uller metric is defined via flat geometry, which leads to an incoherence between the internal geometry and its asymptotics in $\PML(S)$ \cite{Masur:twoboundaries, Len, LLR, CMW, brock2016limit}.

The boundary $\boundary(\MCG(S),\mathfrak S)$ provides the first compactification of $\MCG(S)$ so that the action of $\MCG(S)$ on itself by left multiplication extends to a continuous action on the boundary with the dynamical properties we discuss below (see also Section~\ref{subsubsec:auts_and_HHG}).  While many of these dynamical properties were originally proven via the $\MCG(S)$-action on $\TT(S)$ with its Thurston compactification, many of the pathologies described above vanish in our construction, as we discuss presently.  

\subsection*{On geometrically finite subgroups of $\MCG(S)$}
Problem~5 of~\cite{Hamenstadt:problems} and Section 6 of~\cite{Mosher:problems} in Farb's book \cite{farb2006problems} regard the development of a notion of geometrical finiteness for subgroups of $\MCG(S)$.  Mosher suggests a definition that requires an external proper hyperbolic space $X$ on which the candidate subgroup acts with a collection of cusp subgroups in some appropriate sense; geometric finiteness would then require that $X$ and $\partial X$ embed quasiisometrically in $\TT(S)$ and continuously in $\PML(S)$, respectively.  Masur's theorem makes it unreasonable to expect a simultaneous continuous embedding $X \cup \partial X \to \TT(S) \cup \PML(S)$.

We will argue that replacing $\TT(S) \cup \PML(S)$ with $\MCG(S) \cup \boundary \MCG(S)$ as in Definition \ref{defn:gf} generates a robust theory of geometrical finiteness.  In particular, we prove:

\begin{thmn}\label{thm:gf}
Suppose that $H<\MCG(S)$ is one of the following:
\begin{enumerate}
\item The standard embedding of $\MCG(Y)$ for some proper subsurface $Y \subset S$;
\item Convex cocompact in the sense of \cite{FarbMosher};
\item A finitely generated Veech group;
\item A Leininger-Reid combination subgroup \cite{leininger2006combination}.
\end{enumerate}
Then $H$ is a geometrically finite subgroup of $\MCG(S)$.
\end{thmn}

Hence geometrical finiteness generalizes convex cocompactness for subgroups of $\MCG(S)$ to a broader class of groups.   Theorem \ref{thm:gf}(a) is proven in Theorem \ref{thm:subsurface boundary embedding} and Theorem \ref{thm:gf}(b) is Theorem \ref{thm:convex cocompact}.  We discuss presently the Veech and Leininger-Reid examples in more detail.

%
%

\subsubsection*{Veech and Leininger-Reid combinations subgroups} 
For Mosher (see Problem~6.1 of \cite{Mosher:problems}), the main test cases for a definition of geometrical finiteness for subgroups of mapping class groups are finitely generated Veech groups and the Leininger-Reid subgroups.  It is worth noting that while the former are explicitly defined via flat geometry and the latter somewhat less so, the aforementioned coherence pathologies between the Teichm\"uller geometry and the Thurston compactification give an obstruction to considering embeddings of natural boundaries associated to them into $\PML(S)$.  We prove that this obstruction disappears with $\partial \MCG(S)$.  We now briefly give some background.

Given a holomorphic quadratic differential $q$ on $S$, there is an associated copy of $\mathbb{H}^2$ called a Teichm\"uller disk, $TD(q)$, which is a convex subset of $\TT(S)$.  The stabilizer of $TD(q)$ in $\MCG(S)$ is $\mathrm{Aff}(q)$, those elements with a representative which act by affine homemorphisms with respect to the flat metric determined by $q$.  A Veech group $V$ is a subgroup of $\mathrm{Aff}(q)$ which acts properly on $TD(q)$; we consider only finitely-generated Veech groups. The visual boundary of $TD(q)$ is naturally identified by $\PML(q)$ which admits a natural embedding in $\PML(S)$ that parametrizes the limit set of $V$ in $\PML(S)$ \cite{KentLein:subgroupgeom}, but a theorem of Masur \cite{Masur:twoboundaries} implies that this embedding does not give an everywhere continuous extension $TD(q) \cup \PML(q) \hookrightarrow \TT(S) \cup \PML(S)$.

In ~\cite{leininger2006combination}, Leininger-Reid construct subgroups of $\MCG(S)$ which are combinations of Veech groups; some are surface groups in which all but one conjugacy class is pseudo-Anosov.  The \emph{boundary} of such a surface subgroup is its limit set in $\partial \mathbb{H}^2$.  Problem 3.3 of \cite{Reid:problems} asks if there is a continuous, equivariant embedding of this boundary into $\PML(S)$.

While we do not answer this question directly, we do prove something strictly stronger for $\partial \MCG(S)$:

\Restate{Theorem}{thm:leininger_reid_limit_set}{}
{Let $H< \MCG(S)$ be either a finitely generated Veech or Leininger-Reid subgroup as above.  Then the inclusion $H \hookrightarrow \MCG(S)$ extends to a continuous $H$-equivariant embedding $\boundary H \hookrightarrow \boundary \MCG(S)$ with closed image.  In particular, $H$ is a geometrically finite subgroup of $\MCG(S)$.}

\subsubsection*{Other candidates for geometrical finiteness}
Perhaps the next best candidates for geometrically finite subgroups of $\MCG(S)$ are the various right-angled Artin groups constructed by Clay-Leininger-Mangahas \cite{CLM} and Koberda \cite{Kob}.  These subgroups are HHGs and the former are even known to be quasiisometrically embedded in $\MCG(S)$. 

\begin{question}
Are the Clay-Leininger-Mangahas and Koberda right-angled Artin subgroups of $\MCG(S)$ geometrically finite?  Hierarchically quasiconvex?\footnote{Since we initially posted this paper, Mousley answered this question negatively in~\cite{Mousley:nomap}.}
\end{question}

\subsubsection*{The HHS boundary of Teichm\"uller space and $\PML(S)$}

 Slight modifications of the above hierarchical structures endow the Teichm\"uller space $\Teich(S)$, with either the Teichm\"uller or Weil-Petersson metrics, with an HHS structure, as explained in~\cite{BehrstockHagenSisto:HHS_I,BehrstockHagenSisto:HHS_II} using results of~\cite{Brock:pants,Durham:augmented,EskinMasurRafi:large_scale_rank}; see also~\cite{Bowditch:large_scale_rank,Bowditch:wp} for closely-related results. 

\begin{question}\label{question:teichmuller_boundary}
How is the HHS boundary $\boundary\Teich(S)$ of $\Teich(S)$, with the Teichm\"uller metric and the above HHS structure, related to the projective measured lamination space $\PML(S)$?
\end{question}

In fact, there is a natural map $\PML(S)\to\boundary\Teich(S)$ which collapses certain simplices of measures on given laminations to points, while being injective on the set of uniquely ergodic laminations, whose image in $\boundary\Teich(S)$ can be identified with a subset of $\boundary\fontact S\subset\boundary\Teich(S)$.  A promising strategy is to attempt to use this map, along with a result of Edwards~\cite{Edwards,Daverman:book}, to prove that $\boundary\Teich(S)$ is homeomorphic to $\PML(S)$, i.e. to $\mathbb S^{2\xi(S)-1}$.  The missing ingredient is a positive answer to:

\begin{question}\label{question:disjoint_discs}
Does $\boundary\Teich(S)$ have the \emph{disjoint discs property}?
\end{question}

A metric space $M$ has the disjoint disks property if any two maps $D^2\to M$ admit arbitrarily small perturbations with disjoint image; the above question makes sense since it is not hard to see, using Proposition~\ref{prop:properties}, that $\boundary\Teich(S)$ is metrizable.  The difficulty here involves nonuniquely ergodic laminations, which cause a similar problem to the extensions discussed above related to the Leininger-Reid subgroups.

Another question, subject to much recent study, is about the limit sets of Teichm\"uller geodesics in Thurston's compactification.  The analogous question in our setting is:

\begin{question}\label{question:teich limit}
What are the limit sets of Teichm\"uller geodesics in $\partial \TT(S)$?
\end{question}

There are now several constructions of geodesics with limits sets that are bigger than a point \cite{Len, LLR, CMW, brock2016limit}, but these constructions fundamentally depend on the fact that filling minimal laminations can admit simplices of measures, which collapse in $\partial \TT(S)$.  The geodesics constructed in \cite{LLR, CMW, brock2016limit} will have unique limits $\partial \TT(S)$ as their asymptotics with respect to $\partial \TT(S)$ are determined by their asymptotics in the curve graph $\fontact S$.  On the other hand, the situation becomes more opaque for Teichm\"uller geodesics with vertical laminations with multiple components.  Using work of Rafi \cite{Rafi:hypinteich}, one can determine that the coefficients $a_Y$ of the components $Y \subset S$ supporting the potential limits in $\partial \TT(S)$ are determined by limits of ratios of the rates of divergence in the various subsurface curve graphs $\fontact Y$.  However, it seems unlikely that these limits of ratios always exist, suggesting that such geodesics need not have unique limits in $\partial \TT(S)$.

\subsection*{Dynamical and structural results}

Our second main collection of applications of the boundary are about the dynamics of the action on the boundary and the structure of subgroups.  In Section~\ref{subsubsec:auts_and_HHG}, we study automorphisms of hierarchically hyperbolic spaces:

\subsubsection*{Classification of automorphisms}\label{subsubsec:classification_intro}
Given $f\in\Aut(\mathfrak S)$, the set $\BIG(f)$ of $U\in\mathfrak S$ for which $\langle f\rangle\cdot x$ (for some basepoint $x\in\cuco X$) projects to an unbounded set in $\fontact U$ is a possibly empty finite set of pairwise-orthogonal domains preserved by the action of $\langle f\rangle$ on $\mathfrak S$.  We classify $f$ according to the nature of $\BIG(f)$.  First, if $\BIG(f)=\emptyset$, then $f$ has bounded orbits in each $\fontact U$ and hence has bounded orbits in $\cuco X$, by Proposition~\ref{prop:elliptic}; in this case, $f$ is \emph{elliptic}. Second, if $\langle f\rangle\cdot x$ projects to a quasi-line in $\fontact U$ for some $U\in\BIG(f)$, then $\langle f\rangle\cdot x$ is a quasi-line in $\cuco X$, by Proposition~\ref{prop:axial}, and $f$ is \emph{axial}.  Otherwise, $f$ is \emph{distorted}.

If $\BIG(f)=\{S\}$, then $f$ is \emph{irreducible}, and $f$ is \emph{reducible} otherwise.  Perhaps the most important class of HHS automorphisms are irreducible axial automorphisms.  In the mapping class group, these are the pseudo-Anosov elements; in a hierarchically hyperbolic cube complex, these are the rank-one elements that do not virtually preserve hyperplanes~\cite{BehrstockHagenSisto:HHS_I,Hagen:boundary}.  In the case where $(G,\mathfrak S)$ is a hierarchically hyperbolic group, each irreducible axial element is Morse -- this follows from Theorem~\ref{thm:acyl_morse} -- but the converse does not hold.  The question of when irreducible axial elements exist is of major interest later.

\subsubsection*{Dynamics and fixed points}\label{subsubsec:dynamics_and_fixed_points_intro}
In Section~\ref{subsec:dynamics}, we study the dynamics of $f\in\Aut(\mathfrak S)$ on $\boundary\cuco X$.  First, we show that irreducible axial automorphisms act as expected:

\Restate{Proposition}{prop:irreducible axial}{(North-south dynamics)}{If $g \in \Aut(\mathfrak S)$ is irreducible axial, then $g$ has exactly two fixed points $\lambda_+, \lambda_- \in \partial \cuco X$.  Moreover, for any boundary neighborhoods $\lambda_+ \in U_{+}$ and $\lambda_- \in U_{-}$, there exists an $N>0$ such that $g^N(\partial \cuco X - U_-) \subset U_+$.}

In Proposition~\ref{prop:irreducible distorted} and Proposition~\ref{prop:irreducible distorted dynamics}, we show that if $f$ is irreducible distorted, then $f$ fixes a unique point $p\in\boundary\cuco X$, which is an ``attracting fixed point''.  We also prove analogues of these results for reducible automorphisms (Propositions~\ref{prop:reducible big set} and~\ref{prop:exterior}).

We then study hierarchically hyperbolic groups.  First, we rule out distortion:

\Restate{Theorem}{thm:HHGs have no distorteds}{(Coarse semisimplicity)}{If $(G, \mathfrak S)$ is a hierarchically hyperbolic group, then each $g \in G$ is either elliptic or axial; in fact $g$ is undistorted in each element of $\BIG(g)$.}

In the event that $G$ contains irreducible axial elements, we have:

\Restate{Theorem}{thm:dense orbit}{(Topological transitivity)}{Let $(G,\mathfrak S)$ be hierarchically hyperbolic with an irreducible axial element and let $G$ be nonelementary.  Then any $G$--orbit in $\boundary G$ is dense.}

Below, we will describe when $(G,\mathfrak S)$ has an irreducible axial element.

\subsection*{Uses of the boundary}\label{subsec:main_applications_intro}
We use the boundary, and actions thereon, in numerous ways.

\subsubsection*{Finding and exploiting irreducible axials}\label{subsubsec:r_r_intro}
In Section~\ref{sec:rank_rigidity}, we study irreducible axial elements of groups of automorphisms of hierarchically hyperbolic spaces.  The setting is an HHS $(\cuco X,\mathfrak S)$ with $\cuco X$ proper and $\mathfrak S$ countable, and we consider a countable subgroup $G\leq\Aut(\mathfrak S)$.  This holds, for example, when $\cuco X=G$ is an HHG.  The main technical statement is:

\medskip\par\noindent
{\bf Propositions \ref{prop:alternative_HHG_version},\ref{prop:alternative_non_parabolic_version} (Finding irreducible axials)} \emph{Suppose that either $G$ acts properly and coboundedly on $\cuco X$ and cofinitely on $\mathfrak S$, or $G$ acts with unbounded orbits in $\cuco X$ and no fixed point in $\boundary\fontact S$.  Then either $G$ contains an irreducible axial element, or there exists $U\in\mathfrak S-\{U\}$ which is fixed by a finite-index subgroup of $G$.}\par\medskip

These two propositions are proved in tandem.  The strategy is to consider probability measures on $G$ and corresponding $G$--stationary measures on $\boundary\cuco X$; the main lemma, Lemma~\ref{lem:fixed_point} shows that, unless $G$ has a finite orbit in $\boundary\fontact S$ or $\mathfrak S-\{S\}$, such a measure must be supported on $\boundary\fontact S\subset\boundary\cuco X$.  In particular, if $\fontact S$ is bounded, then there must be a finite orbit in $\mathfrak S-\{S\}$. We emphasize that, for the above proposition and all of its applications, compactness of the HHS boundary (i.e. Theorem~\ref{thm:cpt}) is absolutely vital.

Using the above propositions, we prove:

\Restate{Theorem}{thm:tits}{(HHG Tits alternative)}{Let $(G, \mathfrak S)$ be an HHG and let $H\leq G$.  Then $H$ either contains a nonabelian free group or is virtually abelian.}

By analyzing supports of global fixed points in the boundary of an HHS, we then prove:

\Restate{Theorem}{thm:ost hhg}{(Omnibus Subgroup Theorem)}{Let $(G,\mathfrak S)$ be a hierarchically hyperbolic group and let $H\leq G$.  Then there exists an element $g \in H$ with $\Active(H)= \BIG(g)$.  Moreover, for any $g' \in H$ and each $U \in \BIG(g')$, there exists $V \in \BIG(g)$ with $U \nest V$.}

Here, $\Active(H)$ is the set of domains $U$ on which $H$ has unbounded projection. The theorem we actually prove is more general than the above, but the version stated here is sufficient to imply the Omnibus Subgroup Theorem for mapping class groups, due to Handel-Mosher~\cite{HandelMosher:omnibus}, which they proved as an umbrella theorem for several subgroup structure theorems, including the Tits alternative; see also \cite{Mangahas:SR} for further discussion.

We also obtain a coarse/HHS version of the rank-rigidity conjecture for CAT(0) spaces:

\medskip\par\noindent
{\bf Theorems \ref{thm:rank_rigidity_non_parabolic},\ref{thm:rank_rigidity_HHG} (Coarse rank-rigidity)} \emph{Let $(\cuco X,\mathfrak S)$ be an HHS with $\cuco X$ unbounded and proper and $\mathfrak S$ countable.  Let $G\leq\Aut(\mathfrak S)$ be a countable subgroup and suppose that one of the following holds:
\begin{enumerate}
 \item $G$ acts essentially on $\cuco X$ with no fixed point in $\boundary\cuco X$;
 \item $G$ acts properly and coboundedly on $\cuco X$ and cofinitely on $\mathfrak S$.
\end{enumerate}
Then either $(\cuco X,\mathfrak S)$ is a \emph{product HHS with unbounded factors} or there exists an axial element $g\in G$ such that $\BIG(g)$ consists of a single domain $W$ such that $\fontact U$ is bounded if $U\orth W$.}\par\medskip

Such an element $g$ is a \emph{rank-one automorphism}; all of its quasigeodesic axes of any fixed quality lie in some neighborhood of one another (of radius depending on the quality).  The HHS is a \emph{product with unbounded factors} if there exists $U\in\mathfrak S$ such that $\cuco X$ coarsely coincides with the standard product region $P_U$, and each of $E_U,F_U$ is unbounded.  

In particular, if $\cuco X$ is any of the cube complexes shown in~\cite{BehrstockHagenSisto:HHS_I} to be hierarchically hyperbolic (i.e. those admitting ``factor-systems''), then our methods allow us to recover the Caprace-Sageev rank-rigidity theorem from~\cite{CapraceSageev:rank_rigidity} for $\cuco X$: 

\Restate{Corollary}{cor:rr_cubes}{(Rank-rigidity for many cube complexes)}{Let $\cuco X$ be a CAT(0) cube complex with a factor-system.  Let $G$ act on $\cuco X$ and suppose that one of the following holds:
\begin{enumerate}
 \item $\cuco X$ is unbounded and $G$ acts on $\cuco X$ properly and cocompactly;
 \item $G$ acts on $\cuco X$ with no fixed point in $\cuco X\cup\simp\cuco X$.
\end{enumerate}
Then $\cuco X$ contains a $G$--invariant convex subcomplex $\cuco Y$ such that either $G$ contains a rank-one isometry of $\cuco Y$ or $\cuco Y=\cuco A\times\cuco B$, where $\cuco A$ and $\cuco B$ are unbounded convex subcomplexes. }

It is difficult to construct cube complexes without factor-systems that satisfy the remaining hypotheses of this theorem.  At least in the cocompact case, we believe that our proof works without explicitly hypothesizing the existence of a factor system -- see Question~A of~\cite{BehrstockHagenSisto:HHS_II}, which asks whether the presence of a geometric group action on a cube complex guarantees that a factor system exists (see Remark~\ref{rem:factor_system}).\footnote{After we initially posted this paper, Hagen and Susse showed that every CAT(0) cube complex with a geometric group action admits a factor system and is thus hierarchically hyperbolic~\cite{HagenSusse}.}

\subsection*{Other applications, examples, and questions}\label{subsubsec:other_applications_intro}

\subsubsection*{The HHS boundary in the cubical case}\label{subsec:cubical_intro}
If $\cuco X$ is a CAT(0) cube complex with a factor-system $\mathfrak F$ (here $\mathfrak F$ more properly denotes the set of parallelism classes of elements of the factor system), then the resulting hierarchically hyperbolic structure (which is fundamentally derived from the hyperplanes of $\cuco X$ and how they interact) has a boundary which is, perhaps unsurprisingly, closely related to the \emph{simplicial boundary} $\simp\cuco X$ introduced in~\cite{Hagen:boundary} (which is derived from how certain infinite families of hyperplanes interact).  Specifically:

\Restate{Theorem}{thm:simplicial_HHS}{(Simplicial and HHS boundaries)}{Let $\cuco X$ be a CAT(0) cube complex with a factor system $\mathfrak F$, and let $(\cuco X,\mathfrak F)$ be the associated hierarchically hyperbolic structure.  There is a topology $\mathcal T$ on the simplicial boundary $\simp\cuco X$ so that:
\begin{enumerate}
 \item There is a homeomorphism $b:(\simp\cuco X,\mathcal T)\to\boundary(\cuco X,\mathfrak F)$,
 \item for each component $C$ of the simplicial complex $\simp\cuco X$, the inclusion $C\hookrightarrow(\simp\cuco X,\mathcal T)$ is an embedding.
\end{enumerate}
In particular, if $\mathfrak F,\mathfrak F'$ are factor systems on $\cuco X$, then $\boundary(\cuco X,\mathfrak F)$ is homeomorphic to $\boundary(\cuco X,\mathfrak F')$. }

This theorem highlights the relationship between the question of when factor systems exist, and when $\cuco X$ is \emph{visible} in the sense that every simplex of the simplicial boundary corresponds to a geodesic ray in $\cuco X$; this is discussed in Remark~\ref{rem:FS_vis}.

\subsubsection*{Detecting splittings and cubulations from the boundary}\label{subsubsec:detecting_splittings_intro}
It is not difficult to show, from the definitions and Stallings' theorem on ends of groups~\cite{Stallings:ends}, that if $(G,\mathfrak S)$ is a hierarchically hyperbolic group, then $\boundary(G,\mathfrak S)$ is disconnected if and only if $G$ splits over a finite subgroup.

\begin{question}\label{question:JSJ}
Can the JSJ splitting of $G$ over slender subgroups (see~\cite{FujiwaraPapasoglu,DunwoodySageev,RipsSela}) be detected by examining separating spheres in $\boundary(G,\mathfrak S)$, as is the case for hyperbolic groups and splittings over two-ended subgroups~\cite{Bowditch:jsj}?
\end{question}

One can also consider producing actions of hierarchically hyperbolic groups on CAT(0) cube complexes other than trees.  As usual, this divides into two separate issues, namely detecting a profusion of codimension--$1$ subgroups and then choosing a finite collection sufficient to produce an action on a cube complex with good finiteness properties.  It appears as though $\boundary(G,\mathfrak S)$ can be used to produce a proper action on a cube complex from a sufficiently rich collection of hierarchically quasiconvex codimension--$1$ subgroups by a method exactly analogous to that used to cubulate various hyperbolic groups in~\cite{BergeronWise}.  The main difference is that $G$ does not act as a uniform convergence group on $\boundary(G,\mathfrak S)$; one must replace the space of triples of distinct boundary points by the space of triples $(p,q,r)\in\boundary G$ such that any two of $p,q,r$ are \emph{antipodal}, i.e. joined by a bi-infinite hierarchy path.  

\begin{question}\label{question:codim_1}
Let $(G,\mathfrak S)$ be a hierarchically hyperbolic group.  Give conditions on $G$ ensuring that for any antipodal $p,q\in\boundary G$, there exists a hierarchically quasiconvex codimension--$1$ subgroup $H$ so that $p,q$ are in distinct components of $\boundary gH$ for some $g\in G$.
\end{question}

We have not included a detailed discussion of the above ``boundary cubulation for HHG'' technique in the present paper since there are not yet any applications; these could be provided by an answer to Question~\ref{question:codim_1}.

\subsubsection*{Poisson boundaries and $C^*$--simplicity}\label{subsubsec:c_star_intro}
In Section~\ref{subsubsec:poisson}, we show that the boundary of an HHG is a topological model for the Poisson boundary:

\Restate{Theorem}{thm:poisson}{(Poisson boundary)}{Let $(G, \mathfrak S)$ be an HHG with $\diam \fontact S  = \infty$, $\mu$ be a nonelementary probability measure on $G$ with finite entropy and finite first logarithmic moment, and $\nu$ the resulting $\mu$-stationary measure on $\partial G$.  Then $(\partial G, \nu)$ is the Poisson boundary for $(G, \mu)$.}

In fact, $\boundary \fontact S$ is a model for the Poisson boundary~\cite{BehrstockHagenSisto:HHS_I}, but $\boundary(G,\mathfrak S)$ has the advantage of being compact, while in general $\boundary\fontact S$ is not compact.  The space $\boundary G$ is a \emph{$G$--boundary}, i.e. a compactum on which $G$ acts minimally and proximally.  Moreover:

\begin{propi}
The action of $G$ on $\boundary G$ is topologically free, i.e. for each $g\in G-\{1\}$, the set of $p\in\boundary\cuco X$ with $gp\neq p$ is dense in $\boundary\cuco X$.
\end{propi}

\begin{proof}
Let $g\in G-\{1\}$, let $q\in\boundary G$, and let $U$ be a neighborhood of $q$.  Suppose for a contradiction that $g$ fixes $U$ pointwise.  By Proposition~\ref{prop:alternative_HHG_version}, $G$ contains an irreducible axial element, so by Proposition~\ref{prop:stable dense}, $\boundary\fontact S$ is dense in $\boundary G$, whence, since $G$ is non-elementary, $g$ fixes infinitely many distinct points of $\boundary\fontact S$.  If $g$ is reducible axial, then Lemma~\ref{lem:bounded big papa} yields a contradiction, since $g$ cannot fix any point in $\boundary\fontact S$ by the lemma.  If $g$ is irreducible axial, then $g$ fixes exactly two points in $\boundary\fontact S$, again a contradiction.  Otherwise, $g$ is elliptic and hence has finite order and we are done by hypothesis. 
\end{proof}

By a result of Kalantar-Kennedy~\cite[Theorem~1.5]{KalantarKennedy}, the above proposition gives a new proof that a nonelementary HHG $G$ with $\boundary\fontact S$ unbounded is $C^*$--simple (i.e. the reduced $C^*$--algebra of $G$ is simple) provided finite-order elements have finite fixed point set in $\boundary\fontact S$.  However, $G$ is known to be $C^*$--simple under these circumstances, since $G$ is acylindrically hyperbolic~\cite{BehrstockHagenSisto:HHS_I} and has no finite normal subgroup~\cite{DGO}.

In light of the HHG structure on cubulated groups discussed above, Theorem~\ref{thm:poisson} should be compared to the results of~\cite{NevoSageev:Poisson}, in which Nevo-Sageev construct the Poisson boundary for a cubical group using the Roller boundary of the cube complex.

\subsection*{Outline of this paper}\label{subsec:structure}
In Section~\ref{sec:background}, we review hierarchically hyperbolic spaces.  In Section~\ref{sec:boundary_definition}, we define the HHS boundary.  Section~\ref{sec:compactness} is devoted to the proof that proper HHS have compact boundaries, and in Section~\ref{sec:gromov_boundary}, we show that the HHS boundary of a hyperbolic HHS is homeomorphic to the Gromov boundary.  In Section~\ref{sec:extending}, we discuss continuous extensions of maps between HHS to the boundary, and consider this phenomenon in the context of Veech and Leininger-Reid subgroups of the mapping class group.  Automorphisms of hierarchically hyperbolic structures induce homeomorphisms of the boundary; in Section~\ref{subsubsec:auts_and_HHG}, we classify automorphisms and study fixed sets and dynamics of the actions of automorphisms on the boundary.  In particular, in Section~\ref{sec:semisimple}, we show that cyclic subgroups of hierarchically hyperbolic groups are undistorted.  Section~\ref{sec:essential} is a brief technical discussion of essential HHS and actions, supporting Section~\ref{sec:rank_rigidity}, in which we prove the coarse rank rigidity theorem and some of its consequences.  In Section~\ref{sec:cube_complex}, we consider CAT(0) cube complexes with HHS structures coming from~\cite{BehrstockHagenSisto:HHS_I}, relating the HHS boundary to the \emph{simplicial boundary} from~\cite{Hagen:boundary}.

\subsection*{Acknowledgments}\label{subsec:acknowledgment}
M.F.H. is grateful to Dan Guralnik and Alessandra Iozzi for a discussion clarifying the situation now described in Remark~\ref{rem:defn_subtlety} and to Dave Futer and Brian Rushton for a discussion about some alternative topologies on the boundary of a cube complex related to the one discussed in Section~\ref{sec:cube_complex}.  M.G.D. is grateful to Henry Wilton and Cambridge University for their hospitality during a visit in the fall of 2015.  We are grateful to Jason Behrstock for several helpful discussions and for assistance with the HHS infographic mentioned above. We are grateful to Bob Daverman for a useful \url{mathoverflow.net} comment on the shrinkable decompositions of manifolds, which has implications for the structure of the HHS boundary of Teichm\"{u}ller space.  We thank Robert Tang for a useful suggestion related to Section~\ref{sec:extending}.  We also thank the organizers of the \textit{Young Geometric Group Theory IV} meeting, where initial discussions about this project took place.  We are also very grateful to Sarah Mousley, Jacob Russell, and the anonymous referee for numerous helpful comments and corrections.

\section{Background}\label{sec:background}
\subsection{Hierarchically hyperbolic spaces}\label{subsec:definition}
We begin by recalling the definition of a hierarchically hyperbolic space, introduced in~\cite{BehrstockHagenSisto:HHS_I} and axiomatized in a more efficient fashion in~\cite{BehrstockHagenSisto:HHS_II} as follows.  
We begin by defining a hierarchically hyperbolic space.  We 
will work in the context of a \emph{quasigeodesic space},  $\cuco X$, 
i.e., a 
metric space where any two
points can be connected by a uniform-quality quasigeodesic.

\begin{defn}[Hierarchically hyperbolic space]\label{defn:space_with_distance_formula}
The $q$--quasigeodesic space  $(\cuco X,\dist_{\cuco X})$ is a \emph{hierarchically hyperbolic space} if there exists $\delta\geq0$, an index set $\mathfrak S$, whose elements we call \emph{domains}, and a set $\{\fontact W:W\in\mathfrak S\}$ of $\delta$--hyperbolic spaces $(\fontact U,\dist_U)$,  such that the following conditions are satisfied:\begin{enumerate}
\item\textbf{(Projections.)}\label{item:dfs_curve_complexes} There is a set $\{\pi_W:\cuco X\rightarrow2^{\fontact W}\mid W\in\mathfrak S\}$ of \emph{projections} sending points in $\cuco X$ to sets of diameter bounded by some $\xi\geq0$ in the various $\fontact W\in\mathfrak S$.  Moreover, there exists $K$ so that each $\pi_W$ is $(K,K)$--coarsely Lipschitz.

 \item \textbf{(Nesting.)} \label{item:dfs_nesting} $\mathfrak S$ is
 equipped with a partial order $\nest$, and either $\mathfrak
 S=\emptyset$ or $\mathfrak S$ contains a unique $\nest$--maximal
 element; when $V\nest W$, we say $V$ is \emph{nested} in $W$.  We
 require that $W\nest W$ for all $W\in\mathfrak S$.  For each
 $W\in\mathfrak S$, we denote by $\mathfrak S_W$ the set of
 $V\in\mathfrak S$ such that $V\nest W$.  Moreover, for all $V,W\in\mathfrak S$
 with $V\propnest W$ there is a specified subset
 $\rho^V_W\subset\fontact W$ with $\diam_{\fontact W}(\rho^V_W)\leq\xi$.
 There is also a \emph{projection} $\rho^W_V\colon \fontact
 W\rightarrow 2^{\fontact V}$.  (The notation is
 justified by viewing $\rho^V_W$ as a coarsely constant map $\fontact
 V\rightarrow 2^{\fontact W}$.)
 
 \item \textbf{(Orthogonality.)} 
 \label{item:dfs_orthogonal} $\mathfrak S$ has a symmetric and anti-reflexive relation called \emph{orthogonality}: we write $V\orth W$ when $V,W$ are orthogonal.  Also, whenever $V\nest W$ and $W\orth U$, we require that $V\orth U$.  We require that for each $T\in\mathfrak S$ and each $U\in\mathfrak S_T$ for which $\{V\in\mathfrak S_T:V\orth U\}\neq\emptyset$, there exists $W\in \mathfrak S_T-\{T\}$, so that whenever $V\orth U$ and $V\nest T$, we have $V\nest W$.  Finally, if $V\orth W$, then $V,W$ are not $\nest$--comparable.
 
 \item \textbf{(Transversality and consistency.)}
 \label{item:dfs_transversal} If $V,W\in\mathfrak S$ are not
 orthogonal and neither is nested in the other, then we say $V,W$ are
 \emph{transverse}, denoted $V\transverse W$.  There exists
 $\kappa_0\geq 0$ such that if $V\transverse W$, then there are
  sets $\rho^V_W\subseteq\fontact W$ and
 $\rho^W_V\subseteq\fontact V$ each of diameter at most $\xi$ and 
 satisfying: $$\min\left\{\dist_{
 W}(\pi_W(x),\rho^V_W),\dist_{
 V}(\pi_V(x),\rho^W_V)\right\}\leq\kappa_0$$ for all $x\in\cuco X$.
 
 For $V,W\in\mathfrak S$ satisfying $V\nest W$ and for all
 $x\in\cuco X$: $$\min\left\{\dist_{
 W}(\pi_W(x),\rho^V_W),\diam_{\fontact
 V}(\pi_V(x)\cup\rho^W_V(\pi_W(x)))\right\}\leq\kappa_0.$$ 
 
 \noindent The preceding two inequalities are the \emph{consistency inequalities} for points in $\cuco X$.  Finally, if $U\nest V$, then $\dist_W(\rho^U_W,\rho^V_W)\leq\kappa_0$ whenever $W\in\mathfrak S$ satisfies either $V\propnest W$ or $V\transverse W$ and $W\north U$.
 
 \item \textbf{(Finite complexity.)} \label{item:dfs_complexity} There exists $n\geq0$, the \emph{complexity} of $\cuco X$ (with respect to $\mathfrak S$), so that any set of pairwise--$\nest$--comparable elements has cardinality at most $n$.
  
 \item \textbf{(Large links.)} \label{item:dfs_large_link_lemma} There
exist $\lambda\geq1$ and $E\geq\max\{\xi,\kappa_0\}$ such that the following holds.
Let $W\in\mathfrak S$ and let $x,x'\in\cuco X$.  Let
$N=\lambda\dist_{_W}(\pi_W(x),\pi_W(x'))+\lambda$.  Then there exists $\{T_i\}_{i=1,\dots,\lfloor
N\rfloor}\subseteq\mathfrak S_W-\{W\}$ such that for all $T\in\mathfrak
S_W-\{W\}$, either $T\in\mathfrak S_{T_i}$ for some $i$, or $\dist_{
T}(\pi_T(x),\pi_T(x'))<E$.  Also, $\dist_{
W}(\pi_W(x),\rho^{T_i}_W)\leq N$ for each $i$. 
 
 \item \textbf{(Bounded geodesic image.)} \label{item:dfs:bounded_geodesic_image} For all $W\in\mathfrak S$, all $V\in\mathfrak S_W-\{W\}$, and all geodesics $\gamma$ of $\fontact W$, either $\diam_{\fontact V}(\rho^W_V(\gamma))\leq E$ or $\gamma\cap\neb_E(\rho^V_W)\neq\emptyset$. 
 
 \item \textbf{(Partial Realization.)} \label{item:dfs_partial_realization} There exists a constant $\alpha$ with the following property. Let $\{V_j\}$ be a family of pairwise orthogonal elements of $\mathfrak S$, and let $p_j\in \pi_{V_j}(\cuco X)\subseteq \fontact V_j$. Then there exists $x\in \cuco X$ so that:
 \begin{itemize}
 \item $\dist_{V_j}(x,p_j)\leq \alpha$ for all $j$,
 \item for each $j$ and 
 each $V\in\mathfrak S$ with $V_j\nest V$, we have 
 $\dist_{V}(x,\rho^{V_j}_V)\leq\alpha$, and
 \item if $W\transverse V_j$ for some $j$, then $\dist_W(x,\rho^{V_j}_W)\leq\alpha$.
 \end{itemize}

\item\textbf{(Uniqueness.)} For each $\kappa\geq 0$, there exists
$\theta_u=\theta_u(\kappa)$ such that if $x,y\in\cuco X$ and
$\dist(x,y)\geq\theta_u$, then there exists $V\in\mathfrak S$ such
that $\dist_V(x,y)\geq \kappa$.\label{item:dfs_uniqueness}
\end{enumerate}
We often refer to $\mathfrak S$, together with the nesting
and orthogonality relations, the projections, and the hierarchy paths,
as a \emph{hierarchically hyperbolic structure} for the space $\cuco
X$.  
\end{defn}

\begin{notation}\label{notation:suppress_pi}
Given $U\in\mathfrak S$, we often suppress the projection
map $\pi_U$ when writing distances in $\fontact U$: given $x,y\in\cuco X$ and
$p\in\fontact U$  we write
$\dist_U(x,y)$ for $\dist_U(\pi_U(x),\pi_U(y))$ and $\dist_U(x,p)$ for
$\dist_U(\pi_U(x),p)$. To measure distance between a 
pair of sets, we take the infimal distance 
between the two sets. 
Given $A\subset \cuco X$ and $U\in\mathfrak S$ 
we let $\pi_{U}(A)$ denote $\cup_{a\in A}\pi_{U}(a)$.
\end{notation}

\begin{rem}[Summary of constants]\label{rem:constants}
Each hierarchically hyperbolic space $(\cuco X,\mathfrak S)$ is
associated with a collection of constants often, as above, denoted 
$\delta,\xi,n,\kappa_0,E,\theta_u,K$, where:
\begin{enumerate}
 \item $\fontact U$ is $\delta$--hyperbolic for each $U\in\mathfrak S$,
 \item each $\pi_U$ has image of diameter at most $\xi$ and each $\pi_U$ is $(K,K)$--coarsely Lipschitz, and each $\rho^U_V$ has (image of) diameter at most $\xi$,
 \item for each $x\in\cuco X$, the tuple $(\pi_U(x))_{U\in\mathfrak S}$ is $\kappa_0$--consistent,
 \item $E$ is the constant from the bounded geodesic image axiom.
 \end{enumerate}
Whenever working in a fixed hierarchically hyperbolic space, we use the above notation freely.  We can, and shall, assume that $E\geq q,E\geq\delta,E\geq\xi,E\geq\kappa_0,E\geq K$, and $E\geq\alpha$.
\end{rem}

\begin{lem}[``Finite dimension'']\label{lem:pairwise_orthogonal}
Let $(\cuco X,\mathfrak S)$ be a hierarchically hyperbolic space of complexity $n$ and let $U_1,\ldots,U_k\in\mathfrak S$ be pairwise-orthogonal.  Then $k\leq n$.
\end{lem}

\begin{proof}
Definition~\ref{defn:space_with_distance_formula}.\eqref{item:dfs_orthogonal} provides $W_1\in\mathfrak S$, not $\nest$--maximal, so that $U_2,\ldots,U_k\nest W_1$.  Using Definition~\ref{defn:space_with_distance_formula} inductively yields a sequence $W_{k-1}\propnest W_{k-2}\propnest\ldots\propnest W_1\nest S$, with $S$ $\nest$--maximal, so that $U_{i-1},\ldots,U_k\nest W_i$ for $1\leq i\leq k-1$.  Hence $k\leq n$ by Definition~\ref{defn:space_with_distance_formula}.\eqref{item:dfs_complexity}.
\end{proof}

The next lemma is a simple consequence of the axioms and also appears in~\cite{BehrstockHagenSisto:HHS_III}:

\begin{lem}\label{lem:orthogonal_close}
Let $U,V,W\in\mathfrak S$ satisfy $U\orth V$, and $U,V\notorth W$, and $W\not\nest U,V$.  Then $\dist_W(\rho^U_W,\rho^V_W)\leq 2E$.
\end{lem}

\begin{proof}
Our assumptions imply that $U\propnest W$ or $U\transverse W$, and the same is true for $V$.  Applying partial realization yields a point $x\in\cuco X$ so that $\dist_T(x,\rho^U_T),\dist_T(x,\rho^V_T)\leq E$ whenever $T\not\nest U,V$ and $T\notorth U,V$.  The claim follows from the triangle inequality.
\end{proof}

\begin{defn}
For $D\geq 1$, a path $\gamma$ in $\cuco X$ is a \emph{$D$--hierarchy path} if
 \begin{enumerate}
  \item $\gamma$ is a $(D,D)$-quasi-geodesic,
  \item for each $W\in\mathfrak S$, $\pi_W\circ\gamma$ is an unparametrized $(D,D)$--quasi-geodesic.
An unbounded hierarchy path $[0,\infty)\to\cuco X$ is a \emph{hierarchy ray}.
\end{enumerate}
\end{defn}

The following theorems are proved in~\cite{BehrstockHagenSisto:HHS_II}:

\begin{thm}[Realization theorem]\label{thm:realization}
Let $(\cuco X,\mathfrak S)$ be hierarchically hyperbolic.  Then for each $\kappa$ there exists $\theta_e,\theta_u$ such that the following holds. Let $\tup
 b\in\prod_{W\in\mathfrak S}2^{\fontact W}$ have each coordinate 
 correspond to a subset of $\fontact W$ of diameter at most $\kappa$; for each $W$, let $b_W$ denote the $\fontact
 W$--coordinate of $\tup b$. Suppose that whenever $V\transverse W$ we have
 $$\min\left\{\dist_{ W}(b_W,\rho^V_W),\dist_{ V}(b_V,\rho^W_V)\right\}\leq\kappa$$
 and whenever $V\nest W$ we have
 $$\min\left\{\dist_{ W}(b_W,\rho^V_W),\diam_{\fontact V}(b_V\cup\rho^W_V(b_W))\right\}\leq\kappa.$$
 Then the set of all $x\in \cuco X$ so that $\dist_{ W}(b_W,\pi_W(x))\leq\theta_e$ for all $\fontact W\in\mathfrak S$ is non-empty and has diameter at most $\theta_u$.
\end{thm}

\begin{thm}[Existence of Hierarchy Paths]\label{thm:monotone_hierarchy_paths}
Let $(\cuco X,\mathfrak S)$ be hierarchically hyperbolic. Then there exists $D_0$ so that any $x,y\in\cuco X$ are joined by a $D_0$-hierarchy path. 
\end{thm}

\begin{thm}[Distance Formula]\label{thm:distance_formula}
 Let $(X,\mathfrak S)$ be hierarchically hyperbolic. Then there exists $s_0\geq\xi$ such that for all $s\geq s_0$ there exist
 constants $K,C$ such that for all $x,y\in\cuco X$,
 $$\dist_{\cuco X}(x,y)\asymp_{(K,C)}\sum_{W\in\mathfrak S}\ignore{\dist_{ W}(\pi_W(x),\pi_W(y))}{s}.$$
\end{thm}

\noindent The notation $\ignore{A}{B}$ denotes the quantity which is $A$ if $A\geq B$ and $0$ otherwise.

\subsection{Hieromorphisms, automorphisms, and hierarchically hyperbolic groups}\label{subsec:hier_aut_HHG}
Morphisms in the category of hierarchically hyperbolic spaces were defined in~\cite{BehrstockHagenSisto:HHS_II}, along with the related notion of a hierarchically hyperbolic group; we recall these definitions here.

\begin{defn}[Hieromorphism]\label{defn:hieromorphism}
Let $(\cuco X,\mathfrak S)$ and $(\cuco X',\mathfrak S')$ be hierarchically hyperbolic structures on the spaces $\cuco X,\cuco X'$ respectively. A \emph{hieromorphism} $(f,\pi(f),\{\rho(f,U)\colon U\rightarrow\pi(f)(U)\mid U\in\mathfrak S\})\colon(\cuco X,\mathfrak S)\rightarrow(\cuco X',\mathfrak S')$ consists of a map $f:\cuco X\rightarrow\cuco X'$, a map $\pi(f):\mathfrak S\rightarrow\mathfrak S'$ preserving nesting, transversality, and orthogonality, and a set $\{\rho(f,U)\colon U\rightarrow\pi(f)(U)\mid U\in\mathfrak S\}$ of quasiisometric embeddings with uniform constants such that the following two diagrams coarsely commute for all nonorthogonal $U,V\in\mathfrak S$:
\begin{center}
$
\begin{diagram}
\node{\cuco X}\arrow[3]{e,t}{f}\arrow{s,r}{\pi_U}\node{}\node{}\node{\cuco X'}\arrow{s,r}{\pi_{\pi(f)(U)}}\\
\node{\fontact U}\arrow[3]{e,t}{\rho(f,U)}\node{}\node{}\node{\fontact \pi(f)(U)}
\end{diagram}
$
\end{center}
and
\begin{center}
$
\begin{diagram}
\node{\fontact U}\arrow[3]{e,t}{\rho(f,U)}\arrow{s,r}{\rho^U_V}\node{}\node{}\node{\fontact\pi(f)(U)}\arrow{s,r}{\rho^{\pi(f)(U)}_{\pi(f)(V)}}\\
\node{\fontact V}\arrow[3]{e,t}{\rho(f,V)}\node{}\node{}\node{\fontact \pi(f)(V)}
\end{diagram}
$
\end{center}
where $\rho^U_V\colon\fontact U\to\fontact V$ is the map from Definition~\ref{defn:space_with_distance_formula}.
\end{defn}

\begin{defn}[Automorphism of an HHS, automorphism group]\label{defn:automorphism}
A hieromorphism $f: (\cuco X, \mathfrak S) \rightarrow (\cuco X, \mathfrak S)$ is an \emph{automorphism} if $\pi(f): \mathfrak S \rightarrow \mathfrak S$ is a bijection and $\rho(f,U): \fontact U \rightarrow \fontact \pi(f)(U)$ is an isometry for each $U \in \mathfrak S$.  When the context is clear, we will continue to use $f$ to denote $f$, $\pi(f)$, and $\rho(f,U)$.  

Observe that if $f,f'$ are automorphisms of $(\cuco X,\mathfrak S)$, then $f\circ f':\cuco X\to\cuco X$ is also an automorphism: compose the maps $\mathfrak S\to\mathfrak S$, and compose isometries of the hyperbolic spaces in the obvious way.  Declare automorphisms $f,f'$ \emph{equivalent} if $\pi(f)=\pi(f')$ and $\rho(f,U)=\rho(f',U)$ for all $U\in\mathfrak S$.  Note that $f,f':\cuco X\to\cuco X$ uniformly coarsely coincide in this case.  

Denote by $\Aut(\mathfrak S)$ the set of equivalence classes of automorphisms, so $\Aut(\mathfrak S)$ is a group with the obvious multiplication. If $[f]\in\Aut(\mathfrak S)$, then $[f]^{-1}$ is represented by the quasi-inverse of $f$ associated to $\pi(f)^{-1}$ and $\{\rho(f,U)^{-1}:U\in\mathfrak S\}$.

Observe that $\Aut(\mathfrak S)$ quasi-acts on $\cuco X$ by uniform quasi-isometries.  We will sometimes abuse language and refer to individual automorphisms as elements of $\Aut(\mathfrak S)$, and refer to the ``action'' of $\Aut(\mathfrak S)$ on $\cuco X$.  By an \emph{action} of a group $G$ on $(\cuco X,\mathfrak S)$, we mean a homomorphism $G\to\Aut(\mathfrak S)$.  ``Coarse'' properties of an action, like properness and coboundedness, make sense in this context. 
\end{defn}

\begin{defn}[Equivariant]\label{defn:equivariant_hieromorphism}
Let $f:(\cuco X,\mathfrak S)\to(\cuco X',\mathfrak S')$ be a hieromorphism, $G,G'\leq\Aut(\mathfrak S),\Aut(\mathfrak S')$, and $\phi:G\to G'$ a homomorphism.  Then $f$ is \emph{$\phi$--equivariant} if\\
\begin{center}
\begin{minipage}[t]{0.2\textwidth}
$
\begin{diagram}
\node{\mathfrak S}\arrow{e,t}{f}\arrow{s,l}{g}\node{\mathfrak S'}\arrow{s,r}{\phi(g)}\\
\node{\mathfrak S}\arrow{e,t}{f}\node{\mathfrak S'}
\end{diagram}
$
\end{minipage}
and\ \ \ \ \ \ \ 
\begin{minipage}[t]{0.2\textwidth}
$
\begin{diagram}
\node{\fontact U}\arrow{e,t}{f}\arrow{s,l}{g}\node{\fontact f(U)}\arrow{s,r}{\phi(g)}\\
\node{\fontact gU}\arrow{e,t}{f}\node{\fontact \phi(g)f(U)}
\end{diagram}
$
\end{minipage}
\end{center}

\noindent (coarsely) commute for all $g\in G$ and $U\in\mathfrak S$.  This implies that $\phi(g)f(x)\asymp f(gx)$ for all $x\in\cuco X$ and $g\in G$.  If $\phi$ is an isomorphism and $f$ is $\phi$--equivariant, then $f$ is \emph{$G$--equivariant}.
\end{defn}

\begin{defn}[Hierarchically hyperbolic group]\label{defn:HHG}
A finitely generated group $G$ is \emph{hierarchically hyperbolic} if there exists a hierarchically hyperbolic space $(\cuco X,\mathfrak S)$ such that $G\leq\Aut(\mathfrak S)$, the action on $\cuco X$ is proper and cobounded, and $G$ acts on $\mathfrak S$ with finitely many orbits.  In this case we can assume $\cuco X=G$ (with any fixed word-metric) and that the action $G\to\Aut(\mathfrak S)$ sends each $g\in G$ to an automorphism whose underlying map $G\to G$ is left multiplication by $g$. In this case, we say that $(G,\mathfrak S)$ is \emph{hierarchically hyperbolic}. 
\end{defn}

\subsection{Standard product regions}\label{sec:product_regions}
The notion of a standard product region in a hierarchically hyperbolic space, introduced in~\cite{BehrstockHagenSisto:HHS_II}, plays an important role in several places, so we recall the definition here. Let $(\cuco X,\mathfrak S)$ be a hierarchically hyperbolic space and let $U\in\mathfrak S$.  Let $\mathfrak S_U$ be the set of $V\in\mathfrak S$ with $V\nest U$ (in particular, $U\in\mathfrak S_U$ is the unique $\nest$--maximal element).  Let $\mathfrak S_U^\orth$ be the set of $V\in\mathfrak S$ such that $V\orth U$, together with some $\nest$--minimal $A\in\mathfrak S$ such that all such $V\nest A$.

Fix $\kappa\geq\kappa_0$ and let $F_U$ be the space of $\kappa$--consistent tuples in $\prod_{V\in\mathfrak S_U}2^{\fontact V}$ whose coordinates are diameter--$\leq\xi$ sets.  Similarly, let $E_U$ be the set of $\kappa$--consistent tuples in $\prod_{V\in\mathfrak S_U^\orth-\{A\}}2^{\fontact V}$ whose coordinates are diameter--$\leq\xi$ sets.  In fact, $(F_U,\mathfrak S_U)$ and $(E_U,\mathfrak S_U^\orth)$ are hierarchically hyperbolic spaces (the hyperbolic space associated to $A$ is $\image_A(E_U)$), and there are hieromorphisms (see~\cite{BehrstockHagenSisto:HHS_II} or Definition~\ref{defn:hieromorphism}), inducing quasiisometric embeddings, $F_U,E_U\to\cuco X$, extending to a coarsely-defined map $F_U\times E_U\to\cuco X$ whose image is hierarchically quasiconvex in the sense of~\cite{BehrstockHagenSisto:HHS_II} (or see below).  Specifically, each tuple $\tup b\in F_U$ is sent to the tuple that coincides with $\tup b$ on $\mathfrak S_U$, and has coordinate $\rho^U_V$ for all $V\in\mathfrak S-\{U\}$ such that $V\transverse U$ or $U\nest V$, and is fixed at some base element of $E_U$ on $\mathfrak S_U^\orth-\{A\}$.  The map $E_U\to\cuco X$ is defined analogously.  The spaces $F_U,E_U$ are the \emph{standard nesting factor} and the \emph{standard orthogonality factor}, respectively, associated to $U$.  The maps are the \emph{standard hieromorphisms} associated to $U$, and the image $P_U$ of $F_U\times E_U$ is a \emph{standard product region}.  Where it will not cause confusion, we sometimes denote by $E_U,F_U$ the images of the corresponding standard hieromorphisms.

\begin{rem}[Automorphisms of product regions]\label{rem:aut_product}
Let $(\cuco X,\mathfrak S)$ be a hierarchically hyperbolic space and let $U\in\mathfrak S$.  Recall that $(F_U,\mathfrak S_U)$ is a hierarchically hyperbolic space, where the hyperbolic spaces and projections implicit in the hierarchically hyperbolic structure are exactly those inherited from $\mathfrak S$.  Recall that $(E_U,\mathfrak S_U^\orth)$ is a hierarchically hyperbolic space, where $\fontact V$ is as in $(\cuco X,\mathfrak S)$ except when $V=A$ is the $\nest$--maximal element.  The hieromorphism $(E_U,\mathfrak S_U^\orth)\to(\cuco X,\mathfrak S)$ is determined by the choice of $A\in\mathfrak S$ that is $\nest$--minimal among all those containing each $V$ with $V\orth U$, which we take as the $\nest$--maximal element of $\mathfrak S_U^\orth$.

Let $\mathcal A_U$ be the group of automorphisms $g$ of $\mathfrak S$ such that $g\cdot U=U$.  Then there are \emph{restriction homomorphisms} $\theta_U,\theta_U^\orth:\mathcal A_U\to\Aut(\mathfrak S_U),\Aut(\mathfrak S_U^\orth)$ defined as follows.  Given $g\in\mathcal A_U$, let $\theta_U(g)$ act like $g$ on $\mathfrak S_U$ and like $g$ on each $\fontact V$ with $V\nest U$.  

Define $\theta^\orth$ analogously to give an automorphism of $\mathfrak S^\orth_U-\{A\}$ restricting the action of $g$ on $\mathfrak S$, and fixing $A$.  When defining $g:\image_A(E_U)\to\image_A(E_U)$, we draw attention to two cases, which it will be important to distinguish in Section~\ref{sec:rank_rigidity}:
\begin{itemize}
 \item There exist infinitely many $A_i\in\mathfrak S$ that are $\nest$--minimal with the property that $V\nest A_i$ whenever $V\orth U$.  The minimality assumption implies that these $A_i$ are pairwise non-nested, so, using Lemma~\ref{lem:pairwise_orthogonal} and the consistency axiom, we see that $\pi_{A_i}(E_U)$ has diameter bounded independently of $A_i$ (in fact, just in terms of $E$); thus, when building the HHS $(E_U,\mathfrak S_U^\orth)$, we can take the hyperbolic space $\image_A(E_U)$ associated to the maximal element $A$ to be a single point, and define $g:\image_A(E_U)\to\image_A(E_U)$ in the obvious way.  This conclusion holds, more generally, if there are two transverse $\nest$--minimal ``containers'' $A_i,A_j$ for the domains orthogonal to $U$.
 \item The set $\{A_i\}$ of domains that are $\nest$--minimal with the property that $V\nest A_i$ whenever $V\orth U$ is a pairwise-orthogonal set.  In this case, there are at most $n$ such $A_i$, where $n$ is the complexity, by Lemma~\ref{lem:pairwise_orthogonal}.  Again, we choose $A\in\{A_i\}$ arbitrarily and define the HHS structure on $(E_U,\mathfrak S_U^\orth)$ using $A$ as the $\nest$--maximal element, with associated hyperbolic space $\image_A(E_U)$.  Now, if there exists $h\in\Aut(\mathfrak S)$ so that $hA=A_i$ for some $i$, then $\image_{A_i}(E_U)$ is uniformly quasi-isometric to $\image_A(E_U)$.  In particular, $g:\image_A(E_U)\to\image_A(E_U)$ can be defined so that the restriction homomorphism $\theta_U^\orth$ makes sense.
\end{itemize}

Note that, if $f\in\mathcal A_U$ and $x\in P_U\subset\cuco X$, then $\dist_{F_U\times E_U}(\theta_U(f)(r_U(x)),r_U(f(x)))$ is uniformly bounded, where $r_U\colon P_U\cong_{q.i.}F_U\times E_U\to F_U$ is coarse projection to the first factor, and a similar statement holds for $\theta^\orth_U$ and projection to $E_U$.

Finally, recall that the standard product region $P_U$ is defined to be the image of $F_U\times E_U$ under the product of the hieromorphisms $(F_U,\mathfrak S_U),(E_U,\mathfrak S_U^\orth)\to(\cuco X,\mathfrak S)$.  This map is coarsely defined, but it is convenient to fix maps $F_U\times E_U\to\cuco X$ (realizing those hieromorphisms) so that $P_{gU}=gP_U$ for all $U\in\mathfrak S$ and $g\in\Aut(\mathfrak S)$.  Similarly, the image of $F_{gU}$ coincides with $gF_U$, etc.  The set $\{P_U:U\in\mathfrak S\}$ is  $\Aut(\mathfrak S)$--invariant.
\end{rem}

\subsection{Normalized hierarchically hyperbolic spaces and hierarchical quasiconvexity}\label{subsec:normalized}
Hierarchically hyperbolic spaces, in the sense of Definition~\ref{defn:space_with_distance_formula}, need not coarsely surject to the associated hyperbolic spaces, but in almost all cases of interest, they do.  Accordingly:

\begin{defn}[Normalized HHS]\label{defn:normalized_HHS}
The HHS $(\cuco X,\mathfrak S)$ is \emph{normalized} if there exists $C$ such that for all $U\in\mathfrak S$, we have $\fontact U=\neb_{\fontact U}(\pi_U(\cuco X))$.
\end{defn}

\begin{prop}\label{prop:normalizing}
Let $(\cuco X,\mathfrak S)$ be a hierarchically hyperbolic space.  Then $\cuco X$ admits a normalized hierarchically hyperbolic structure $(\cuco X,\mathfrak S')$ with a hieromorphism $f\colon(\cuco X,\mathfrak S')\to(\cuco X,\mathfrak S)$, where $f\colon\cuco X\to\cuco X$ is the identity and $f\colon\mathfrak S'\to\mathfrak S$ is a bijection.  Moreover, if $G\leq\Aut(\mathfrak S)$, then there is a monomorphism $G\to\Aut(\mathfrak S')$ making $f$ equivariant.
\end{prop}

\begin{proof}
Let $\mathfrak S'=\mathfrak S$, and retain the same nesting, orthogonality, and transversality relations.  For each $U\in\mathfrak S'$, the associated hyperbolic space $\fontact_{norm}U$ is chosen to be uniformly quasiisometric to the uniformly quasiconvex subset $\pi_U(\cuco X)$ of $\fontact U$.  The projection $\pi_U:\cuco X\to\fontact_{norm}U$ is, up to composition with a uniform quasiisometry, unchanged (and therefore continues to be coarsely Lipschitz).  Let $p_U:\fontact U\to\fontact_{norm}U$ be the composition of the coarse closest-point projection $\fontact U\to\pi_U(\cuco X)$, composed with the uniform quasiisometry $\pi_U(\cuco X)\to\fontact_{norm}U$.  Then for all $U,V$ with $U\transverse V$ or $U\nest V$, define the relative projection $\fontact_{norm}U\to\fontact_{norm}V$ to be the composition of $p_U\circ\rho^U_V:\pi_U(\cuco X)\to\fontact_{norm} V$ with the quasiisometry $\fontact_{norm}U\to\pi_U(\cuco X)$.  The remaining assertions are a matter of checking definitions.
\end{proof}

Recall from~\cite{BehrstockHagenSisto:HHS_II} that the subspace $\cuco Y$ of $(\cuco X,\mathfrak S)$ is \emph{hierarchically quasiconvex} if there exists $k_0\geq0$ such that $\pi_U(\cuco Y)$ is $k_0$--quasiconvex in $\fontact U$ for all $U\in\mathfrak S$ and, if for all $\kappa\geq\kappa_0$, each $\kappa$--consistent tuple $\tup b\in\prod_{U\in\mathfrak S}\fontact U$ with $U$--coordinate in $\pi_U(\cuco Y)$ for all $U$ has the property that any associated realization point $x\in\cuco X$ lies at distance from $\cuco Y$ depending only on $\kappa$.  

In the interest of staying in the class of normalized hierarchically hyperbolic spaces, we will always work with a normalized hierarchically hyperbolic structure on $\cuco Y$, namely the one provided by Proposition~\ref{prop:normalizing}.  Moreover, we will  (abusively) eschew the notation $\fontact_{norm}U$ and use the same notation for $\pi_U(\cuco Y)$ and its thickening; in other words, we will regard $\pi_U(\cuco Y)$ as a genuine (uniformly) hyperbolic geodesic space.

Finally, we recall the following notion from~\cite[Definition 5.3, Lemma 5.4]{BehrstockHagenSisto:HHS_II}.  Let $\cuco Y\subset\cuco X$ be a hierarchically quasiconvex subspace.  Then there is a coarsely Lipschitz map $\gate_{\cuco Y}:\cuco X\to\cuco Y$ (the coarse Lipschitz constants depend only on the constants from Definition~\ref{defn:space_with_distance_formula} and the constants implicit in the definition of hierarchical quasiconvexity) with the following property: for each $U\in\mathfrak S$ and $x\in\cuco X$, the projection $\pi_U(\gate_{\cuco Y}(x))$ uniformly coarsely coincides with the coarse closest-point projection of $\pi_U(x)$ to the quasiconvex subspace $\pi_U(\cuco Y)$.  The map $\gate_{\cuco Y}$ is the \emph{gate map} associated to $\cuco Y$.

\section{Definition of the boundary }\label{sec:boundary_definition}
Fix a hierarchically hyperbolic space $(\cuco X,\mathfrak S)$.  For each $S\in\mathfrak S$, denote by $\boundary\fontact S$ the Gromov boundary, i.e. the space of equivalence classes of sequences $(x_n\in\fontact S)$, where $(x_n)$ and $(y_n)$ are equivalent if for some (hence any) fixed basepoint $x\in\fontact S$, we have $(x_n,y_n)_x\to\infty$.  In particular, $\boundary\fontact S$ need not be compact if $\fontact S$ is not proper. The topology is as usual.

\begin{rem}[Extending the Gromov product]\label{rem:gromov_product}
For $U\in\mathfrak S$, any $p,q\in\fontact U\cup\boundary\fontact U$ are joined to
$u\in\fontact U$ by $(1,20\delta)$--quasigeodesics, enabling extension of the Gromov product to
$\boundary\fontact U$.
\end{rem}

\subsection{Supports and boundary points}\label{subsec:supports_boundary_points}
We first define $\partial \cuco X=\partial (\cuco X,\mathfrak S)$ as a set.  

\begin{defn}[Support set, boundary point]\label{defn:boundary_point}
A \emph{support set} $\overline S\subset\mathfrak S$ is a set with $S_i\orth S_j$ for all $S_i,S_j\in\overline S$.  Given a support set $\overline S$, a \emph{boundary point} with \emph{support} $\overline S$ is a formal sum $p=\sum_{S\in\overline S}a^p_Sp_S$, where each $p_S\in\boundary\fontact S$, and $a^p_S>0$, and $\sum_{S\in\overline S}a^p_S=1$.  Such sums are necessarily finite, by Lemma~\ref{lem:pairwise_orthogonal}.  We denote the support $\overline S$ of $p$ by $\support(p)$.
\end{defn}

\begin{defn}[Boundary]\label{defn:hhs_boundary}
The \emph{boundary} $\boundary(\cuco X,\mathfrak S)$ of $(\cuco X,\mathfrak S)$ is the set of boundary points.
\end{defn}

\begin{notation}\label{notation:boundary}
When the specific HHS structure is clear, we write $\boundary\cuco X$ to mean $\boundary(\cuco X,\mathfrak S)$.
\end{notation}

\subsection{Topologizing $\boundary\cuco X$}\label{subsec:topology_of_boundary}
We topologize $\boundary\cuco X$ using the visual topologies on the Gromov boundaries of elements of $\{\fontact S:S\in\mathfrak S\}$.  The main challenge is to incorporate these topologies into a coherent topology on the whole boundary, allowing boundary points supported on nonorthogonal domains to interact.  This requires some preliminary definitions.

\begin{defn}[Remote point]\label{defn:remote_point}
Let $\overline S\subset\mathfrak S$ be a support set.  A point $p\in\boundary\cuco X$ is \emph{remote (with respect to $\overline S$, or with respect to some $q\in\boundary\cuco X$ with support $\overline S$)} if:
\begin{enumerate}
 \item $\support(p)\cap\overline S=\emptyset$, and
 \item for all $S\in\overline S$, there exists $T\in\support(p)$ so that $S$ and $T$ are \textbf{not} orthogonal.
\end{enumerate}
Denote by $\remote{\overline S}{\cuco X}$ the set of all remote points with respect to $\overline S$.
\end{defn}

For each $S\in\mathfrak S$, let $\mathcal B(\fontact S)$ be the set of all bounded sets in $\fontact S$.  If $\overline S\subset\mathfrak S$ is a support set, we denote by $\overline S^\orth$ the set of all $U\in\mathfrak S$ such that $U\orth S$ for all $S\in\overline S$.  

\begin{defn}[Boundary projection]\label{defn:boundary_projection}
Let $\overline S\subset\mathfrak S$ be a support set.  For each $q\in\remote{\overline S}{\cuco X}$, let $\overline S_q$ be the union of $\overline S$ and the set of domains $T\in\overline S^\orth$ such that $T$ is not orthogonal to $W_T$ for some $W_T\in\support(q)$.  Define a \emph{boundary projection} $\partial\pi_{\overline{S}}(q)\in\prod_{S\in\overline S_q}\fontact S$ as follows.  Let $q=\sum_{T\in\overline T}a^p_Tq_T$ be a remote point with respect to $\overline S$.  For each $S\in\overline S_q$, let $T_S\in\support(q)$ be chosen so that $S$ and $T_S$ are not orthogonal.  Define the $S$--coordinate $\left(\partial\pi_{\overline{S}}(q)\right)_S$ of $\partial\pi_{\overline S}(q)$ as follows:
\begin{enumerate}
 \item If $T_S\nest S$ or $T_S\transverse S$, then $\left(\partial\pi_{\overline{S}}(q)\right)_S=\rho^{T_S}_{S}$;
 \item otherwise, $S\nest T_S$.  Choose a $(1,20\delta)$--quasigeodesic ray $\gamma$ in $\fontact T_S$ joining $\rho^S_{T_S}$ to $q_{T_S}$.  By the bounded geodesic image axiom, there exists $x\in\gamma$ such that $\rho^{T_S}_S$ is coarsely constant on the subray of $\gamma$ beginning at $x$.  Let $\left(\partial\pi_{\overline{S}}(q)\right)_S=\rho^{T_S}_{S}(x)$.
\end{enumerate}

\end{defn}

\begin{lem}\label{lem:well_defined}
The map $\boundary\pi_{\overline S}$ is coarsely independent of the choice of $\{T_S\}_{S\in\overline S}$.
\end{lem}

\begin{proof}
Suppose that $T_S,T'_S\in\overline T$ are chosen so that $T_S,T'_S$ are not orthogonal to $S$ and suppose that $S\not\nest T_S,T_S'$.  In other words, either $T_S\nest S$ or $T_S\transverse S$ and the same is true for $T'_S$.  By partial realization (Definition~\ref{defn:space_with_distance_formula}.\eqref{item:dfs_partial_realization}), there therefore exists $y\in\cuco X$ so that $\dist_S(\rho^{T_S}_S,y),\dist_S(\rho^{T'_S}_S,y)\leq E$, whence $\rho^{T_S}_S$ and $\rho^{T'_S}_S$ coarsely coincide.  If $S\nest T_S$, then $S\orth T'_S$ since $T_S\orth T_S'$; this contradicts the defining property of $T'_S$.  Hence, in all allowable situations, $\rho^{T_S}_S$ coarsely coincides with $\rho^{T'_S}_S$; the claim follows. 
\end{proof}

Fix a basepoint $x_0\in \cuco X$.  We are now ready to define a neighborhood basis for each $p=\sum_{S\in\overline S}a^p_Sp_S$, where $p_S\in\fontact S$ for all $S\in\support(p)=\overline S$.  For each $S\in\mathfrak S$, choose a cone-topology neighborhood $U_S$ of $p_S$ in $\fontact S\cup\boundary\fontact S$, and choose $\epsilon>0$.  For convenience, given $q\in\boundary\cuco X$, we let $a^q_T=0$ when $T\in\mathfrak S-\support(q)$.

We define the basic set $\neb_{\{U_S\},\epsilon}(p)$ as the union of a \emph{remote part}, a \emph{non-remote part}, and an \emph{interior part}, as follows:

\begin{defn}[Remote part]\label{defn:remote_part}
The \emph{remote part} is: 

$$\neb_{\{U_S\},\epsilon}^{rem}(p)=$$

$$\left\{q\in\remote{\overline S}{\cuco X}\Big |\forall S\in\overline S, (\partial\pi_{\overline{S}}(q))_S\in U_{S},\ \mathrm{and}\ \forall S\in\overline S_q, S'\in\overline S,\,\left|\frac{\dist_{S}(x_0,(\partial\pi_{\overline{S}}(q))_S)}{\dist_{S'}(x_0,(\partial\pi_{\overline{S}}(q))_{S'})}-\frac{a_S^p}{
a_{S'}^p}\right|<\epsilon\ \mathrm{and}\ \sum_{T\in \overline{S}^{\orth}} a^q_T<\epsilon\right\}.$$

\end{defn}

\begin{defn}[Non-remote part]\label{defn:non_remote_part}
Given $p,q\in\boundary\cuco X$, let $A=\support(p)\cap\support(q)$.  The \emph{non-remote part} is:
$$\neb^{non}_{\{U_{S}\},\epsilon}(p)=\left\{q=\sum_{T}a^q_Tq_T\in\partial\cuco X-\remote{\overline{S}}{\cuco X}\Big|\sum_{V\in\support(q)-A}a^q_V<\epsilon,\,\forall T\in A: |a^q_T-a^p_T|<\epsilon, q_T\in U_T\right\}.$$
\end{defn}

\begin{defn}[Interior part]\label{defn:interior_part}
The \emph{interior part} is:

$$\neb^{int}_{\{U_S\},\epsilon}(p)=\left\{x\in\cuco X\Big|\ \forall S,S'\in\overline S, \forall T\in \overline{S}^{\orth}:\pi_{S}(x)\in U_S, \left|\frac{a_S}{a_{S'}} - \frac{\dist_{S}(x_0,x)}{\dist_{S'}(x_0,x)}\right|< \epsilon, \frac{\dist_{T}(x_0,x)}{\dist_{S}(x_0,x)}< \epsilon\right\}.$$
\end{defn}

\begin{defn}[Topology on $\cuco X\cup\boundary\cuco X$]\label{defn:boundary_topology}
For each $p\in\boundary\cuco X$, with $\support(p)=\overline S$, and $\{U_S:S\in\overline S\},\epsilon>0$ as above, let: $$\neb_{\{U_S\},\epsilon}(p)=\neb_{\{U_S\},\epsilon}^{rem}(p)\cup \neb^{non}_{\{U_{S}\},\epsilon}(p)\cup\neb^{int}_{\{U_S\},\epsilon}(p).$$  We declare the set of all such $\neb_{\{U_S\},\epsilon}(p)$ to form a neighborhood basis at $p$.  Also, we include in the topology on $\cuco X\cup \partial \cuco X$ the open sets in $\cuco X$. This topology does not depend on $x_0$.  
\end{defn}

\begin{rem}\label{rem:not_open}
The $\neb_{\{U_S\},\epsilon}(p)$ need not be open; a priori, they may have empty interior!
\end{rem}

The following is an obvious consequence of the definitions:

\begin{prop} \label{prop:babies continuously embed}
For all $U \in \mathfrak S$, the inclusion $\partial \fontact U \hookrightarrow \partial \cuco X$ is an embedding.
\end{prop}

Proposition~\ref{prop:properties} gives basic properties of $\boundary\cuco X$; first we need a definition and some lemmas.

\begin{defn}[Basically Hausdorff]\label{defn:basically_hausdorff}
Let $\cuco H$ be a topological space and let $\mathcal B$ be a neighborhood basis.  Then $(\cuco H,\mathcal B)$ is \emph{basically Hausdorff} if for all distinct $h,h'\in\cuco H$, there exist disjoint $B,B'\in\mathcal B$ with $h\in B,h'\in B'$.
\end{defn}

\begin{lem}\label{lem:boundary_basically_Hausdorff}
Let $(\cuco X,\mathfrak S)$ be hierarchically hyperbolic and let $\overline{\cuco{X}}=\cuco X\cup \partial (\cuco X,\mathfrak S)$.  Then, equipped with the neighborhood basis declared above, $\overline{\cuco X}$ is basically Hausdorff.
\end{lem}

\begin{proof}

Let $p,q\in\overline{\cuco X}$ be distinct.  The statement is obvious when $p$ or $q$ is in $\cuco X$, so assume that $p,q\in\boundary\cuco X$.  Fix a basepoint $x_0\in\cuco X$.

For each $U\in\support(p)$, choose a neighborhood $Y^p_U$ of $p$ in $\fontact U\cup\boundary\fontact U$ that does not contain $(\boundary\pi_{\support(p)}(q))_U$, provided it is defined.  For each $T\in\support(q)$, choose a neighborhood $Y^q_T$ of $q$ in $\fontact T\cup\boundary\fontact T$ that does not intersect $\neb_{1000E+\omega}(\{\pi_T(x_0)\})$ and, when it is defined, $\neb_{1000E+\omega}((\boundary\pi_{\support(q)}(p))_T)$, where $\omega\geq0$ is to be determined; also choose $Y_T^q$ so that $Y_T^p\cap Y_T^q=\emptyset$ when $T\in\support(p)\cap\support(q)$, unless $p_T=q_T$, in which case we choose $Y_T^p=Y_T^q$.  Fix $\epsilon>0$, to be determined.  Let $\neb(p)=\neb_{\{Y^p_U\},\epsilon}(p)$ and $\neb(q)=\neb_{\{Y^q_V\},\epsilon}(q)$.

Finally, for any $w,v \in \partial \cuco X$, let $\support(w)_v = \support(w) \cup (\support(w)^{\orth} - \support(v)^{\orth})$.

We need an auxiliary claim:

\begin{claim}\label{claim:2}
Let $x,p,q\in\boundary\cuco X$.  Suppose there exist $W_p,W_q\in\support(x)$ and $U\in\support(p)_x,V\in\support(q)_x$ so that $W_p\notorth U$ and $W_p\neq U$, and $W_q\notorth V$ and $W_q\neq V$.  Then there exists $y\in P_{W_p}\cap P_{W_q}\subset\cuco X$ such that $(\boundary\pi_{\support(p)}(x))_U$ $100E$--coarsely coincides with $\pi_U(y)$, and $(\boundary\pi_{\support(q)}(x))_V$ $100E$--coarsely coincides with $\pi_V(y)$.
\end{claim}

($P_{W_p}$ is the standard product region associated to $W_p$, defined in Section~1.3.)

\renewcommand{\qedsymbol}{$\blacksquare$}
\begin{proof}[Proof of Claim~\ref{claim:2}]
If $W_p\transverse U$ or $W_p\nest U$, and $W_q\transverse V$ or $W_q\nest V$, then any $y\in P_{W_p}\cap P_{W_q}$ suffices.  If $U\propnest W_p$, use partial realization to see that, given a $(1,20\delta)$--quasigeodesic ray $\gamma$ in $\contact W_p$ with endpoint $x_{W_p}$, we can choose a sequence $(y_n)$ in $P_{W_p}\cap P_{W_q}$ projecting uniformly close to an unbounded sequence in $\gamma$.  This provides the desired $y$.
\end{proof}
\renewcommand{\qedsymbol}{$\Box$}

Suppose that $x\in\neb(p)\cap\neb(q)$.  We consider the following cases:
\begin{enumerate}
 
 \item $x\in\boundary\cuco X$ is $p$--remote and $q$--remote. First of all, notice that by definition of remote, for any $U\in\supp(p)$ there exists $W_p$ as in Claim \ref{claim:2}, and similarly for $V\in\support(q)$. We now consider the following subcases.
 \begin{enumerate}[(a)]
 \item There exists $U\in \supp(p)\cap \support(q)$ with $p_U\neq q_U$. Then we would have that $(\partial \pi_{\supp(p)}(x))_U$ is contained in both $Y^p_U$ and $Y^q_U$, which are disjoint, a contradiction.\label{item:a}
 
 
 \item There exists $U\in \supp(p)\cap \support(q)$ with $p_U = q_U$ but $a^p_U\neq a^q_U$. Let $\mathcal U=\supp(p)\cap \support(q)$. For each $V\in\mathcal U$ we have that the ratio $\frac{d_V(x_0,(\partial \pi_{\supp(p)}(x))_V)}{d_U(x_0,(\partial \pi_{\supp(p)}(x))_U)}$ is $\epsilon$--close to both $a^p_V/a^p_U$ and $a^q_V/a^q_U$. Hence, if there exists $V\in\mathcal U$ so that $a^p_V/a^p_U\neq a^q_V/a^q_U$, we can choose $\epsilon$ small enough to give a contradiction. Otherwise, since the coefficients sum to $1$, the supports of $p$ and $q$ do not coincide, and we deal with this in the next subcases.\label{item:b}
 
 \item Up to swapping $p$ and $q$, there exists $V\in\supp(q)-\supp(p)$, and there exists $U\in\support(p)$ not orthogonal to $V$.  If $U\transverse V$, then by our choice of $\neb(p),\neb(q)$, we have $\dist_U(y,\rho^V_U)>E,\dist_V(y,\rho^U_V)>E$ for $y$ as in Claim \ref{claim:2}, contradicting consistency.  If $U\propnest V$ or $V\propnest U$, then we reach a similar contradiction of consistency.
 
 \item Now assume that the previous case does not apply and, up to swapping $p$ and $q$, there exists $V \in \left(\supp(q)-\supp(p)\right) \cap \supp(p)^\perp$.  Suppose we also have $\support(p)\subseteq\support(q)\cup\support(q)^\orth$ but $\support(p)\cap\support(q)^\orth\neq\emptyset$, since otherwise either~\eqref{item:a} or~\eqref{item:b} holds.  Let $U\in\support(p)-\support(q)$.  By remoteness of $x$, $U\in\support(q)^\orth-\support(x)^\orth$, so $U\in\support(q)_x$.  Hence the definition of $q$--remoteness gives
 
 $$\left|\frac{\dist_U(x_0,(\boundary\pi_{\support(q)}(x))_U)}{\dist_V(x_0,(\boundary\pi_{\support(q)}(x))_V)}-\frac{a^q_U}{a^q_V}\right|<\epsilon.$$
 
 Similarly, we have $V\in\support(p)_x$, so the definition of $p$--remoteness gives:
 
 $$\left|\frac{d_V(x_0,(\partial \pi_{\supp(p)}(x))_V)}{d_U(x_0,(\partial \pi_{\supp(p)}(x))_U)}-\frac{a^p_V}{a^p_U}\right|<\epsilon.$$
 
 Now, since $V\not\in\support(p),U\not\in\support(q)$, we have $a^p_V=a^q_U=0$, so, we may take $y$ to be the point in $\cuco X$ provided by Claim~\ref{claim:2}, and hence we have $\frac{\dist_{V}(y,x_0)}{\dist_U(y,x_0)}<2\epsilon$ and $\frac{\dist_{U}(y,x_0)}{\dist_V(y,x_0)}<2\epsilon$, provided $\omega$ in Claim \ref{claim:2} was chosen sufficiently large in terms of $\epsilon$ and $E$.  This is a contradiction.

 \end{enumerate}

 
 \item $x\in\cuco X$: In this case, $x$ can play the role of $y$ in the above arguments.

 \item $x\in\boundary\cuco X$ is $p$--non-remote and $q$--non-remote: In this case, first choose $\epsilon\in(0,1/2)$ smaller than $|a^p_W-a^q_W|/10$ for each $W\in\support(p)\cap\support(q)$.  The definition of the non-remote part now ensures that $x$ cannot exist.
 
 \item $x\in\boundary\cuco X$ is $p$--remote and $q$--non-remote: In this case, there exists $U\in\support(p),V\in\support(q)$ and $W_p,W_q\in\support(x)$ so that $W_p$ is distinct from and non-orthogonal to $U$ while $W_p=V$ or $W_p\orth V$.  If for each such $W_q$ we have $W_q\in\support(q)^\orth$, then by choosing $\epsilon<1$, we have that $\sum_{T\in\support(x)}a^x_T<1$, a contradiction.  Thus we may take $W_q=V\in\support(q)$.  
 
 Now, choose $y\in P_{W_p}$ so that $(\boundary\pi_{\support(p)}(x))_U$ $100E$--coarsely coincides with $\pi_U(y)$.  If $U=W_q$, then our choice of $\neb(p),\neb(q)$ ensures that $x$ cannot lie in both.  Suppose that $U\transverse W_q$.  Then $\pi_U(y),\rho^{W_q}_U,\rho^{W_p}_U$ all $10E$--coarsely coincide and lie at distance $50E$ from the required neighborhood of $p_U$, so $x\not\in\neb(p)$.  When $U\propnest W_q$ or $W_q\propnest U$, a similar argument shows that $x\not\in\neb(p)\cap\neb(q)$.
 
 Hence it remains to consider the case where $W_q\orth U$.  By definition, $|a^x_{W_q}-a^q_{W_q}|<\epsilon$.  On the other hand, we can assume $W_q\in\supp(p)^\orth$, for otherwise we could re-choose $U$ and $W_q$ to be in one of the above cases.  Thus, by definition, $a^x_{W_q}<\epsilon$.  This yields a contradiction provided we choose, say, $\epsilon\in\bigcap_{T\in\support(q)}\left(0,\frac{a^q_{T}}{10}\right)$.
 
\end{enumerate}
Hence our choice of $\neb(p),\neb(q)$ ensures $\neb(p)\cap\neb(q)=\emptyset$, as required.
\end{proof}

\begin{lem}\label{lem:hausdorff}
$\overline{\cuco X}$ is Hausdorff.
\end{lem}

\begin{proof}
In light of Lemma~\ref{lem:boundary_basically_Hausdorff}, it suffices to show that for all $p\in\boundary\cuco X$, with $p=\sum_{T\in\support(p)}a_Tp_T$, all $\epsilon>0$, and all collections $\{U_T:T\in\support(p)\}$ with each $U_T$ a neighborhood of $p_T$ in $\fontact T\cup\boundary\fontact T$, the corresponding basic set $\neb_{\{U_T\},\epsilon}(p)$ has nonempty interior.

\textbf{The topology of basic convergence:}  Given a sequence $\{p_n\}$ with each $p_n\in\overline{\cuco X}$, we say that $p_n$ \emph{basically converges} to $p\in\boundary\cuco X$ if for all $\epsilon>0$ and all choices of $\{U_T\}$ as above, we have $p_n\in\neb_{\{U_T\},\epsilon}(p)$ for all but finitely many $n\in\naturals$.  Similarly, $\{p_n\}$ \emph{basically converges} to $p\in\cuco X$ if, for all $\epsilon>0$, we have $p_n\in\neb_{\epsilon}(p)$ for all sufficiently large $n$.

Define a topology on $\overline{\cuco X}$ as follows: the set $A\subset\overline{\cuco X}$ is declared to be closed if $a\in A$ whenever there is a sequence $\{a_n\}$ so that $a_n\in A$ for all $n$ and $a_n$ basically converges to $a$.  Denote by $\mma$ the space $\overline{\cuco X}$ endowed with this topology.

\textbf{Nonempty interior of basic sets:}  Let $\neb=\neb_{\{U_T\},\epsilon}(p)$ be a basic set as above.  We claim that $p\in\interior{\neb}$.  Otherwise, there exists a sequence $\{p_n\}$ in $\overline{\cuco X}-\neb$ that basically converges to $p$.  This is a contradiction since basic convergence to $p$ needs $\{p_n\}$ to enter $\neb$.  

\textbf{Equivalence of the topologies:}  To complete the proof that basic sets in $\overline{\cuco X}$ have nonempty interior (with respect to the original topology), and thereby complete the proof of the lemma, it suffices to show that $\overline{\cuco X}$ is homeomorphic to $\mma$.

Now, a set $A\subseteq\overline{\cuco X}$ is closed in $\overline{\cuco X}$ (i.e. has open complement) if and only if, for each $p\in\overline{\cuco X}-A$, we can choose $\epsilon>0$ and neighborhoods $\{U_T:T\in\support(p)\}$ so that $\neb_{\{U_T\},\epsilon}(p)$ is disjoint from $A$.  But this is equivalent to the following: for all basically convergent $\{a_n\}$ with each $a_n\in A$, the (basic) limit $a$ lies in $A$.  This is in turn equivalent to the assertion that $A$ is closed in $\mma$.
\end{proof}

\begin{prop}\label{prop:properties}
Let $(\cuco X,\mathfrak S)$ be hierarchically hyperbolic, and let $\overline{\cuco{X}}=\cuco X\cup \partial (\cuco X,\mathfrak S)$.
\begin{enumerate}
 \item \label{item:hausdorff}$\overline{\cuco X}$ is Hausdorff and, if $\cuco X$ is separable (e.g. if it is proper), then $\overline{\cuco X}$ is separable.
 \item \label{item:closed}$\partial \cuco X$ is closed in $\overline{\cuco X}$,
 \item \label{item:dense}$\cuco X$ is dense in $\overline{\cuco X}$.
\end{enumerate}
\end{prop}

\begin{proof}
The ``Hausdorff'' part of Assertion~\eqref{item:hausdorff} follows from Lemma~\ref{lem:hausdorff}.  Separability of $\overline{\cuco X}$ follows from density of the metric space $\cuco X$ in $\overline{\cuco X}$, i.e. part~\eqref{item:dense}.  Assertion~\eqref{item:closed} is obvious: no bounded neighborhood of an interior point contains a boundary point, so no sequence of boundary points converges to an interior point.  

It remains to prove assertion~\eqref{item:dense}.  Pick a neighborhood $\neb_{\{U_S\},\epsilon}(p)$ of $p=\sum_{S \in \support(p)} a^p_Sp_S\in\partial \cuco X$, with $p_S\in\partial \fontact S$ for $S \in \support(p)$.  For each $S_i \in \support(p) = \{S_1, \dots, S_d\}$, fix a uniform quasigeodesic ray $\gamma_i$ in $\fontact S$ from $\pi_{S}(x_0)$ to $p_S$.
 
 First, suppose that $d=1$.  Then for each $t$, there exists $x_1^t$ such that $\pi_{S_1}(x_1^t)$ coarsely coincides with $\gamma_1(a^p_{S_1}\cdot t)$ and, in view of the quasiisometric embedding $F_{S_1}\times E_{S_1}\rightarrow\cuco X$  described in Subsection \ref{sec:product_regions}, the point $x_1^t$ can be chosen so that $\pi_T(x^t_1)$ coarsely equals $\pi_T(x_0)$ for each $T\orth S_1$.   (Here we have used that $(\cuco X,\mathfrak S)$ is normalized.) 
 
Now suppose $d\geq 2$.  By induction, for all $t$, there exists $x_{d-1}^t\in E_{S_d}$ such that for all $i\leq d-1$, the projection $\pi_{S_i}(x_{d-1}^t)$ coarsely coincides with $\gamma_i(a^p_{S_i}\cdot t)$, and also $\pi_{T}(x_{d-1}^t)$ coarsely coincides with $\pi_T(x_0)$ for each $T$ orthogonal to each $S_i$. In view of the quasiisometric embedding $F_{S_d}\times E_{S_d}\rightarrow\cuco X$, there exists a point $x_{d}^t$ so that $\gate_{E_{S_d}}(x_d^t)$ coarsely coincides with $x_{d-1}^t$ and $\pi_{S_d}(x_d^t)$ coarsely coincides with $\gamma_d(a^p_{S_d}\cdot t)$.  (Here, $\gate_{E_{S_d}}$ is the gate map defined at the end of Section~\ref{sec:background}.)  For each sufficiently large $t$, the point $x_d^t$ lies in $\neb_{\{U_{S_i}\},\epsilon}(p)$, as required.
\end{proof}

\begin{rem}\label{rem:alternate_topology}
By regarding each $\boundary\fontact U$, with $U\in\mathfrak S$, as a discrete set, we can endow $\boundary(\cuco X,\mathfrak S)$ with an alternate topology as a simplicial complex, as follows.  For each $U\in\mathfrak S$ and each $p\in\boundary\fontact U$, we have a $0$--simplex, and the $0$--simplices $p_0,\ldots,p_k\in\boundary\fontact U_0,\ldots,\boundary\fontact U_k$ span a $k$--simplex if $U_i\orth U_j$ for $0\leq i<j\leq k$.  There is an obvious bijection from the resulting simplicial complex to $\boundary(\cuco X,\mathfrak S)$, which is an embedding on each simplex.
\end{rem}

\section{Compactness for proper HHS}\label{sec:compactness}

In this section, we will prove that proper HHS have compact HHS boundaries.
\subsection{Preliminary lemmas}

\begin{defn}\label{defn:level}
 Let $(X,\mathfrak S)$ be hierarchically hyperbolic. The \emph{level} $\ell_U$ of $U\in\mathfrak S$ is defined inductively as follows. If $U$ is $\nest$-minimal, then $\ell_U = 1$.  We inductively define $\ell_U = k+1$ if $k$ is the maximal integer such that there exists $V\nest U$ with $\ell_V=k$ and $V\neq U$.
\end{defn}

The following is a slightly modified version of Lemma~2.5 in~\cite{BehrstockHagenSisto:HHS_II}.  

\begin{lem}\label{lem:nest_progress}
 Let $(X,\mathfrak S)$ be hierarchically hyperbolic. Then there exists $N$ with the following property. Let $x,y\in\cuco X$ and let $\{S_i\}_{i=1,\dots,N}\subseteq \mathfrak S$ be so that $\dist_{\fontact S_i}(x,y)\geq 50E$ for each $i=1, \dots, N$.  Then there exists $S\in\mathfrak S$ and $i$ so that $S_i\propnest S$ and $\dist_{\fontact S}(x,y)\geq 100E$.  Moreover, for each $T \in \mathfrak S$ such that each $S_i \nest T$, we can choose $S \nest T$.
\end{lem}

\begin{proof}
 The proof is by induction on the level $k$ of a $\nest$-minimal $S\in\mathfrak S$ into which each $S_i$ is nested. The base case $k=1$ is empty.
 
Suppose that the statement holds for a given $N=N(k)$ when the level of $S$ as above is at most $k$. Suppose instead that $|\{S_i\}|\geq N(k+1)$ (where $N(k+1)$ is a constant much larger than $N(k)$ that will be determined shortly) and there exists a $\nest$-minimal $S\in\mathfrak S$ of level $k+1$ into which each $S_i$ is nested.  There are two cases.

If $\dist_{\fontact S}(x,y)\geq 100E$, then we are done. If not, then the large link axiom (Definition~\ref{defn:space_with_distance_formula}.\eqref{item:dfs_large_link_lemma}) implies that there exists $K=K(100E)$ and $T_1,\dots,T_K$, each properly nested into $S$ (and hence of level less than $k+1$), so that any $S_i$ is nested into some $T_j$. In particular, if $N(k+1)\geq KN(k)$, there exists $j$ so that at least $N(k)$ elements of $\{S_i\}$ are nested into $T_j$. By the induction hypothesis, we are done.

Note that the proof still works replacing $\mathfrak{S}$ with $\mathfrak{S}_T$ when each $S_i \nest T$.  In this case, we can take $S \nest T$ and the $T_i$ produced by the large link axiom will also have $T_i \nest S \nest T$ for each $i$, as required for the second statement.
\end{proof}

\begin{lem}\label{lem:unbounded_ray}
 Let $(X,\mathfrak S)$ be hierarchically hyperbolic. Then for every hierarchy ray $\gamma$ there exists $S\in\mathfrak S$ so that $\pi_S(\gamma)$ is unbounded. Moreover, if $T\in\mathfrak S$ has the property that $\{\diam_{\fontact T'}(\gamma):T'\nest T\}$ is unbounded, then there exists $S\nest T$ so that $\pi_S(\gamma)$ is unbounded.
\end{lem}

\begin{proof}
The proof of the ``moreover'' part is a minor variation; we prove the first assertion and indicate parenthetically how to adapt the proof. 

 
 By the distance formula (Theorem \ref{thm:distance_formula}) and the fact that $\gamma$ is a quasi-geodesic, there exists an increasing sequence $\{n_i\}$ of natural numbers such that for each positive integer $i$, there exists $S'_i\in\mathfrak S$ so that $\dist_{\fontact S'_i}(\gamma(n_i),\gamma(n_{i+1}))\geq 100 E$.  (For the purposes of the ``moreover'' part, we choose $S'_i$ nested into $T$.)  Since $\gamma$ is a hierarchy path, it makes coarsely monotonic progress in $\fontact U$ for each $U \in \mathfrak S$, and thus for each $t\geq0$ we have
 $$\dist_{\fontact U}(\gamma(0),\gamma(t))\geq 50E \cdot |\{i: n_i\leq t, S'_i=U\}|.$$

 Let $\mathcal S\subset \mathfrak S$ be the collection of domains in which $\gamma$ makes significant progress; that is, $\mathcal S$ is the set of all $S\in\mathfrak S$ so that there exists $t_S\geq 0$ so that for any $t\geq t_S$ we have $\dist_{\fontact S}(\gamma(0),\gamma(t))\geq 50 E$. (In the proof of the ``moreover'' part, we further require that $S$ is nested into $T$.) If $|\mathcal S|<\infty$, then we are done by the above inequality, so assume $|\mathcal{S}|=\infty$.
 
 Let $S\in \mathfrak S$ be $\nest$-minimal with the property that there are infinitely many $S'\in\mathcal S$ nested into $S$. (In the proof of the ``moreover'' part, $S$ is nested into $T$.) Suppose for a contradiction that $\diam_S(\pi_{S}(\gamma))=D<\infty$.
 
 Denote by $\mathcal S^j$ the set of all level-$j$ elements of $\mathcal S$ nested into $S$, and let $k$ be maximal with the property that $\mathcal S^k$ is infinite.  Note that this assumption and finite complexity imply that $\bigcup_{k'>k} \mathcal S^{k'}$ is finite.  To derive a contradiction, we will use the large link axiom and Lemma \ref{lem:nest_progress} to construct an infinite sequence of distinct $S_i \in \bigcup_{k'>k} \mathcal S^{k'}$.
 
 By the large link axiom (Definition~\ref{defn:space_with_distance_formula}.\eqref{item:dfs_large_link_lemma}), there exists $K=K(D)$ so that, for any $t$, there exist $T^t_1,\dots,T^t_K$ properly nested into $S$, such that if $X\in\mathcal S$ has $X \nest S$ and $t_X\leq t$, then $X \nest T^t_j$ for some $j$. If we take $t_0$ large enough, we can apply Lemma \ref{lem:nest_progress} to a sufficiently large subset of $\mathcal S^k$, all of whose elements are nested into some $T^{t_0}_j$, and we get some $S_0$ of level $k_0>k$, so that $\dist_{\fontact S_0}(\gamma(0),\gamma(t))\geq 100 E$ for $t\geq t_0$.  Note that Lemma \ref{lem:nest_progress} allows us to take $S_0 \nest T^{t_0}_j$, so that $S_0 \nest S$ and thus $S_0\in \mathcal S^{k_0}$. By minimality of $S$, there are finitely many elements of $\mathcal S^k$ nested into $S_0$. We can now choose $t_1>t_0$ and apply Lemma \ref{lem:nest_progress} to a sufficiently large subset of $\mathcal S^k$ all whose elements are nested into some $T^{t_1}_j$ but not nested into $S_0$, and get another element $S_1\in\mathcal S^{k_1}$, for some $k_1>k$, which is properly nested into $S$. We can then proceed inductively and construct infinitely many distinct elements $S_i \nest S$ of level greater than $k$, giving us our contradiction.\end{proof}

\subsection{Compactness}
We are ready to prove:

\begin{thm}\label{thm:cpt}
 Let $(\cuco X,\mathfrak S)$ be hierarchically hyperbolic, and let $\overline{\cuco{X}}=\cuco X\cup \partial (\cuco X,\mathfrak S)$. If $\cuco X$ is proper, then $\overline{\cuco{X}}$ is compact.
\end{thm}

\begin{proof}
It suffices to show that $\overline{\cuco{X}}$ is sequentially compact since it is separable by Proposition \ref{prop:properties}.  We will first show that any internal sequence $\{x_n\} \subset \cuco X$ subconverges to some point in $\overline{\cuco X}$.  Then we will show this suffices for the theorem.

\textbf{Internal sequences subconverge:}  Let $\{x_n\} \subset \cuco X$ be a sequence of interior points.  For each $n$, let $\gamma_n$ be a uniformly Lipschitz hierarchy path between $x_0$ and $x_n$, whose existence is guaranteed by Theorem~\ref{thm:monotone_hierarchy_paths}.  Since $\cuco X$ is proper, either the sequence $x_n$ subconverges to an interior point and we are done, or we can assume that the sequence of hierarchy paths $\gamma_n$ converges to a hierarchy ray, $\gamma_{\infty}$.  

Lemma \ref{lem:unbounded_ray} implies there exists $T\in\mathfrak S$ such that $\pi_T\circ\gamma_\infty$ is unbounded.  The collection $\{T_i\}^k_{i=1}$ for which this is true must be a collection of pairwise-orthogonal elements by the consistency inequalities (Definition~\ref{defn:space_with_distance_formula}.\eqref{item:dfs_transversal}). For each $T_i$, the quasigeodesic ray $\pi_{T_i}\circ\gamma_{\infty} \subset \fontact T_i$ represents a point $p_{T_i}\in\partial \fontact T_i$.  Set $\overline{T} = \{T_i\}^k_{i=1}$.

We now consider two cases depending on the behavior of the sequence $\{x_n\}$ in $\overline{T}^{\orth}$.  First, suppose $\liminf_n \sup\{\dist_{\fontact T}(x_0,x_n): T\in \overline{T}^{\orth}\}<\infty$.  Up to passing to a further subsequence of $\{x_n\}$, we have well-defined limits for $1\leq i,j\leq k$
$$r_{i,j}=\lim_n \frac{\dist_{\fontact T_i}(x_0,x_n)}{\dist_{\fontact T_j}(x_0,x_n)}\in [0,\infty],$$
which determine coefficients $\{a^p_i\in[0,1]\}$ such that $a^p_i/a^p_j=r_{i,j}$ and $\sum a^p_i=1$. It is straightforward to check that $\{x_n\}$ eventually lies in the interior part of any $\neb_{\{U_{T_i}\},\epsilon}(p)$, implying that $\{x_n\}$ subconverges to $p=\sum_{T \in \overline{T}} a^{p}_{T}p_{T}$.

Now suppose that, up to passing to a subsequence, $\liminf_n \sup\{\dist_{\fontact T}(x_0,x_n): T \in \overline{T}^{\orth}\}=\infty$. Consider the sequence $\{y_n\}=\{\gate_{E_{\overline{T}}}(x_n)\}$ of gates in the orthogonal complement of $\overline{T}$.

Since $\left(E_{\overline{T}}, \mathfrak{S}_{\overline{T}^{\orth}}\right)$ is an HHS with complexity strictly less than that of $(\cuco X, \mathfrak S)$, by induction on the complexity of $(\cuco X,\mathfrak S)$, the sequence $\{y_n\}$ subconverges to $q \in \partial \cuco X$, where $\support(q) = \{T_i\}^{k'}_{i=k+1}$ and $T_i \orth T_j$ whenever $i\leq k < j$.  Since $\left(E_{\overline{T}}, \mathfrak{S}_{\overline{T}^{\orth}}\right) \subset (\cuco X, \mathfrak S)$ is hierarchically quasiconvex, we can take $q \in \partial E_{\overline{T}}$.  For each $j>k$, let $q_{T_j} \in \partial \fontact T_j$, so that $q$ is a linear combination of the $q_{T_j}$.  As before, up to passing to a further subsequence, for any $1\leq i,j \leq k'$, we can define
$$r_{i,j}=\lim_n \frac{\dist_{\fontact T_i}(x_0,x_n)}{\dist_{\fontact T_j}(x_0,x_n)}\in [0,\infty],$$
which determine coefficients $\{a^p_{T_i}\}_{i=1}^k \cup \{a^q_{T_j}\}_{j=k+1}^{k'}$ such that
\begin{itemize}
\item $a^r_{T_i}/a^{r'}_{T_j}=r_{i,j}$, when $r, r' \in \{p,q\}$ and $a^r_{T_i},a^{r'}_{T_j}$ are defined, and
\item $\sum^{k}_{i=1} a^p_{T_i} + \sum_{j=k+1}^{k'} a^q_{T_j}=1$.
\end{itemize}
If some $a_{T_i}^{r} = 0$ for $r \in \{p,q\}$, we disregard $T_i$. We now claim that $\{x_n\}$ (sub)converges to 
$$p=\sum_{i=1}^k a^{p}_{T_i}p_{T_i} + \sum_{i=k+1}^{k'} a^{q}_{T_i} q_{T_i}.$$

Pick a neighborhood $\neb_{\{U_{T_i}\},\epsilon}(p)$ of $p$.  For large enough $n$, $x_n \in \neb_{\{U_{T_i}\},\epsilon}(p)$ because:
\begin{itemize}
 \item $\pi_{T_i}(x_n)\in U_{T_i}$ for $i\leq k$ since $\left(\pi_{T_i}(x_n)|p_{T_i}\right)_{\pi_{T_i}(x_0)}\rightarrow \infty$,
 \item $\pi_{T_i}(x_n)\in U_{T_i}$ for $i>k$ since $\pi_{T_i}(x_n)$ coarsely equals $\pi_{T_i}(y_n)$ and $y_n\to q$,
 \item $\left|\frac{a^r_{T_j}}{a^{r'}_{T_i}} - \frac{\dist_{T_j}(x_0,x_n)}{\dist_{T_i}(x_0,x_n)}\right|< \epsilon$ by definition,  when $r, r' \in \{p,q\}$ and $a^r_{T_i},a^{r'}_{T_j}$ are defined, and 
 \item $\frac{\dist_{T}(x_0,x_n)}{\dist_{T_i}(x_0,x_n)}< \epsilon$ for $T \in \left(\{T_i\}_{i=1}^{k'}\right)^{\orth}$ and any $1\leq i \leq k'$, as we now show. 
\end{itemize}
Let $T \in \left(\{T_i\}_{i=1}^{k'}\right)^{\orth}$ and choose $i$ so that $a^r_{T_i}\neq 0$, for $r \in \{p,q\}$.  Observe that 
$$\frac{\dist_{T}(x_0,x_n)}{\dist_{T_i}(x_0,x_n)}=\frac{\dist_{T}(x_0,x_n)}{\dist_{T_{k+1}}(x_0,x_n)} \cdot \frac{\dist_{T_{k+1}}(x_0,x_n)}{\dist_{T_i}(x_0,x_n)}.$$
The first term on the right-hand side can be made arbitrarily small by increasing $n$ since $\dist_{T}(x_0,x_n)$ (resp. $\dist_{T_{k+1}}(x_0,x_n)$) coarsely coincides with $\dist_{T}(x_0,y_n)$ (resp. $\dist_{T_{k+1}}(x_0,y_n)$) and 
$\{y_n\}$ converges to $q$.  
Since the second term converges to $r_{k+1,i}<\infty$, this proves the claim and completes the internal sequence case.

\textbf{Reduction to the internal sequence case:}  Recall the definition of the boundary projection, Definition \ref{defn:remote_part}).  By passing to a subsequence if necessary, it suffices to consider any boundary sequence $\{z_n\} \subset \partial {\cuco X}$, where $z_n = \sum_{S \in \support(z_n)} a^{z_n}_{S} p^n_{S}$ for each $n$.  

We first find $\{x_n\}\subset\cuco X$ with the properties \eqref{item:unbounded}--\eqref{internal condition 7} below, and then verify that $\{z_n\}$ subconverges to the limit of $\{x_n\}$:
\begin{enumerate}
 \item \label{item:unbounded} $\dist_{\cuco X}(x_0,x_n) \geq n$,
 \item \label{item:follows_z_n_in_support} $\left(\pi_{S}(x_n)|p^n_S\right)_{\pi_{S}(x_0)}\geq n$ for each $S \in \support(z_n)$ (we remind the reader that the notation $(\bullet|\bullet)_{_\bullet}$ denotes the Gromov product with respect to the subscripted basepoint),
 \item \label{item:correct_slope} $\left|\frac{a^n_S}{a^n_{S'}} - \frac{\dist_{S}(x_0,x_n)}{\dist_{S'}(x_0,x_n)}\right|< 1/n$ for any distinct $S, S' \in \support(z_n)$,
 \item For any $T \in (\support(z_n))^{\orth}$ and $S \in \support(z_n)$, we have $\frac{\dist_{T}(x_0,x_n)}{\dist_{S}(x_0,x_n)}< 1/n$. \label{item:correct_slope_x}
 \item For all $n$ and $S^n \in \support(z_n)$, if $T \pitchfork S^n$ or $S^n \nest T$, then $\dist_T(\rho_T^{S^n}, x_n) < K$, for some uniform $K>0$.  Moreover, for all such $T$, we have $\dist_{T}(x_0,x_n)\leq\dist_{ S^n}(x_0,x_n)$. \label{internal condition 5}
 \item $\{x_n\}$ converges to $p=\sum_{T \in \support(p)} a^p_T p_T\in\boundary\cuco X$ with the following property: if there are infinitely many $n$ for which $z_n\in\boundary^{rem}\cuco X$ (with respect to $\support(p)$), then there are infinitely many remote $z_n$ such that the following holds for some fixed $T\in\support(p)$: there exists $S^n_T\in\support(z_n)$ so that $S^n_T\transverse T$ or $S^n_T\propnest T$, or $T\propnest S_T^n$ but $\dist_{S_T^n}(\rho^T_{S^n_T},x_0)\leq 100K'E$, for some constant $K'\geq 1$ depending on $\{z_n\}$ and $p$ but not on $n$.  Moreover, for all such $T$, we have $\dist_{\fontact T}(x_0,x_n)\leq\dist_{\fontact S^n}(x_0,x_n)$.\label{internal condition 6}
 \item \label{internal condition 7}$\{x_n\}$ converges to $p=\sum_{T \in \support(p)} a^p_T p_T\in\boundary\cuco X$ with the following property: if there are infinitely many $n$ for which $z_n\in\boundary^{rem}\cuco X$ (with respect to $\support(p)$), then there are infinitely many remote $z_n$ such that $\dist_T((\boundary\pi_{\support(p)}(z_n))_T,x_n)\leq K''$ for some $K''$ independent of $n$ and all $T\in\support(p)_{z_n}$.  Moreover, for all such $T$, we have $\dist_{\fontact T}(x_0,x_n)\leq\dist_{\fontact S^n}(x_0,x_n)$.
\end{enumerate}

To see that such an internal sequence exists, choose a sequence $\{x_n\}$ so that $x_n\in P$ for all $n$, where:
$$P=\image\left(\prod_{S\in\support(z_n)}F_S\to\cuco X\right);$$
the sequence $\{x_n\}$ satisfies \eqref{item:unbounded}-\eqref{item:correct_slope_x} (which can be done since they are component-wise conditions); and $$\min_{S\in\support(z_n)}\dist_{\cuco X}(\gate_{F_S}(x_n),x_0)/\dist_{\cuco X}(\gate_{F_S}(x_0),x_0)\to\infty$$ as $n\to\infty$.  Here we fix, for each $n$, a basepoint $(p_S)_{S\in\support(z_n)}$ and let $F_S=F_S\times\{(P_{S'})_{S'\neq S}\}$.

(Recall from~\cite[Remark 5.12]{BehrstockHagenSisto:HHS_II} that, whenever $U_1,\ldots,U_k\in\mathfrak S$ are pairwise orthogonal, we have a standard quasi-isometric embedding $\prod_{i=1}^k F_{U_i}\to\cuco X$ whose image is hierarchically quasiconvex and which is, for each $i\leq k$, the restriction of the usual map $F_{U_i}\times E_{U_i}\to\cuco X$.)

We can verify condition \eqref{internal condition 5} by examining the product regions $\prod_{S\in\support(z_n)}F_{S}\to\cuco X$.  Let $T\transverse S^n$ or $S^n\propnest T$ for $S^n\in\support(z_n)$.  Since $x_n$ coarsely lies in $\prod_{S\in\support(z_n)}F_{S}$, it follows that $\diam_T(\rho^{S_n}_T\cup\pi_T(F_{S^n})) \asymp 1$ and $\dist_T(\pi_T(F_{S^n}), x_n) \asymp 1$.  We thus have, for some uniform $C$,  $$\dist_T(x_0,x_n)\leq C\dist_{\cuco X}\left(x_0,\prod_{S\in\support(z_n)}F_S\right)+C.$$ For sufficiently large $n$, our choice of $\{x_n\}$ ensures that $\dist_{S_n}(x_0,x_n)\geq C\dist\left(x_0,\prod_{S\in\support(z_n)}F_S\right)+C$, verifying the ``moreover'' part of assertion~\eqref{internal condition 5}.  

Let $\{x_n\}$ satisfy \eqref{item:unbounded}--\eqref{internal condition 5}.  We now prove that there is a subsequence of $\{x_n\}$ satisfying \eqref{internal condition 6}.

By replacing $\{x_n\}$ with a subsequence (and replacing $\{z_n\}$ with the corresponding subsequence of $\{z_n\}$), we can apply the proof that internal sequences subsequentially converge to conclude $\{x_n\}$ converges to $p=\sum_{T \in \support(p)} a^p_T p_T\in\boundary\cuco X$.

Consider the set $\mathbb G$ of $n\in\naturals$ so that $z_n$ is remote with respect to $p$.  If $\mathbb G$ is finite, then~\eqref{internal condition 6} holds vacuously.  Otherwise, by replacing $\mathbb G$ with an infinite subset, we find $T\in\support(p)$ so that for all $n\in\mathbb G$, there exists $S^n\in\support(z_n)$ with either $T\transverse S^n$ or $S^n\propnest T$ or $T\propnest S^n$. 

First consider the case where $\{S^n:n\in\mathbb G\}$ is infinite.  By passing to a subsequence if necessary, and then applying finite complexity, Lemma~\ref{lem:pairwise_orthogonal}, and Ramsey's theorem, we can assume that $S^n\transverse S^m$ when $n\neq m$.  Let $\mathbb G_T\subseteq\naturals$ be the set of $n\in\mathbb G$ such that $T\propnest S^n$.  Then for all $m,n\in \mathbb G_T$, we have $\dist_{S^m}(\rho^T_{S^m},\rho^{S^n}_{S^m})\leq E$ by the consistency inequalities.  Hence, again by the consistency inequalities and the triangle inequality, we have $\dist_{S^n}(\rho^T_{S^n},x_0)\leq 2E$ for all but at most one element of $\mathbb G_T$.  Indeed, if $\dist_{S^n}(\rho^T_{S^n},x_0)>2E$, then $\dist_{S^n}(\rho^{S^m}_{S^n},x_0)>E$ for any $m\in \mathbb G_T-\{n\}$, so by consistency $\dist_{S^m}(\rho^{S^n}_{S^m},x_0)\leq E$; the claim follows from the triangle inequality since $\dist_{S^m}(\rho^T_{S^m},\rho^{S^n}_{S^m})\leq E$.  Hence, by replacing $\{z_n\}$ with a subsequence, for all $T\in\support(p)$ with $T\propnest S^n$, we have $\dist_{S^n}(\rho^T_{S^n},x_0)\leq 100K'E$.  Letting $S^n_T=S^n$ for $n\in\mathbb G$, this establishes assertion~\eqref{internal condition 6} when $\{S^n:n\in\mathbb G\}$ is infinite.

When $\{S^n:n\in\mathbb G\}$ is finite, we can assume that $S^n=S^m$ for all $m,n$ by passing to a subsequence.  Hence, there exists $S\in\mathfrak S$ so that for all $n\in\mathbb G$, and all $U\in\support(z_n)$, either $U=S$ or $U\orth T$.  Fix $T$ and $S$ as above, and replace $(z_n)$ with a subsequence so that for each $n\in\mathbb G$, we have $S\in\support(z_n)$.  Then, for each $n\in\mathbb G$, set $S^n_T=S$ and observe that either $S\nest T,S\transverse T$, or $T\nest S$. In the latter case, take $K'=\dist_S(\rho^T_S,x_0)$, which depends on $p$ and $\{z_n\}$ but not on $n$.  This completes the proof of~\eqref{internal condition 6}.

We now deduce condition~\eqref{internal condition 7} from \eqref{item:unbounded}--\eqref{internal condition 6}.  Assume $\mathbb G$ is infinite, so that by~\eqref{internal condition 6}, there exists $T'\in\support(p)$ so that, after replacing $\mathbb G$ with an infinite subset if necessary, we have for each $n\in\mathbb G$ some $S^n_{T'}\in\support(z_n)$ so that $\dist_{S^n_{T'}}(\rho^T_{S^n_{T'}},x_0)\leq100K'E$.  Let $T\in\support(p)_{z_n}$.  First suppose that $T\propnest S^n_{T'}$.  Then since $T\orth T'$ or $T=T'$, Lemma~\ref{lem:orthogonal_close} implies that $\dist_{S^n_{T'}}(\rho^{T}_{S^n_{T'}},x_0)\leq200K'E$.  It follows from~\eqref{item:follows_z_n_in_support} that $(\pi_{S^n_{T'}}(x_n)|p^n_{S^n_{T'}})_{\rho^{T}_{S^n_{T'}}}\to\infty$ as $n\to\infty$ so that, by discarding finitely many $n$ and applying the bounded geodesic image axiom, we have $\dist_{T}((\boundary\pi_{\support(p)}(z_n))_{T},x_n)\leq E$ for all $n\in\mathbb G$.  In the remaining cases, where $T\transverse S^n_{T'}$ or $S^n_{T'}\propnest T$, then we reach the same conclusion, using~\eqref{internal condition 5} instead of~\eqref{internal condition 6}.  This completes the proof of condition \eqref{internal condition 7}.

\emph{Subconvergence of $\{z_n\}$:}  Fix a neighborhood $\neb=\neb_{\{U_S\},\epsilon}(p)$ of $p$; we must check that for infinitely many values of $n$, we have $z_n\in\neb$.  For each $n$, either $z_n\in\boundary^{rem}\cuco X$ (recall that this means that $\support(z_n)\cap\support(p)=\emptyset$ and for all $T\in\support(p)$, there exists $S\in\support(z_n)$ with $T\notorth S$) or $z_n\in\boundary\cuco X-\boundary^{rem}\cuco X$ (so that either $\support(z_n)\cap\support(p)\neq\emptyset$ or there exists $T\in\support(p)$ with $T\orth S$ for all $S\in\support(z_n)$).

\emph{The non-remote case:}  We will consider the non-remote case first.  Recall that $z_n = \sum_{S \in \support(z_n)} a^{z_n}_S p^n_S$.  We must check the following conditions:

\begin{enumerate}[(a)]
\item For each $S \in \support(p) \cap \support(z_n)$, and infinitely many $n$, we have $p_S^n \in U_S$.  \label{nonremote condition 1}
\item For each $S \in \support(p) \cap \support(z_n)$ and infinitely many $n$, we have $a_S^n \rightarrow a_S^p$.  \label{nonremote condition 2}
\item The sum $\sum_{T \in \support(p) - \support(z_n)} a_T^p < K'\epsilon$ for infinitely many $n$ and some uniform $K'$. \label{nonremote condition 3}

\end{enumerate}

Up to passing to a subsequence, \eqref{nonremote condition 1} follows from \eqref{item:follows_z_n_in_support} and the fact that $x_n \rightarrow p$.

For \eqref{nonremote condition 2}, we have three cases.  If $\support(p) \cap \support(z_n) = \emptyset$, then this holds vacuously.  If $\support(p) \cap \support(z_n)$ has multiple elements, then this follows from \eqref{item:correct_slope} and the fact that $x_n \rightarrow p$.  If $\support(p) \cap \support(z_n) = \{S\}$, then this follows from \eqref{item:correct_slope} and \eqref{nonremote condition 3}, proved momentarily.

To see \eqref{nonremote condition 3}, first observe that $\support(p) - \support(z_n) \subset \left(\support(z_n)\right)^{\orth}$ by non-remoteness.  Let $T \in \support(p) - \support(z_n)$ and $S \in \support(p) \cap \support(z_n)$; note that such an $S \in \support(p) \cap \support(z_n)$ exists, otherwise one of $x_n \rightarrow p$ or \eqref{item:correct_slope_x} is contradicted.  By definition of $x_n \rightarrow p$,
$$\left| \frac{a^p_T}{a^p_S} - \frac{\dist_T(x_0, x_n)}{\dist_S(x_0, x_n)}\right| < \epsilon.$$

It follows from \eqref{item:correct_slope_x} that $\frac{\dist_T(x_0, x_n)}{\dist_S(x_0, x_n)}< \frac{1}{n}$.  Since each $a^p_S \leq 1$, it follows that 
$$ \sum_{T\in \support(p) - \support(z_n)} a_T^p < \xi(\cuco X) \left(\epsilon + \frac{1}{n}\right)\leq2\xi(\cuco X)\epsilon,$$
completing the proof of \eqref{nonremote condition 3} and thus the non-remote case.

\emph{The remote case:} We must check the following conditions:

\begin{enumerate}[(i)]
 \item For any $T\in\support(p)$, and infinitely many $n$, we have $\left(\partial \pi_{\support(p)}(z_n)\right)_T  \in U_T$.\label{gen con 1}
 \item For infinitely many $n$ and any $T\in\support(p)_{z_n},T'\in\support(p)$, we have $$\left|\frac{\dist_{T}(x_0,(\partial\pi_{\support(p)}(z_n))_T)}{\dist_{T'}(x_0,(\partial\pi_{\support(p)}(z_n))_{T'})}-\frac{a_T^p}{a_{T'}^p}\right|<\epsilon.$$ \label{gen con 2}
 \item \label{gen con 3} We have $\sum_{T\in\support(p)^\orth\cap\support(z_n)}a^{z_n}_T<K\epsilon$ for some uniform $K$.
\end{enumerate}

For any $T \in \support(p)$ and each $n$, choose $S^n_T \in \support(z_n)$ so that $T$ and $S^n_T$ are not orthogonal.  If $\mathbb G$ is infinite, then we may pass to a subsequence so that $S_T^n$ and $T$ are always non-orthogonal: that is, $T\propnest S_T^n$, or $T\transverse S_T^n$, or $S_T^n\propnest T$.

We now show that assertion~\eqref{gen con 1} holds for infinitely many $n$; the proof divides into three cases according to the above possibilities, which influence the definition of $(\partial\pi_{\support(p)}(z_n))_T$.

First, if $S_T^n\transverse T$, then $(\partial\pi_{\support(p)}(z_n))_T=\rho^{S_T^n}_T$.  In this case, \eqref{gen con 1} follows immediately from conditions \eqref{item:follows_z_n_in_support} and \eqref{internal condition 5} in the definition of $\{x_n\}$.  The same is true if $S^n_T\propnest T$.  If $T \nest S^n_T$, then \eqref{gen con 1} follows from \eqref{item:follows_z_n_in_support}, \eqref{internal condition 7}, and the triangle inequality.

Assertion \eqref{gen con 2}, in the case when $T,T'\in\support(p)$, follows from \eqref{internal condition 7}. In fact, since $\{x_n\}$ converges to $p$, we have
$$\left|\frac{\dist_T(x_0,x_n)}{\dist_{T'}(x_0,x_n)}-\frac{a_T^p}{a_{T'}^p}\right|\to 0,\ \ \ (*)$$
and $\dist_T(x_0,x_n)\to\infty$, $\dist_{T'}(x_0,x_n)\to\infty$. By \eqref{internal condition 7}, we have that $\dist_T(x_0,x_n)$ coarsely coincides with $\dist_{T}(x_0,(\partial\pi_{\support(p)}(z_n))_T)$, and similarly for $T'$. Hence, $(*)$ implies that the ratio in Assertion \eqref{gen con 2} satisfies the required inequality. If $T\in \support(p)_{z_n}-\support(p)$, then we have to verify $\left|\frac{\dist_{T}(x_0,(\partial\pi_{\support(p)}(z_n))_T)}{\dist_{T'}(x_0,(\partial\pi_{\support(p)}(z_n))_{T'})}\right|\to 0$. We still know $(*)$ (with $\frac{a_T^p}{a_{T'}^p}$ replaced by $0$) and $\dist_{T'}(x_0,x_n)\to\infty$. If $\dist_T(x_0,x_n)$ does not diverge, we are done. If it does, we can approximate $\dist_{T}(x_0,(\partial\pi_{\support(p)}(z_n))_T)$ by $\dist_T(x_0,x_n)$ and we can conclude as above.  

It remains to verify assertion~\eqref{gen con 3}.  For each $n$, let $T^n\in(\support(p))^\orth\cap\support(z_n)$ and choose $S^n\in\support(z_n)-(\support(p))^\orth$.  Fix $P\in\support(p)$ so that, after passing to a subsequence, $P$ is not orthogonal to any of the $S^n$.  By either~\eqref{internal condition 5} or~\eqref{internal condition 7}, we have $\dist_{\fontact S^n}(x_0,x_n)/\dist_{\fontact P}(x_0,x_n)\leq 1$, while $\dist_{\fontact P}(x_0,x_n)/\dist_{\fontact T^n}(x_0,x_n)<\epsilon$ since $x_n\to p$.  Hence $a_{T^n}^{z_n}/a_{S^n}^{z_n}\leq\epsilon+\frac{1}{n}$, by~\eqref{item:correct_slope}, and the desired inequality follows since the number of terms in the sum is bounded by $\xi(\cuco X)$,  as in the non-remote case.  This completes the proof that $\{z_n\}$ subconverges to $p$, and thus completes the proof that $\boundary\cuco X$ is compact.
\end{proof}

\section{The HHS boundary of a Gromov hyperbolic space}\label{sec:gromov_boundary}
In this section, we prove that the HHS boundary of a hyperbolic space is its Gromov boundary, regardless of the chosen HHS structure.

\begin{lem}\label{lem:hyp no product}
Let $(\cuco X, \mathfrak S)$ be hierarchically hyperbolic.  If $\cuco X$ is hyperbolic, then there exists $C>0$ such that if $U, V \in \mathfrak S$ and $U \orth V$, then either $\diam \fontact U < C$ or $\diam \fontact V < C$.
\end{lem}

\begin{proof}
Recall from~\cite{BehrstockHagenSisto:HHS_II} that if $U \orth V$, then there exists a quasiisometric embedding $F_{U} \times F_{V} \hookrightarrow \cuco X$.  Hyperbolicity uniformly bounds the diameter of one of the factors.
\end{proof}

\begin{lem}\label{lem:hyp unique ray}
Let $(\cuco X, \mathfrak S)$ be hierarchically hyperbolic and $\cuco X$ hyperbolic.  If $\gamma: [0,\infty) \rightarrow \cuco X$ is a hierarchy ray with $\gamma(0)=x_0$, then there exists a unique $U \in \mathfrak S$ with $\pi_U \circ \gamma: [0,\infty) \rightarrow \fontact U$ a parametrized quasigeodesic ray.  In particular, $\diam_{\fontact V}(\gamma) < \infty$ for all $V \in \mathfrak S$ with $V \neq U$.
\end{lem}

\begin{proof}
By Lemma \ref{lem:unbounded_ray}, there exists $U \in \fontact S$ such that $\diam_{\fontact U}(\gamma)$ is unbounded.  Let $V \in \mathfrak S$ be such that $V \neq U$; by Lemma \ref{lem:hyp no product}, there are three cases: $V \nest U$, $U \nest V$, and $V \pitchfork U$.

Let $t_M \in [0,\infty)$ be such that $\dist_{\fontact U}(\gamma(0),\gamma(t)) > E^2$ for $t\geq t_M$.  If $U \nest V$, then by the consistency inequality, $\dist_V(\gamma(t), \rho^V_U(\gamma(0))) < E$ for all $t > t_M$.  If $V \nest U$, then $\dist_{\fontact V}(\gamma(t), \rho^U_V) < E$ for all $t > t_M$.  Similarly, if $U \pitchfork V$, then $\dist_{\fontact V}(\gamma(t), \rho^V_U) < E$ for all $t > t_M$ by the transverse case of the consistency inequality.  Thus, in each case, $\diam_{\fontact V}(\gamma) < \infty$.
\end{proof}

\begin{thm}\label{thm:hyperbolic}
 Let $(\cuco X,\mathfrak S)$ be hierarchically hyperbolic and suppose that $\cuco X$ is hyperbolic. Let $\overline{\cuco X}^{Gr}=\cuco X\cup \partial^{Gr}\cuco X$, where $\partial^{Gr}\cuco X$ is the Gromov boundary of $\cuco X$, and let $\
\overline{\cuco X}=\cuco X\cup \partial\cuco X$. Then the identity map $\cuco X\to \cuco X$ extends uniquely to a homeomorphism $\overline{\cuco X}^{Gr}\to\overline{\cuco X}$.
\end{thm}

\begin{proof}
Lemma \ref{lem:hyp no product} gives $\partial \cuco X = \coprod_{U \in \mathfrak S} \partial \fontact U$ and Lemma \ref{lem:hyp unique ray} gives $|\support(p)|=1$ for all $p\in\partial \cuco X$.

Fix $x_0 \in \cuco X$ and let $p \in \partial^{Gr} \cuco X$.  Let $\gamma_p:[0,1) \rightarrow \overline{\cuco X}^{Gr}$ be a geodesic from $x_0$ to $p$.  For any $n \in \naturals$, let $\gamma_n:[0,n) \rightarrow \cuco X$ be a hierarchy path between $x_0$ and $\gamma_p(n)$.  Since $\cuco X$ is hyperbolic, each $\gamma_n$ uniformly fellow-travels $\gamma_p$ and thus $\gamma = \lim_n \gamma_n$ is a hierarchy ray from $x_0$ to $p$.  The ray $\gamma$ is independent of the choice of $(\gamma_n)$ and is thus uniquely determined by $p$. By Lemma \ref{lem:hyp unique ray}, there exists a unique $U \in \mathfrak S$ such that $\diam_{\fontact U}(\gamma)$ is an unbounded quasigeodesic ray.  By hyperbolicity of $\fontact U$, there exists $q\in \partial \fontact U$ such that $\pi_{\fontact U}(\gamma)$ limits to $q$.

The above discussion yields a well-defined map $\phi^{Gr}: \partial^{Gr} \cuco X \rightarrow \partial \cuco X$ given by $\phi^{Gr}(p) = q$.  Define $\phi:\overline{\cuco X}^{Gr} \rightarrow \overline{\cuco X}$ by $\phi|_{\cuco X} = id_{\cuco X}$ and $\phi|_{\overline{\cuco X}^{Gr}} = \phi^{Gr}$.  We claim that $\phi$ is a homeomorphism.

\textbf{Bijectivity:}  The map $\phi$ is clearly bijective on $\cuco X$.  Let $p,q\in\boundary^{Gr}\cuco X$ and suppose that $p\neq q$.  Then there exist geodesic rays $\gamma_p, \gamma_q:[0,\infty] \rightarrow \cuco X$ with $[\gamma_p] = p$, $[\gamma_q]=q$, and $\gamma_p(0)=\gamma_q(0) = x_0$.  Since $p \neq q$, hyperbolicity of $\cuco X$ implies that $\dist_{\cuco X}(\gamma_p(t), \gamma_q(t)) \rightarrow \infty$.

By Lemma \ref{lem:hyp unique ray}, $\gamma_p$ and $\gamma_q$ have unique domains $U_p$ and $U_q$, respectively, to which they have unbounded projections.  If $U_p \neq U_q$, we are done.  Otherwise, $U_p = U_q=U$, and Lemma \ref{lem:hyp unique ray}, the distance formula, and the triangle inequality imply that $d_U(\gamma_p(t), \gamma_q(t)) \rightarrow \infty$, whence $\phi(p) \neq \phi(q)$, by definition.  Thus $\phi$ is injective; surjectivity of $\phi$ follows from Theorem~\ref{thm:realization}.

\textbf{Basic sets in $\overline{\cuco X}$:}  For convenience, we describe basic sets $\mathcal N(p)$, for $p\in\boundary(\cuco X,\mathfrak S)$, in our current simple situation.  Observe that $\support(p)$ consists of a single $S\in\mathfrak S$, while $\remote{\support(p)}{\cuco X}$ consists of those $q\in\boundary(\cuco X,\mathfrak S)$ with $\support(q)=\{T\}$ with $T\neq S$.  It is automatic that $T$ is not orthogonal to $S$: if $T\orth S$, then Lemma~\ref{lem:hyp no product} implies only one of $\fontact S$ or $\fontact T$ can be unbounded and thus have nonempty Gromov boundary.  It follows that $\support(q)\cap(\support(p))^\orth=\emptyset$.

Choosing $\epsilon>0$ and $p\in\mathcal U_S\subset\fontact S\cup\boundary\fontact S$, a remote neighborhood of $p$ in $\overline{\cuco X}$ is:

$$\neb_{\mathcal U_S,\epsilon}^{rem}(p)=\left\{q\in\bigsqcup_{S\neq T}\boundary\fontact T\Big |\rho^T_S\in \mathcal U_{S}\right\}.$$

\noindent Meanwhile, the nonremote part of the boundary is just $\boundary\contact S$, so

$$\neb^{non}_{\mathcal U_S,\epsilon}(p)=\mathcal U_S.$$

\noindent Finally, the interior part is:

$$\neb^{int}_{\mathcal U_S,\epsilon}(p)=\left\{x\in\cuco X\Big|\pi_{S}(x)\in U_S, \frac{\dist_{T}(x_0,x)}{\dist_{S}(x_0,x)}< \epsilon\,\,\forall T\orth S\right\}.$$

The above descriptions will be useful in proving that $\phi$ is a homeomorphism.

\textbf{Continuity of $\phi,\phi^{-1}$:}  Choose $p\in\boundary(\cuco X,\mathfrak S)$, supported on $S\in\mathfrak S$, a neighborhood $\mathcal U_S$ of $p\in\boundary\fontact S$, and $\epsilon>0$.  We may assume that
$$\mathcal U_S=\left\{y\in\fontact S\cup\boundary\fontact S\Big|\exists (p_n):p_n\to p,\,\liminf_n\left(y\mid \pi_S(p_n)\right)_{\pi_S(x_0)}>r\right\}$$
for some $r\geq0$.  Choose $q\in\boundary^{Gr}\cuco X$ so that $\phi(q)=p$.  For each $r'\geq 0$, let $$\mathbf U(q,r')=\left\{y\in\cuco X\cup\boundary^{Gr}\cuco X\Big|(y|q)_{x_0}\geq r'\right\}.$$  Recall that sets of this type yield a neighborhood basis in $\overline{\cuco X}^{Gr}$.

We exhibit $r'\geq0$, depending on $p,r,\epsilon$ and the distance formula constants, such that 
$$\phi\left(\mathbf U(q,r')\right)\subseteq \neb_{\mathcal U_S,\epsilon}(p).$$  

Indeed, if $y\in\mathbf U(q,r')\cap\boundary^{Gr}\cuco X$, and $r'$ is sufficiently large, then any geodesic ray or segment representing $[\pi_S\circ\gamma_y]$ has an initial segment of length at least $r$ lying $2\delta$-close to the corresponding segment for $p$.  This implies that $\phi(y)\in\mathcal U_S$, which is exactly the non-remote part of $\neb_{\mathcal U_S,\epsilon}(p)$ (regardless of the choice of $\epsilon$). If $y\in\mathbf U(q,r')$ is an interior point, and $r'$ is sufficiently large, then similarly $\pi_S(x)\in\mathcal U_S$.

If $T\orth S$, then, by Lemma~\ref{lem:hyp no product}, there exists a uniform $C>0$ such that $\dist_T(x_0,y)\leq C$.  Moreover, choosing $r'$ sufficiently large compared to $r,C,$ and the constants in the distance formula, we have $\dist_S(x_0,y)\geq C/\epsilon$.  Hence either $y$ is interior or $y \in \partial \fontact S$, and so $$\phi(\mathbf U(q,r'))\subseteq \neb^{non}_{\mathcal U_S,\epsilon}(p)\cup\neb^{int}_{\mathcal U_S,\epsilon}(p).$$

Continuity follows easily: Given an open set $\mathcal O\subseteq\overline{\cuco X}$, let $q\in\phi^{-1}(\mathcal O)$.  Then, since $\mathcal O$ is open, it contains a neighborhood $\mathcal N$ of $\phi(q)$.  The preceding discussion shows that $q$ lies in some neighborhood $\mathbf U$ which in turn lies in $\phi^{-1}(\mathcal N)\subset\phi^{-1}(\mathcal O)$, so $\phi^{-1}(\mathcal O)$ is open.  Continuity of $\phi^{-1}$ is proved similarly.
\end{proof}

\section{Extending hieromorphisms to the boundary}\label{sec:extending}

Hieromorphisms need not extend continuously to the boundary, but under additional hypotheses on the quasi-isometries implicit in the hieromorphism, such extensions do exist.  However, the class of hieromorphisms that extend continuously to the boundary is contained in a larger class of maps with this property, and, given the examples we study later in this section, it is in our interest to focus on this larger class of maps.

\begin{defn}[Slanted hieromorphism]\label{defn:slanted_hieromorphism}
Let $(\cuco X,\mathfrak S),(\cuco X',\mathfrak S')$ be hierarchically hyperbolic spaces.  A \emph{slanted hieromorphism} $f:(\cuco X,\mathfrak S)\to(\cuco X',\mathfrak S')$ consists of:
\begin{enumerate}
 \item a map $f:\cuco X\to\cuco X'$;
 \item a map $\pi(f):\mathfrak S\to 2^{\mathfrak S'}$ such that $\pi(f)(U)$ is a collection of pairwise-orthogonal elements of $\mathfrak S'$ for each $U\in\mathfrak S$;
 \item for each $U\in\mathfrak S$, a map $\rho(f,U):\fontact U\to\prod_{V\in \pi(f)(U)}\fontact V$
\end{enumerate}
such that:
\begin{enumerate}[(I)]
 
 \item if $U,V\in\mathfrak S$ satisfy $U\propnest V$, then for each $W'\in\pi(f)(V)$, there exists $W\in\pi(f)(U)$ with $W\propnest W'$, and for every $W\in\pi(f)(U)$ there exists (a unique) $W'\in\pi(f)(V)$ with $W\propnest W'$;
 \item if $U,V\in\mathfrak S$ satisfy $U\orth V$, then $W\orth W'$ for all distinct $W\in \pi(f)(U),W'\in \pi(f)(V)$;
 \item if $U,V\in\mathfrak S$ satisfy $U\transverse V$, then for all $W\in \pi(f)(U)$, there exists $W'\in \pi(f)(V)$ with $W\transverse W'$ and vice versa;
 \item each $\rho(f,U)$ is a (uniform) quasiisometric embedding (where $\Pi_{W\in\pi(f)(U)}\fontact W$);
 \item for all $U\in\mathfrak S$, the following diagram (uniformly) coarsely commutes:
 
\begin{center}
$
\begin{diagram}
\node{\cuco X}\arrow[3]{e,t}{f}\arrow{s,r}{\pi_U}\node{}\node{}\node{\cuco X'}\arrow{s,r}{\Pi_{W\in\pi(f)(U)}\pi_W}\\
\node{\fontact U}\arrow[3]{e,t}{\rho(f,U)}\node{}\node{}\node{\Pi_{W\in\pi(f)(U)}\fontact W}
\end{diagram}
$
\end{center}

\item if $U,V\in\mathfrak S$ satisfy $U\propnest V$ or $U\transverse V$, then

\begin{center}
$
\begin{diagram}
\node{\fontact U}\arrow[3]{e,t}{\rho(f,U)}\arrow{s,r}{\rho^U_V}\node{}\node{}\node{\Pi_{W\in\pi(f)(U)}\fontact W}\arrow{s,r}{g}\\
\node{\fontact V}\arrow[3]{e,t}{\rho(f,V)}\node{}\node{}\node{\Pi_{W'\in\pi(f)(V)}\fontact W'}
\end{diagram}
$
\end{center}
uniformly coarsely commutes, where $g$ is a coarsely constant map so that: if $U\propnest V$, then for each $W'\in\pi(f)(V)$, the $W'$--coordinate of $g$ is $\rho^W_{W'}$ for some (hence any, by Lemma~\ref{lem:orthogonal_close}) $W\in\pi(f)(U)$ with $W\propnest W'$, and if $U\transverse V$, then for each $W'\in\pi(f)(V)$, the $W'$--coordinate of $g$ is $\rho^W_{W'}$ for some (hence any) $W\in\pi(f)(U)$ with $W\transverse W'$;
\item if $V\propnest U$, then 

\begin{center}
$
\begin{diagram}
\node{\fontact U}\arrow[3]{e,t}{\rho(f,U)}\arrow{s,r}{\rho^U_V}\node{}\node{}\node{\Pi_{W\in\pi(f)(U)}\fontact W}\arrow{s,r}{h}\\
\node{\fontact V}\arrow[3]{e,t}{\rho(f,V)}\node{}\node{}\node{\Pi_{W'\in\pi(f)(V)}\fontact W'}
\end{diagram}
$
\end{center}
uniformly coarsely commutes, where the map $h$ is defined as follows: given $(x_{W'})_{W'\in\pi(f)(U)}$, for each $W\in\pi(f)(V)$, the $W$--coordinate of $h((x_{W'}))$ is $\rho^{W''}_W(x_{W''})$, where $W''$ is the unique element of $\pi(f)(U)$ with $W\propnest W''$.
\end{enumerate}

\end{defn}

\begin{rem}[Hieromorphisms are slanted hieromorphisms]\label{rem:hier_is_slant_hier}
Any hieromorphism $f$ is a slanted hieromorphism in which $|\pi(f)(U)|=1$ for all $U\in\mathfrak S$.
\end{rem}

\begin{rem}\label{rem:more_general}
There is presumably a still more general version of Definition~\ref{defn:slanted_hieromorphism} encompassing morphisms $f:(\cuco X,\mathfrak S)\to(\cuco X',\mathfrak S')$ where $f:\cuco X\to\cuco X'$ is a map, $f:2^{\mathfrak S}\to2^{\mathfrak S'}$ sends pairwise-orthogonal sets to pairwise-orthogonal sets, and $f$ sends appropriate products of hyperbolic spaces to products of hyperbolic spaces.  Simple examples like rotation in $\Euclidean^2$ require such a definition in order to be regarded as maps of hierarchically hyperbolic spaces.
\end{rem}

\begin{defn}[Coarse similarity]\label{defn:coarse_similarity}
Let $M,M'$ be metric spaces.  Then $f\colon M\to M'$ is a \emph{$(\lambda,\epsilon)$--coarse similarity} if there exist $\lambda>0,\epsilon\geq0$ such that for all $p,q\in M$, $$\lambda\dist_{M}(p,q)-\epsilon\leq\dist_{M'}(f(p),f(q))\leq\lambda\dist_{M}(p,q)+\epsilon.$$
\end{defn}

\begin{defn}[Extensible slanted hieromorphism]\label{defn:extensible}
Let $f:(\cuco X,\mathfrak S)\to(\cuco X',\mathfrak S')$ be a slanted hieromorphism.  Then $f$ is \emph{extensible} if there exist  $0<\lambda_1\leq\lambda_2$ and $K<\infty$ such that:
\begin{enumerate}
 \item $\pi(f):\mathfrak S\to2^{\mathfrak S'}$ is injective;
 \item for all $V\in\mathfrak S'$, either there is $U\in\mathfrak S$ with $V\in\pi(f)(U)$ or $\diam_{\fontact V}(\pi_V(f(\cuco X)))\leq K$;
 \item for all $U\in\mathfrak S$ and $W\in\pi(f)(U)$, the composition $$\fontact U\stackrel{\rho(f,U)}{\longrightarrow}\prod_{V\in\pi(f)(U)}\fontact V\to \fontact W$$ is a $(\lambda,\lambda')$--coarse similarity, where the second map is the canonical projection and $\lambda\in[\lambda_1,\lambda_2]$ ($\lambda$ can depend on $U,V$) and $\lambda'\geq0$.\label{item:similarity}
\end{enumerate}

\end{defn}

\begin{thm}[Extending slanted hieromorphisms to the boundary]\label{thm:extension to boundary maps}
Let $(\cuco X,\mathfrak S)$ and $(\cuco X',\mathfrak S')$ be hierarchically hyperbolic structures on the spaces $\cuco X,\cuco X'$ respectively.  Suppose that $f\colon(\cuco X,\mathfrak S)\rightarrow(\cuco X',\mathfrak S')$ is an extensible slanted hieromorphism.  Then there is a map $\bar f\colon\overline{\cuco X}\to\overline{\cuco X'}$ such that
\begin{enumerate}
 \item $\bar f|_{\cuco X}=f$;\label{item:extension}
 \item $\bar f|_{\boundary\cuco X}$ is injective;\label{item:boundary_injective}
 \item for all $f(p)\in\boundary\cuco X'$ and basic neighborhoods $f(p)\in\mathcal N$ of $\overline{\cuco X}'$, the set $\bar f^{-1}(\mathcal N)$ contains a basic neighborhood of $p\in\overline{\cuco X}$, i.e., $\bar{f}$ is continuous at each point in $\partial \cuco X$;\label{item:mostly_continuous}
\end{enumerate}
In particular, if $\cuco X$ is proper, then $\bar f|_{\boundary\cuco X}$ is an embedding with closed image and, if $f$ is an embedding, then $\bar f:\overline{\cuco X}\to\overline{\cuco X'}$ is an embedding whose image is closed.
\end{thm}

\begin{proof}
For  convenience, when the domains of the various maps are understood, we shall denote each map $f:\cuco X\to\cuco X',\pi(f):\mathfrak S\to2^{\mathfrak S'}$, and $\rho(f,U):\fontact U\to\Pi_{W\in\pi(f)(U)}\fontact W$ by $f$.

\textbf{Boundary maps on hyperbolic domains:}  Let $U\in\mathfrak S$.  To each sequence $(x_n)$ in $\fontact U$, associate the sequence $(f(x_n))_n$ in $\Pi_{W\in\fontact\pi(f)(U)}\fontact W$.  For each $W\in\pi(f)(U)$, let $w_n(W)\in\fontact W$ be the $W$--coordinate of $f(x_n)$.  Fix a basepoint $x\in\fontact U$ and $p_W = \pi_W(\rho(f,U)(x))\in\fontact W$ for each $W\in\pi(f)(U)$.    

Suppose that $(x_n)_n$ represents a point in $\boundary\fontact U$, i.e. $(x_i|x_j)_x\to\infty$ as $i,j\to\infty$.  Since $\rho(f,U)$ is a uniform quasiisometric embedding, we have for each $W\in\pi(f)(U)$ that $(w_i(W)|w_j(W))_{p_W}\to\infty$ as $i,j\to\infty$.  Hence $w_i(W)$ converges to a point $p(W)\in\boundary\fontact W$.

For each $W\in\pi(f)(U)$, choose $\alpha_W\in(0,1]$ so that $$\frac{\alpha_W}{\alpha_{W'}}=\lim_n\frac{\dist_W(p_W,w_n(W))}{\dist_{W'}(p_{W'},w_n(W'))}$$ for all $W,W' \in \pi(f)(W)$, which exists because of the coarse similarity assumption.  Then define $p\in\star_{W\in\fontact\pi(f)(U)}\partial \fontact W$ to be the linear combination $\sum_{W \in \pi(f)(U)}\alpha_Wp_W$.  The assignment $\bar f_U((x_n))=p$ thus provides a map $\bar f_U:\fontact U\cup\boundary \fontact U\to\Pi_{W\in\fontact\pi(f)(U)}\fontact W\cup\star_{W\in\fontact\pi(f)(U)}\fontact W$ extending the map $\rho(f,U)$.  

For any $U\in\mathfrak S$, the map $\bar f_U$ defined above is injective since the composition of $f$ with any of the canonical projections $\Pi_{W\in\pi(f)(U)}\fontact W\to \fontact W$ is a uniform quasiisometric embedding, and quasiisometric embeddings coarsely preserve Gromov products.  

\textbf{Definition of $\bar f$:}  Let $p\in\boundary\cuco X$, so that $p=\sum_{U\in\support(p)}\beta_Up_U$, where $p_U\in\boundary\fontact U$ for each $U$, each $\beta_U\in(0,1]$, and $\sum_U\beta_U=1$.  For each $U\in\support(p)$, we defined $\bar f_U(p_U)=\sum_{W\in\pi(f)(U)}\alpha^U_Wq_W$ above, where $q_W\in\boundary\fontact W$ and $\sum_W\alpha^U_W=1$.  Let $$\bar f(p)=\sum_{U\in\support(p)}\sum_{W\in\pi(f)(U)}\beta_U\alpha^U_W \cdot q_W,$$
which is a point in $\boundary\cuco X'$ since $\sum_U\sum_W\beta_U\alpha^U_W=1$ and since $\bigcup_{U\in\support(p)}\pi(f)(U)$ is a pairwise-orthogonal set by Definition~\ref{defn:slanted_hieromorphism} since $f$ is a slanted hieromorphism.

\textbf{Injectivity of $\bar f|_{\boundary\cuco X}$:}  Injectivity of $\bar f|_{\boundary\cuco X}$ follows from injectivity of $\bar f_U$ on each $\boundary\fontact U,U\in\mathfrak S$ together with injectivity of $\pi(f)$ and the fact that each $\bar f_U:\fontact U\to\Pi_{W\in\pi(f)(U)}\fontact W$ is ``fully supported'' in the sense that each $\alpha^U_W>0$.

\textbf{Continuity at boundary points:}  First consider $p\in\boundary\cuco X$.  By Proposition~\ref{prop:properties}, there exists $(x_n)$ in $\cuco X$ such that $x_n\to p$ as $n\to\infty$.  We check that $f(x_n)$ converges to $\bar f(p)$.

Fix a basepoint $x\in\cuco X$, so that $p=\sum_{U\in\support(p)}a_Up_U$ where $\sum_Ua_U=1$, each $a_U>0$, and for all $U,U'\in\support(p)$, $$\left|\frac{\dist_U(x,x_n)}{\dist_{U'}(x,x_n)}-\frac{a_U}{a_{U'}}\right|\to0 \text{ and } \frac{\dist_V(x,x_n)}{\dist_U(x,x_n)}\to0$$ whenever $U\in\support(p)$ and $V\in\support(p)^\orth$, and finally $\pi_U(x_n)\to p_U$ for all $U\in\support(p)$.

Consider the sequence $(w_n)=(f(x_n))$.  For each $U\in\support(p)$ and $W\in\pi(f)(U)$, let $c_W:\Pi_{V\in\pi(f)(U)}\overline {\fontact V} \to\overline{\fontact W}$ be the canonical projection.  By hypothesis, for each such $W$ we have $|\dist_W(f(x),w_n)-\lambda_W\dist_U(x,x_n)|\leq\lambda'_W$, where $\lambda_W\in[\lambda_1,\lambda_2]$ and $\lambda'_W\geq0$.  Hence for each $U\in\support(p)$ and $W\in\pi(f)(U)$, we have that $\pi_W(w_n)=c_W\circ\bar f(\pi_U(x_n))\to c_W\circ\bar f(p_U)$ and $\bar f(\pi_U(x_n))\to\sum_{W\in\pi(f)(U)}\beta_U\alpha_Wc_W\cdot \bar f(p_U)$ as required.  Moreover, if $V\in\mathfrak S'$ does not belong to $\pi(f)$, then $\dist_V(f(x),w_n)$ is uniformly bounded by Definition~\ref{defn:extensible}(2).

Finally, if $V\in\mathfrak S-\support(p)$, then $\dist_V(x,x_n)$ is dominated by $\dist_U(x,x_n)$ for any $U\in\support(p)$. Hence, for such $V$, we have that  $\dist_W(f(x),f(x_n))$ is dominated by $\dist_Z(f(x),f(x_n))$ whenever $W\in\pi(f)(V)$ and $Z\in\pi(f)(U)$ for some $U\in\support(p)$, since each $\rho(f,U)$ is a uniform quasiisometric embedding.  Thus $f(x_n)$ converges to $\bar f(p)$.

More generally, given any sequence $(z_k)$ in $\overline{\cuco X}$ converging to $p\in\boundary X$, we can use the ideas in the proof of Theorem \ref{thm:cpt} to build a sequence of internal sequences $(x_{k,i})$, so that $\lim_i x_{k,i} = z_k$ for each $k$.  Namely, for each $k$, we can take a sequence $(x_{k,i}) \rightarrow z_k$ (if $z_k \in \cuco X$, then we choose $x_{k,i}= z_k$ to be constant), and then we choose $N_k>0$ large enough so that if $n>N_k$, then the sequence $(x_{k,n})$ will satisfy conditions \eqref{item:unbounded}--\eqref{internal condition 7} from the proof of Theorem \ref{thm:cpt}.  This will force that $\lim_i x_{k,i} =  z_k$, and then since $\lim_k z_k =  p$, the above conditions will force $\lim_k x_{k,n} = p$. 

Now since $\lim_n x_{k,n}= p$ and $\lim_i x_{k,i}= z_k$, the internal case above implies $\lim_n \bar f(x_{k,n})= \bar f(p)$ and $\lim_i \bar f(x_{k,i}) = z_k$.  Together, these imply that $\lim_k\bar f(z_k)= \bar f(p)$.  Thus $\bar f$ is continuous at boundary points.

\textbf{When $\cuco X$ is proper:} Assertion~\eqref{item:mostly_continuous} combines with Theorem~\ref{thm:cpt} and Proposition~\ref{prop:properties}.\eqref{item:hausdorff} to imply that $\bar f$ is an embedding; compactness of $\boundary\cuco X$ implies that its image is closed. If in addition, $f$ is an embedding, then $\bar f:\overline{\cuco X}\to\overline{\cuco X'}$ is an embedding, since assertion~\eqref{item:mostly_continuous} again combines with Proposition~\ref{prop:properties}.\eqref{item:hausdorff} and Theorem~\ref{thm:cpt} to imply that $\bar f$ is a continuous injection from a compact space to a Hausdorff space.
\end{proof}

\begin{rem}\label{rem:relax_extensible}
Theorem~\ref{thm:extension to boundary maps} holds under slightly more general conditions: condition~\eqref{item:similarity} of Definition~\ref{defn:extensible} need only be imposed on $U\in\mathfrak S$ in cases where either there exists $V\in\mathfrak S$ with $U\orth V$ or $|\pi(f)(U)|>1$ or both.  For any $U$ with empty orthogonal complement and for which $\pi(f)(U)=\{V\}$ for some $V\in\mathfrak S'$, it suffices to require that $\rho(f,U):\fontact U\to\fontact V$ is a uniform quasiisometric embedding.
\end{rem}

\subsection{Limit sets of hierarchically quasiconvex sets}\label{subsubsec:hier_quasiconvex}
Let $(\cuco X,\mathfrak S)$ be a proper hierarchically hyperbolic space and let $\cuco Y\subseteq\cuco X$ be hierarchically quasiconvex.  Let $\Lambda\cuco Y$ be the set of boundary points $p=\sum_{U\in\support(p)}a_Up_U \in \partial \cuco X$ such that for all $U\in\support(p)$, there is a sequence $p_U^n\in\pi_U(\cuco Y)$ converging to $p_U$.  

\begin{prop}[Hierarchically quasiconvex subspaces have limit sets]\label{prop:hierarchically_quasiconvex_set}
$\cuco Y\cup\Lambda\cuco Y$ is a closed subset of $\overline{\cuco X}$, and $\cuco Y$ is dense in $\cuco Y\cup\Lambda\cuco Y$.  Hence $\cuco Y$ has an HHS structure so that $\cuco Y\cup\Lambda\cuco Y=\overline{\cuco Y}$.   
\end{prop}

\begin{proof}
This is a definition chase and an application of Proposition~\ref{prop:properties}.
\end{proof}

\begin{rem}\label{rem:PARANOIA}
When $\pi_U|_{\cuco Y}$ is either surjective or uniformly bounded for each $U$, Theorem~\ref{thm:extension to boundary maps}, together with the HHS structure on $\cuco Y$ inherited from $\cuco X$, implies that $\Lambda\cuco Y$ is homeomorphic to the HHS boundary $\boundary\cuco Y$.  This holds in particular for the main examples of hierarchically quasiconvex subspaces that we use, namely product regions:
\end{rem}

\begin{rem}[Boundaries of standard product regions]\label{rem:boundaries_of_standard_products}
Let $U\in\mathfrak S$, and recall from Section~\ref{sec:product_regions} that there is a quasiisometric embedding $F_U\times E_U\to\cuco X$ coming from the standard hieromorphisms.  By definition, $\boundary F_U$ consists of exactly those $\sum_{V}a_Vp_V\in\boundary\cuco X$ where the support set $\{V\}$ consists entirely of elements of $\mathfrak S_U$, while $\boundary E_U$ consists of linear combinations of the same form, but with each $V\in\mathfrak S_U^\orth$.  In particular, under the map $F_U\times E_U\to\cuco X$, we see that the images of $\boundary (F_U\times\{e_1\}),\boundary (F_U\times\{e_2\})\to\boundary\cuco X$ are identical.  Moreover, the subspace $\boundary F_U\subset\boundary\cuco X$ is closed.  Finally, $\boundary P_U\subset\cuco X$ is a closed subset homeomorphic to $\boundary F_U\star\boundary E_U$, where $\star$ denotes the spherical join.
\end{rem}

\subsection{Geometrically finite subgroups of mapping class groups}

In this subsection, we will show that certain interesting subgroups of mapping class groups have a well-defined limit set in the boundary.  Before doing so, we give a quick sketch of relevant facts about mapping class groups and Teichm\"uller spaces. For more details about the HHG structure of the mapping class group, the reader is referred to \cite[Section~11]{BehrstockHagenSisto:HHS_II}.

Fix a finite type surface $S$. The mapping class group $\MCG(S)$ of $S$ acts properly and cocompactly on the marking graph $\mathcal M(S)$ of $S$ \cite{MasurMinsky:II}. The vertices of the marking graph, called markings, are isotopy classes of certain collections of curves on $S$ (pants decomposition together with certain transverse curves). $\MCG(S)$ and $\mathcal M(S)$ are quasiisometric via the orbit map, and we will identify $\MCG(S)$ with an orbit in $\mathcal M(S)$ from now on. The mapping class group can be given a hierarchically hyperbolic structure by considering the collection $\mathfrak S$ of all its (isotopy classes of essential) subsurfaces and associating to each $Y\in\mathfrak S$ its curve graph $\fontact Y$, a graph whose vertices are isotopy classes of essential simple closed curves on $Y$, except when $Y$ is an annulus (a case that will be more subtle to deal with later, and which we will hence explain in more detail here). When $Y$ is an annulus, $\fontact Y$ has vertices the isotopy classes of arcs connecting the two boundary components, and two such vertices are adjacent if they can be represented by disjoint arcs. The maps $\pi_Y:\MCG(S)\to 2^{\fontact Y}$ are called subsurface projections and, when $Y$ is not an annulus, they are defined more or less by intersecting the curves in the marking with $Y$. When $Y$ is an annulus $\pi_Y$ is defined in the following way. Let $\hat Y$ be the annular cover of $S$ where the core of the annulus lifts to a simple closed curve. There is a natural compactification $\overline Y$ of $\hat Y$ which is a closed annulus, and that can be identified with $Y$. Given a marking $m$, lift to $\hat Y$ all the curves in $m$, except possibly the (only) one which is isotopic to the core of $Y$. Each such lift can be compactified to an arc in $\overline Y$, and we can finally define $\pi_Y(m)$ to be the collection of all such arcs that connect distinct boundary components of $\overline Y$.

We now comment briefly on Teichm\"uller space $\TT(S)$ endowed with the Teichm\"uller metric. A point on Teichm\"uller space corresponds to a hyperbolic metric on $S$, and we can hence consider the systole map $\mathrm{Sys}:\TT(S)\to 2^{\fontact S}$ that maps points in Teichm\"uller space to the shortest curves in the corresponding hyperbolic metric.  The set of systoles is non-empty and pairwise disjoint, thus giving a bounded subset of $\fontact S$.

\subsubsection{Subsurface mapping class groups}

For any nonpants subsurface $Y \subset S$ there is a natural embedding $\iota_Y:\MCG(Y) \hookrightarrow \MCG(S)$ which takes any mapping class $f_Y\in \MCG(Y)$ to a mapping class $f \in \MCG(S)$ so that $f|_Y \equiv Y$ and $f|_{S \setminus Y} \equiv id_{S \setminus Y}$; if $Y$ is an annulus, we take $\MCG(Y)$ to be the cyclic subgroup generated by the Dehn (half) twist about the core of $Y$.

We can also see this map in terms of markings: For each component $X \subset S \setminus Y$ (including annuli with core curves in $\partial Y$), fix a marking $\mu_X \in \mathcal M(X)$; if $X$ is an annulus, then $\mu_X \in \fontact X$.  Define a map $\iota_Y: \mathcal M(Y) \rightarrow \mathcal M(S)$ by 
$$\iota_Y(\mu_Y) = \mu_Y \sqcup \coprod_{\alpha \in \partial Y} \alpha \sqcup \coprod_{X \in S \setminus Y} \mu_X$$
for any marking $\mu_Y \in \mathcal M(Y)$.

The map $\iota_Y$ extends to a hieromorphism in the obvious way and it follows from the distance formula that it is a quasiisometric embedding.  Moreover, since $\diam_Z (\iota_Y(\mathcal M(Y)))$ is uniformly bounded for each $Z \in \mathfrak S \setminus \mathfrak S_Y$ and $\iota_Y$ is surjective for each $W \in \mathfrak S_Y$, it is easy to see that $\iota_Y(\mathcal M(Y))$ is a hierarchically quasiconvex subspace of $\mathcal M(S)$.  Hence we have by Proposition \ref{prop:hierarchically_quasiconvex_set}:

\begin{thm}\label{thm:subsurface boundary embedding}
For any nonpants subsurface $Y \subset S$, the natural inclusion $\iota_Y: \MCG(Y) \hookrightarrow \MCG(S)$ equivariantly extends to a continuous embedding $\partial \iota_Y: \partial \MCG(Y)\hookrightarrow \partial \MCG(S)$.
\end{thm}

\subsection{Convex cocompactness subgroups}
%
%
%
%
%

Convex cocompact subgroups of mapping class groups are a much-studied class of hyperbolic subgroups of mapping class groups, mainly because they are precisely the class of subgroups of $\MCG(S)$ whose corresponding surface subgroup extensions are hyperbolic.  Importantly, they satisfy several strong equivalent characterizations, which we state in the following theorem-definition with parts due variously to Farb-Mosher \cite{FarbMosher}, Hamenst\"adt \cite{HamCC}, Kent-Leininger \cite{KentLein}, and the first author with Taylor \cite{DT15}: 

\begin{thm}\label{thm:convex cocompact}
A subgroup $H < \MCG(S)$ is convex cocompact if it satisfies any of the following equivalent conditions:
\begin{enumerate}
\item Any orbit of $H$ in $\TT(S)$ is quasiconvex;
\item Any orbit of $H$ in $\fontact S$ is quasiisometrically embedded;
\item Any orbit of $H$ in $\mathcal M(S)$ is quasiisometrically embedded and has uniformly bounded subsurface projections;
\item $H$ is a stable subgroup of $\MCG(S)$;
\item The corresponding extension $\Gamma_H$ of $\pi_1(S)$ is Gromov hyperbolic.
\end{enumerate}
\end{thm}

The following is a  corollary of Proposition \ref{prop:hierarchically_quasiconvex_set} and Theorems \ref{thm:hyperbolic} and \ref{thm:convex cocompact}:

\begin{cor}
If $H< \MCG(S)$ is a convex cocompact subgroup of $\MCG(S)$, then the inclusion map $H \hookrightarrow \MCG(S)$ $H$-equivariantly extends to a continuous embedding $\partial_{Gr} H \hookrightarrow \partial \MCG(S)$.
\end{cor}

\begin{proof}
It follows immediately from properties (2) and (3) of Theorem \ref{thm:convex cocompact} that $H$ is a hierarchically quasiconvex subgroup of $\MCG(S)$.  Since $H$ is hyperbolic, Theorem \ref{thm:hyperbolic} implies that the boundary of the induced HHS structure on $H$ inside of $\MCG(S)$ is homeomorphic to $\partial_{Gr} H$.  The result then follows from Proposition \ref{prop:hierarchically_quasiconvex_set}.  
\end{proof}

In the rest of the section, we will consider finitely generated Veech subgroups and the Leininger-Reid combination subgroups of $\MCG(S)$, which are generally not hierarchically quasiconvex.  Recall that for both classes of groups, their actions on $\TT(S)$ do not extend continuously everywhere to embeddings of their boundaries into $\PML(S)$.  The main goal of the remainder of this section is to prove that such an extension does exist for both classes of groups into $\partial \MCG(S)$.

\subsubsection{Veech subgroups} \label{subsec:veech}
The construction of Veech and Leininger-Reid subgroups involves holomorphic quadratic differentials. We will not work with them directly, so we do not need to define them, but we will rather work with the $q$--metric associated to a holomorphic quadratic differential $q$ on the surface $S$. This is a singular flat metric on $S$ which is locally isometric to $\mathbb R^2$ except at finitely many points called \emph{singularities}.

Given a holomorphic quadratic differential $q$ on $S$, there exists a convex subset $TD(q) \subset \TT(S)$ with $TD(q) \cong \mathbb{H}^2$ called a \emph{Teichm\"uller disk}.  Let $\mathrm{Aff}^+(q)$ denote the affine group of $q$.  Following ~\cite{leininger2006combination}, we call any subgroup $G(q) \leq \mathrm{Aff}^+(q) \leq \MCG(S)$, with $G(q)$ acting properly on $TD(q)$, a \emph{Veech subgroup}, except that we will also ask that $G(q)$ be finitely generated.  Veech subgroups have the property that every element of $G(q)$ is either pseudo-Anosov or a multitwist about some annular decomposition $A$ of $q$~\cite{Veech}, where this annular decomposition comes from a finite measured foliation with only closed leaves naturally associated to $q$.

Consider the Veech subgroup $G = G(q) \leq \MCG(S)$.  Let $\cuco X_G$ be the orbit of $G$ of a fixed marking $\mu$ in the marking graph $\mathcal M(S)$. Given a multitwist $g \in G$ with annular decomposition $A_g=\{\alpha_1,\ldots,\alpha_{n_g}\}$, let $$\pi_g:\cuco X_G\rightarrow \prod_{1 \leq i \leq n_g} \fontact \alpha_i$$ be given by $\pi_g(\nu) = (\pi_{\alpha_1}(\nu), \dots, \pi_{\alpha_{n_g}}(\nu))$ for $\nu \in \cuco X_G$.   If $g = T^{k_1}_{\alpha_1} \cdots T^{k_{n_g}}_{\alpha_{n_g}}$, let $$L_g  = \langle g \rangle \cdot \pi_g  (\mu)  \subset  \prod_{1 \leq i \leq n_g} \fontact \alpha_i.$$  Note that $L_g\cong \reals$, and in fact $L_g$ is the projection of the $g$-orbit of $\mu$ and thus coarsely the line in $\reals^{n_g}$ with slope $(k_1, \dots k_{n_g})$, where we identify the origin of $\reals^{n_g}$ with the projection of $\mu$.  For each $L_g$, let $\pi_{L_g}:\prod_{1 \leq i \leq n_g} \fontact \alpha_i \rightarrow L_g$ be the standard projection onto $L_g$, considered as a subspace of $\reals^{n_g}$ identified as above.

We now define an HHS structure $(G, \mathfrak{S}_G)$ on $G$ as follows:

\begin{enumerate}
\item[(Domains)]: $S$ is the unique nest-maximal domain in $\mathfrak{S}_G$, and for every primitive multitwist $g \in G$ with corresponding annular decomposition $A_g = \{\alpha_{g,1}, \dots \alpha_{g, n_g}\}$, we include a domain $U_g \in \mathfrak{S}_G$.
\item[(The spaces)]: To $S$, we associate $\pi_S(G\cdot\mu)\subset\fontact S$ and to each $U_g$, we set $\fontact U_g = L_g$ and declare $U_g \nest S$ for each $g$; moreover, we specify that $U_g \transverse U_{g'}$ for each primitive $g \neq g'$.
\item[(Projections)]: $\pi_S: \cuco X_G \rightarrow \fontact S$ is the standard projection; for each $U_g$, we define $\pi_{U_g}: \cuco X_G \rightarrow L_g$ by $\pi_{U_g}(\nu) = \pi_{L_g}(\pi_{g}(\nu))$ for each $\nu \in \cuco X_G$.
\item[(Relative projections)]: Given $U, V \in \mathfrak S_G$, we define $\rho^U_V: \fontact U \rightarrow \fontact V$ by:
\begin{enumerate}
\item[($U \nest V$)]: In this case $V = S$ and $U = U_g$ for some primitive $g$, then $\rho^V_U = \pi_{L_g} \circ \pi_g$.
\item[($U \transverse V$)]: If $U=U_g$ and $V = U_{g'}$, then 

$$\rho^{U_g}_{U_{g'}} = \pi_{U_{g'}}(\langle g \rangle \cdot \mu).$$
\end{enumerate}
\end{enumerate}

\begin{lem}\label{lem:Veech is an HHS}
If $G$ is finitely generated, then $(G, \mathfrak{S}_G)$ is an HHS structure on $G$, and $G < \Aut(\mathfrak{S}_G)$.
\end{lem}

\begin{proof}
We need to prove that $(G, \mathfrak{S}_G)$ satisfies the axioms; since it clearly satisfies projections, nesting, orthogonality, and finite complexity, it suffices to prove it satisfies the consistency, large link, bounded geodesic image, partial realization, and uniqueness axioms.  Hyperbolicity of the associated spaces uses Lemma~\ref{lem:veech_quasiconvex} (the only part for which we need finite generation of $G$).

There is no nontrivial orthogonality, so partial realization holds by construction.  Bounded geodesic image holds by the bounded geodesic image axiom in $(\MCG(S),\mathfrak S)$ and the definition of $\rho^S_{U_g}$.  The consistency and large link axioms hold for a similar reason.  Uniqueness follows from uniqueness in $(\MCG(S),\mathfrak S)$ together with Lemma~\ref{lem:veech_bounded}.   
\end{proof}

\begin{lem}\label{lem:veech_quasiconvex}
The projection $\pi_S(G \cdot \mu)$ is quasiconvex in $\fontact S$.
\end{lem}

\begin{proof}
Consider the action of $G$ on the corresponding Teichm\"uller disk $TD(q)$. Since the action is proper, this makes $G$ a finitely generated Fuchsian group. Hence, $G$ is geometrically finite \cite{Marden:Fuchsian_geom_finite}, so that it acts with cofinite volume on a convex subspace $C_G\subseteq TD(q)$. Consider now the image of $C_G$ and $TD(q)$ in $\fontact S$. Since geodesics in $\TT(S)$ map to quasi-geodesics in $\fontact S$ \cite{MasurMinsky:I} and $C_G$ is a convex subspace of $\TT(S)$, it follows that $\pi_S(C_G)$ is quasiconvex in $\fontact S$.

Now, it is not hard to see that $\pi_S(C_G)$ coarsely coincides with $\pi_S(G \cdot \mu)$. In fact, $C_G$ contains a $G$--equivariant collection of horodisks so that the action on the complement $C'_G$ is cocompact, and cocompactness implies that $\pi_S(G \cdot \mu)$ coarsely coincides with the image in $\fontact S$ of $C'_G$. Moreover, each horodisk is stabilized by a multitwist, and the corresponding curves are short in all hyperbolic metrics corresponding to points in the horodisk. This implies that the whole horodisk maps to a uniformly bounded subset of $\fontact S$ under the systole map, namely a neighborhood of the aforementioned curves. To sum up, the projection of the Teichm\"uller disk to $\fontact S$ is quasiconvex and coarsely coincides with the projection of $C'_G$, which in turn coarsely coincides with the projection of $G \cdot \mu$, and we are done. 
\end{proof}

\begin{lem}\label{lem:veech_bounded}
There exists $V>0$ such that for any $U \in \mathfrak S-\{S\}$, either $\diam_{U}(\pi_U(G \cdot \mu)) \leq V$ or $U = \alpha_{i} \in A_g$ for some annular decomposition $A_g$.  In the latter case, $\pi_{U}$ is (uniformly) coarsely surjective.
\end{lem}

\begin{proof}
Let $U\propnest S$ be a subsurface and let $\Delta\subset U$ be its spine, which is obtained by puling tight $\partial U$ with respect to the $q$-metric, so that vertices of $\Delta$ are singular points and edges are saddle connections (i.e. geodesics connecting singularities and intersecting the singular set only at the endpoints).  There exists a natural retraction $r:U\to\Delta$ and for each edge $e$ of $\Delta$, let $\delta_e=r^{-1}(m_e)$, where $m_e$ is the midpoint of $e$.  Each $\delta_e$ is either a curve or an arc in $(U,\boundary U)$.  We now divide into three cases.

\textbf{$U$ is non-annular:}  In this case, $\Delta$ has a degree--$3$ vertex $v$.  Suppose that $\mu$ has a base curve $\alpha$ that traverses each saddle connection in $\Delta$ at most once.  Then $v$ has some incident edge $e$ so that $\delta_e$ is disjoint from $\alpha$.  Now, for any $g\in\mathrm{Aff}^+(q)$, we have that $g\cdot \Delta$ is the spine of $g\cdot U$, with vertices that are singular points and edges saddle connections.  In particular, $g\cdot \alpha$ is a curve using each saddle connection of $\Delta$ at most once, so $\dist_{\mathcal{AC}U}(\alpha,g\cdot \alpha)\leq 3$, where $\mathcal{AC}U$ denotes the arc-and-curve graph of $U$.  Since there is a 2-Lipschitz retraction $\mathcal{AC}(U)\to\fontact U$ \cite{MasurMinsky:II}[Lemma 2.2], it follows that $\diam_U(G\cdot\mu)$ is uniformly bounded.

Since $G(q)$ preserves the set of all singularities, saddle connections, and geodesic representatives of curves, we are done provided we choose the marking $\mu$ in such a way that each of its base curves traverses each saddle connection at most once.

\textbf{$U\in A_g$ for some $g$:}  Let $g\in G(q)$ be a multitwist about curves $\alpha_1,\ldots,\alpha_n$, so that $g=\prod_{i=1}^nT_{\alpha_i}^{k_i}$, where $k_i\in\integers-\{0\}$.  Hence $\pi_U$ is $k_i$--surjective (where $U=\alpha_i$).  Indeed, $\pi_U(g\cdot\mu)=\pi_U(T_{\alpha_i}^{k_i}\cdot\mu)$, and the $k_i$ are uniformly bounded since the action of $G(q)$ on the corresponding Teichm\"uller disc is geometrically finite, and thus there are finitely many conjugacy classes of multitwists in $G(q)$; see the proof of Lemma~\ref{lem:veech_quasiconvex}. 

\textbf{$U$ an annulus and $U \notin A_g$ for any $g$:}  The spine $\Delta$ of $U$ contains at least one singularity, and the angle at the singularity is greater than $\pi$ on both sides. Let $\widehat U$ be the annular cover of $S$ corresponding to $U$. The lift $\widehat \Delta$ of $\Delta$ disconnects $\widehat U$ into two connected components, and we will refer to the closure of each such connected component as a \emph{side} of $\widehat \Delta$. Consider a singularity along $\widehat \Delta$ and a saddle connection entering the singularity. Then, for any side of $\widehat \Delta$ there exists a unique geodesic ray emanating from the given singularity, forming an angle of $\pi$ with the given saddle connection and contained in the given side of $\widehat\Delta$. We let $\{\alpha_i\}$ be the open arcs in $\widehat U$ that can be formed by concatenating two such rays lying in opposite sides of $\widehat \Delta$. It is readily seen that any two $\alpha_i$ have intersection number at most $1$. The bound on the diameter of the projection onto $\fontact U$ now follows from the fact that any arc in the subsurface projection onto $\fontact U$ of some curve in $S$ can be represented either by a geodesic transverse to a saddle connection in $\widehat\Delta$, which is easily seen to be disjoint from some $\alpha_i$, or a geodesic containing one of the singularities, which is easily seen to intersect an appropriate $\alpha_i$ containing that singularity at most once.
\end{proof}

\begin{lem}\label{lem:veech_slanted_h}
There exists a $G$--equivariant extensible slanted hieromorphism $(G,\mathfrak S_G)\to(\MCG(S),\mathfrak S)$.
\end{lem}

\begin{proof}
At the level of spaces, the map $G\to\MCG(S)$ is the inclusion.  Define $\pi(f):\mathfrak S_G\to 2^{\mathfrak S}$ as follows: let $\pi(f)(S)=\{S\}$, and for each primitive multitwist $g$, let $\pi(f)(U_g)=A_g$, where $A_g$ is the set of pairwise-disjoint annuli corresponding to the multicurve supporting $g$.  This is $G$--equivariant since $hA_g=A_{hgh^{-1}}$ for each multitwist $g$ and each $h\in G$.

The map $\rho(f,S):\fontact S\to\fontact S$ is the identity.  For each primitive multitwist $g=T^{k_1}_{\alpha_1} \cdots T^{k_{n_g}}_{\alpha_{n_g}}$, the map $\rho(f,U):L_g\to\prod_i\fontact \alpha_i$ was specified above.  Observe that the composition of this map with any of the canonical projections to $\fontact\alpha_i$ is a coarse similarity with multiplicative constants determined by $\{k_1,\ldots,k_{n_g}\}$.  These constants are uniformly bounded since there are finitely many conjugacy classes of multitwists in $G(q)$.
\end{proof}

Combining Lemma~\ref{lem:veech_slanted_h} and Theorem~\ref{thm:extension to boundary maps}, Remark~\ref{rem:relax_extensible}, and Theorem \ref{thm:hyperbolic} yields:

\begin{cor}\label{cor:veech_limit_set}
For any Veech subgroup $G<\MCG(S)$, the inclusion $G\to\MCG(S)$ extends continuously to an equivariant embedding $\boundary_{Gr} G\to\boundary\MCG(S)$ with closed image.
\end{cor}

\begin{rem}\label{rem:non_hqc_veech}
Corollary~\ref{cor:veech_limit_set} does not follow from Proposition~\ref{prop:hierarchically_quasiconvex_set} because the Veech subgroup $G$ is not hierarchically quasiconvex in $\MCG(S)$ whenever it contains a multitwist supported on a multicurve with more than one component; indeed, in this case there are realization points in $\MCG(S)$ whose images in each curve graph lie in the image of $G$, but which are arbitrarily far from $G$.
\end{rem}

\subsubsection{Leininger-Reid surface subgroups}
We now turn to the Leininger-Reid surface subgroups constructed in~\cite[Theorem~6.1]{leininger2006combination}.  Again, we show that these are non-hierarchically quasiconvex subgroups of $\MCG(S)$ that nonetheless have well-defined limit sets in $\boundary\MCG(S)$.  The setup is as follows: 
\begin{enumerate}
 \item Let $q_1,\ldots,q_n$ be holomorphic quadratic differentials, with $A_0\in\fontact S$ the core of the annular decomposition of each $q_i$ such that each complementary component has negative Euler characteristic;
 \item Suppose $G_0=G_0(q_i)$ for all $i\leq n$;
 \item Suppose $h\in\MCG(S)$ centralizes $G_0$ and is pure and pseudo-Anosov on all components of $S-A_0$.
\end{enumerate}
Then, for $$H=G(q_1)*_{G_0}h^{k_2}G(q_2)h^{-k_2}*_{G_0}\cdots *_{G_0}h^{k_n}G(q_n)h^{-k_n},$$ the map $H\to\MCG(S)$ is an embedding, whenever $N=\min\{|k_i-k_j|:i,j\in\{1,\dots,n\}, i\neq j\}$ (where we set $k_1=0$) is large enough.  Moreover, every element of $\image(H\to\MCG(S))$ (which we denote by $H$) is either pseudo-Anosov or conjugate into an elliptic or parabolic subgroup of some $h^{k_i}G(q_i)h^{-k_i}$.  In particular, the $G(q_i)$ can be chosen so that $H$ fails to be hierarchically quasiconvex for the reason explained in Remark~\ref{rem:non_hqc_veech}.

In the remainder of this section, we prove:

\begin{thm}\label{thm:leininger_reid_limit_set}
The inclusion $H\to\MCG(S)$ extends continuously to an equivariant embedding $\boundary H\to\boundary\MCG(S)$ with closed image.
\end{thm}

\begin{proof}
This follows from Theorem~\ref{thm:extension to boundary maps}, Remark~\ref{rem:relax_extensible}, and Proposition~\ref{prop:leininger_reid_slanted} below.
\end{proof}

Our goal is now to state and prove Proposition \ref{prop:leininger_reid_slanted}, which says that the inclusion of $H$ into $\MCG(S)$ is a slanted hieromorphism. We need control over various projections, which we achieve in the following preliminary lemmas.

\begin{lem}\label{lem:conj_are_quasiconvex}
 There exists a constant $Q$ so that, for any $i$ and any $k$, $\pi_S(h^{k}G(q_i)h^{-k})$ is $Q$--quasiconvex.
\end{lem}

\begin{proof}
Apply quasiconvexity of the $\pi_S(G(q_i))$ and boundedness of $\{\pi_S(1,h^k)\}_{k\in\mathbb Z}$.
\end{proof}

Denote by $\mathcal Y$ the set of connected components in $S$ of the complement of the annuli in the annular decomposition of the multitwists in $G_0$.

\begin{lem}\label{lem:proj_1_and_y0}
 There exists $K$ so that for any $Y$ transverse to some $Y_0\in\mathcal Y$ we have $\dist_Y(\rho^{Y_0}_Y,1)\leq K$.
\end{lem}

\begin{proof}
 This is because $\rho^{Y_0}_Y$ coarsely coincides with $\pi_Y(P_{Y_0})$, and the fact that $\pi_Y$ is coarsely Lipschitz (note that there are finitely many $Y_0$).
\end{proof}

\begin{lem}\label{lem:translates_overlap}
 For each $g\in G(q_i)- G_0$ for some $i$ and each $Y\in\mathcal Y$, there exists $Y'\in\mathcal Y$ so that $g\cdot Y'$ is transverse to $Y$.
\end{lem}

\begin{proof}
This is a restatement of \cite[Lemma 4.1]{leininger2006combination}.
\end{proof}

\begin{lem}\label{lem:proj_on_complement}
 There exists $C, M$ with the following property. For any $g=g_1h^{m_1}\dots g_kh^{m_k}$ with $g_i\in G(q_{j(i)})- G_0$ and $|m_i|\geq M$ for each $i\leq k$, we have $\dist_{Y_0}(1,g)\leq C$ for each $Y_0\in\mathcal Y$.
\end{lem}

\begin{proof}
Let $K$ be as in Lemma \ref{lem:proj_1_and_y0}. Proceed by induction on $k$, for $C$ to be determined. If $k=0$, there is nothing to prove.

 Suppose $k\geq 1$. Fix $Y_0\in\mathcal Y$ and let $Y=g_1 Y'$ with $Y'\in\mathcal Y$ chosen via Lemma \ref{lem:translates_overlap} so that $Y'\transverse Y_0$. By induction, $\dist_{Y}(g_1h^{m_1},g)= \dist_{Y'}(1,g_2h^{m_2}\dots g_kh^{m_k})\leq C$, since $hY=Y$ for any $Y\in\mathcal Y$ by hypothesis so that $g_1h^{m_1}\cdot Y'=g_1\cdot Y'=Y$.
 
 By Lemma \ref{lem:veech_bounded}, $\dist_{Y}(1,g_1)$ is uniformly bounded by some $V$. Hence $\dist_{Y}(1,g)\geq \dist_{Y}(g_1,g_1h^{m_1})-C-V=\dist_{Y'}(1,h^{m_1})-C-V$. If $|m_1|$ is large enough, then this quantity is larger than $K+10E$. Since $Y_0\transverse Y$, consistency implies that we have $\dist_{Y_0}(\rho^{Y}_{Y_0},g)\leq E$. Also,
 $$\dist_{Y_0}(\rho^{Y}_{Y_0},1) \leq \dist_{Y_0}(\rho^{Y}_{Y_0},g_1)+V=\dist_{g_1^{-1}Y_0}(\rho^{Y'}_{g_1^{-1}Y_0},1)+V\leq V+K,$$
 hence $\dist_{Y_0}(1,g)\leq 2E+V+K$. Thus we set $C=2E+V+K$, which determines $M$.
\end{proof}

\begin{prop}\label{prop:leininger_reid_slanted}
The subgroup $H\leq\MCG(S)$ admits a hierarchically hyperbolic space structure $(H,\mathfrak S_H)$ so that there is an extensible slanted hieromorphism $(H,\mathfrak S_H)\to(\MCG(S),\mathfrak S)$ induced by the inclusion $H\hookrightarrow\MCG(S)$. 
\end{prop}

\begin{proof}
We follow a very similar procedure to that used for individual Veech subgroups.  In particular, $\mathfrak S_H$ is defined exactly as $\mathfrak S_G$ was, except that there is now a domain $U_g$ for each primitive multitwist in $H$.  To verify that this yields an HHS structure, we must check that:
\begin{enumerate}
 \item $\pi_S(H)$ is quasiconvex.  \label{item:L_R_small_1}
 \item $\pi_{U}(H)$ is uniformly bounded unless $U\in A_g$ for some $g\in H$.\label{item:L_R_small_2}
\end{enumerate}

Once the properties above are proven, arguing exactly as in the proof of Lemma~\ref{lem:Veech is an HHS} and Lemma~\ref{lem:veech_slanted_h} yields the desired slanted hieromorphism and completes the proof.

We now set conventions and notations that we use throughout the proof. When some $g=g_1\dots g_k\in H$ with $g_i\in h^{k_{j(i)}}G(q_{j(i)})h^{-k_{j(i)}}- G_0$ is any fixed element of $H$, we denote $p_l=\pi_S(g_1\dots g_l)$ (with $p_0=\pi_S(1)$), and let $\gamma_l$ be a geodesic in $\fontact S$ from $p_{l-1}$ to $p_l$, so that the concatenation of the $\gamma_l$ is a path from $\pi_S(1)$ to $\pi_S(g)$. Furthermore, notice that we can write $g=h^{m_0}g'_1h^{m_1}\dots g'_kh^{m_k}$ for some $g'_i\in G(q_{j(i)})- G_0$ (more specifically, $g'_i=h^{-k_{j(i)}}g_ih^{k_{j(i)}}$), and that $|m_l|$ for $l<k$ is bounded below by $N$ (recall that this is the minimal value of $|k_i-k_j|$ for $i\neq j$). We set $h_l=h^{m_0}g'_1h^{m_1}\dots g'_l$.

In the following claim, we study geodesics connecting $\pi_S(1)$ to $\pi_S(g)$ for arbitrary $g\in G$. The claim easily implies that geodesics from $\pi_S(1)$ to $\pi_S(g)$ stay close to $\pi_S(H)$ for any $g\in H$ because each $\gamma_l$ is contained in a coset of some $h^{k_{j(i)}}G(q_{j(i)})h^{-k_{j(i)}}$ and such cosets are uniformly quasiconvex byLemma \ref{lem:conj_are_quasiconvex}. Hence, the claim proves that $\pi_S(H)$ is quasiconvex, which is item \ref{item:L_R_small_1} above. 

\begin{claim}\label{claim:concat_close_to_geod}
 There exists a constant $R$ with the following property. For any $g\in H$, the Hausdorff distance between $\bigcup_l \gamma_l$ and $[\pi_S(1),\pi_S(g)]$ is bounded by $R$, where $[\pi_S(1),\pi_S(g)]$ is any geodesic in $\fontact S$ from $\pi_S(1)$ to $\pi_S(g)$.  Moreover, for any $Y\in\mathcal Y$ we have that $\dist_{h_lY}(1,h_l),\dist_{h_lY}(g,h_lh^{m_l})\leq C$.
\end{claim}

\renewcommand{\qedsymbol}{$\blacksquare$}
\begin{proof}
We first show $\bigcup_l \gamma_l$ is uniformly close to $[\pi_S(1),\pi_S(g)]$.

It suffices to show that the endpoints of all $\gamma_l$ lie within controlled distance of $[\pi_S(1),\pi_S(g)]$. Any such endpoint $x$ coarsely coincides with both $\pi_S(h_l)$ and $\pi_S(h_lh^{m_l})$, for some $l$ (since $\{\pi_S(h^m)\}_{m\in\mathbb Z}$ is a bounded set). Pick any $Y\in \mathcal Y$, and set $Z=h_l \cdot Y$. By Lemma \ref{lem:proj_on_complement} we have $\dist_Z(h_lh^{m_l},g)\leq C$ and $\dist_Z(1,h_l)\leq C$. Hence, if $m_l$ is large enough, we get $\dist_Z(1,g)\geq \dist_Y(1,h^{m_l})-2C\geq 100E$. Notice that by bounded geodesic image $\rho^Z_S$ needs to be within $10E$ of geodesics from $\pi_S(h_l)$ and $\pi_S(h_lh^{m_l})$, which both coarsely coincide with the endpoint $x$ we are interested in. If geodesics from $\pi_S(1)$ to $\pi_S(g)$ did not pass close to $x$ we could then conclude that they do not pass close to $\rho^Z_S$, which would imply by bounded geodesic image that $\dist_Z(1,g)\leq 5E$. But this is not the case, and hence we get a bound on the distance from $x$ to $[\pi_S(1),\pi_S(g)]$, as required.

Let us now prove that points on $[\pi_S(1),\pi_S(g)]$ are close to $\bigcup_l \gamma_l$. Suppose by contradiction that there exists $x\in[\pi_S(1),\pi_S(g)]$ with $\dist_S(x,\bigcup_l \gamma_l)\geq 2C+1$. Let $x_1,x_2\in [\pi_S(1),\pi_S(g)]$ lie on distinct sides of $x$ (in the natural order of $[\pi_S(1),\pi_S(g)]$), with $x_1$ closer to $\pi_S(1)$ than $x$, and satisfy $\dist_S(x_i,x)=C+1$. Then any $y\in \bigcup \gamma_l$ lies in $\neb_C([\pi_S(1),x_1])\cup \neb_C([x_2,\pi_S(g)])$. However, the two neighborhoods are disjoint and the connected set $\bigcup \gamma_l$ contains points in both, a contradiction.
\end{proof}

\renewcommand{\qedsymbol}{$\Box$}

Let us now take $U\in\mathfrak S-\{S\}$ and $g\in H$ with $\dist_U(1,g)\geq 100E$. We need to show that either $U$ belongs to some $A_{g'}$ or $\dist_U(1,g)$ is bounded independently of $U,g$.

We proved in the claim that, for any $Y\in\mathcal Y$, the projections of $1$ and $g$ on $h_l\cdot Y$ coarsely coincide with the projections of $h_l$ and $h_lh^{m_l}$ respectively, and hence that $d_{h_l\cdot Y}(1,g)>100E$ if $|m_l|\geq N$ is large enough. Since $m_l$ can take finitely many values, we therefore get the desired bound whenever $U$ is of the form $h_l\cdot Y$. We now assume that $U$ is neither belongs to some $A_{g'}$ nor it is of the form $h_l\cdot Y$. Hence, for any $l$ there exists $Y$ so that $h_l\cdot Y \transverse U$ overlap, and hence are comparable in the partial order $\preceq$; see Proposition 2.8 of ~\cite{BehrstockHagenSisto:HHS_II}.

Another fact about $\preceq$ is that whenever $Y,Y'\in \mathcal Y$ and $l$ are so that $h_l\cdot Y \transverse h_{l+1}\cdot Y'$, we have $h_l\cdot Y\preceq h_{l+1}\cdot Y'$, again provided $|m_l|\geq N$ is large enough. In fact, $\rho^{h_{l+1}Y'}_{h_lY}=h_{l+1}\rho^Y_{h_{l+1}^{-1}h_lY'}$ coarsely coincides with $\pi_{h_l\cdot Y}(h_{l+1})$ (Lemma \ref{lem:proj_1_and_y0}), which in turn coarsely coincides with $\pi_{h_l\cdot Y}(h_lh^{m_l})$ by Lemma \ref{lem:veech_bounded} since $h_{l+1}=h_lh^{m_l}g'_{l+1}$. Finally, $\pi_{h_l\cdot Y}(h_lh^{m_l})$ coarsely coincides with $\pi_{h_l\cdot Y}(g)$ by what we said above.

By looking at a predecessor and a successor of $U$, we then see that the projections of $1,g$ onto $U$ coarsely coincide with those of $h_l\cdot Y,h_{l+1}\cdot Y'$ for some $l$ and $Y,Y'$. But these latter projections coarsely coincide with those of $h_l$ and $h_lh^{m_l}g'_{l+1}$. The projections of $h_l$ and $h_lh^{m_l}$ are uniformly close by boundedness of $m_l$, while the projections of $h_lh^{m_l}$ and $h_lh^{m_l}g'_{l+1}$ are uniformly close by Lemma \ref{lem:veech_bounded}. This concludes the proof.
\end{proof}

\section{Automorphisms of HHS and their actions on the boundary}\label{subsubsec:auts_and_HHG}
The most important special case of an extensible hieromorphism is an automorphism of $(\cuco X,\mathfrak S)$.  For any automorphism $f:(\cuco X, \mathfrak S) \rightarrow (\cuco X, \mathfrak S)$, each isometry $f: \fontact U \rightarrow \fontact (f(U))$ extends to a homeomorphism $\hat{f}:\partial \fontact U \rightarrow \partial \fontact (f(U))$, yielding an application of Theorem \ref{thm:extension to boundary maps}:

\begin{cor}[Extensions of automorphisms to the boundary] \label{cor:auto extend}
Any $f\in\Aut(\mathfrak S)$ extends to a bijection $\overline{\cuco X}\to\overline{\cuco X}$ which restricts to a homeomorphism on $\boundary\cuco X$.
\end{cor}

\begin{proof}
Let $f:(\cuco X, \mathfrak S) \rightarrow (\cuco X, \mathfrak S)$ be an automorphism.  Let $p \in \partial \cuco X$, with $p = \sum_{i=1}^n a_{T_i}^p p_{T_i}$, where the $T_i$ are pairwise orthogonal and $p_{T_i} \in \partial \fontact T_i$.  Define a map $\hat{f}: \partial \cuco X \rightarrow \partial \cuco X$ by 
$$\hat{f}(p) = \sum_{i=1}^n a_{T_i}^p \hat{f}(p_{T_i}),$$
where $\hat{f}:\partial \fontact T_i \rightarrow \partial \fontact (f(T_i))$ is induced by $f:\fontact T_i\to\fontact T_i$.  Let $\overline{f}: \overline{\cuco X} \rightarrow \overline{\cuco X}$ be the extension of $f$ which is $\hat{f}$ on $\partial \cuco X$; extend $f^{-1}$ similarly.  Since $f$ is an automorphism, $\overline{f}$ is clearly a bijection.  Continuity of $\overline f,\overline f^{-1}$ on the boundary follows from Theorem \ref{thm:extension to boundary maps}.
\end{proof}

When $(G,\mathfrak S)$ is a hierarchically hyperbolic group, $\boundary G$ is defined.  In general, if $\cuco X,\cuco X'$ are hierarchically hyperbolic with respect to the same collection $\mathfrak S$, then there is a quasiisometry $\cuco X\to\cuco X'$ extending to the identity on the boundary.  Indeed, the definition of $\boundary\cuco X$ depends only on $\mathfrak S$ and the attendant hyperbolic spaces.
  
\begin{cor}\label{cor:hhg_boundary}
Let $(G,\mathfrak S)$ be a hierarchically hyperbolic group.  Then the action of $G$ on itself by left multiplication extends to an action of $G$ on $\overline G$ by homeomorphisms.
\end{cor}

Section~\ref{subsec:classification} is devoted to automorphisms, whose fixed points in $\partial \cuco X$ we study in Section~\ref{subsec:dynamics}.

\subsection{Classification of HHS automorphisms}\label{subsec:classification}
In this subsection, we will classify HHS automorphisms by their actions on $\mathfrak S$.  Let $g \in \Aut(\mathfrak S)$ and fix a basepoint $X \in \cuco X$.   Set $$\BIG(g) = \{U \in \mathfrak S|\diam_{\fontact U}(\langle g \rangle \cdot X) \text{ is unbounded}\}.$$  Observe that $g \cdot U \in \BIG(g)$ if $U \in \BIG(g)$, since $g\colon\fontact U\to\fontact (gU)$ is an isometry.

\begin{lem}\label{lem:Big is fixed}
There exists $M=M(\mathfrak S)>0$ so that for all $g\in\Aut(\mathfrak S)$ and $U\in\BIG(g)$, we have $g^M\cdot U=U$.  
\end{lem}

\begin{proof}
Consider the orbit $\langle g\rangle\cdot U$ in $\mathfrak S$.  

If there exists $n\geq1$ so that $g^n\cdot U\propnest U$, then $g^{kn}\cdot U\propnest g^{(k-1)n}\cdot U\propnest\cdots\propnest g^n\cdot U\propnest U$ for all $k\geq1$, so we either contradict finite complexity (if $\langle g\rangle\cdot U$ is infinite) or the fact that $\nest$ is a partial order (if $\langle g\rangle\cdot U$ is finite).  Hence $g^n\cdot U\not\propnest U$ unless $n=0$.  Similarly, $U\not\propnest g^nU$ unless $n=0$.

Next, consider the case where $U\in\BIG(g)$ and $g^n\cdot U\transverse U$ for some $n\geq1$.  Then since $U\in\BIG(g)$, we can choose arbitrarily large $m\in\naturals$ so that $\dist_U(X,g^m\cdot X)>T=100E+\dist_U(g^{-1}\cdot X,X)+f(m)$, where $f:\naturals\to\naturals$ is increasing.  Hence $\dist_{g^nU}(g^{m+1}\cdot X,g\cdot X)>T$, since $g:\contact U\to \contact gU$ is an isometry.    The triangle inequality shows that $\dist_{g^nU}(g^M\cdot X,X)> T-2\dist_{gU}(X,g^n\cdot X)=100E+f(m)$.  By considering at least two such values of $m$, we see that consistency is contradicted (specifically, we contradict Lemma~2.3 of~\cite{BehrstockHagenSisto:HHS_II}).

It follows that if $U\in\BIG(g)$, then, for all $n\in\integers$, either $g^n\cdot U=U$ or $g^n\cdot U\orth U$.  Hence $\langle g\rangle\cdot U$ is a pairwise-orthogonal collection.  Hence there exists a global $M$, depending only on the complexity and Lemma~\ref{lem:pairwise_orthogonal}, so that $g^M\cdot U=U$ for each $U\in\BIG(g)$, establishing the first assertion.
\end{proof}

\begin{prop}\label{prop:elliptic}
The automorphism $g\in\Aut(\mathfrak S)$ is elliptic if and only if $\BIG(g)=\emptyset$.
\end{prop}

\begin{proof}
If $\langle g\rangle\cdot X$ is bounded, then $\BIG(g)=\emptyset$ since projections are coarsely Lipschitz.

Conversely, suppose that $\BIG(g)=\emptyset$.  We will show that there exists $D=D(g)$ so that $\diam_V(\pi_V(\langle g\rangle\cdot X))\leq D$ for all $V\in\mathfrak S$.  From this and the distance formula (Theorem~\ref{thm:distance_formula}), it follows that $g$ is elliptic.  Hence suppose that no such $D$ exists.

We need two facts:
\begin{enumerate}[(I)]
 \item For each $N\geq0$, there exists $P=P(N,\mathfrak S)$ so that for all $U\in\mathfrak S$ and $h\in\Aut(\mathfrak S)$, either some positive power of $h$ fixes $U$ or $\{U,g\cdot U,\ldots,g^P\cdot U\}$ contains a set of $N$ pairwise-transverse elements.  Indeed, as in the proof of Lemma~\ref{lem:Big is fixed}, for any $p$, the elements of $\{U,g\cdot U,\ldots,g^{p-1}\cdot U\}$ are pairwise $\nest$--incomparable, and any pairwise-orthogonal subset has cardinality bounded by the complexity $\chi$ of $\mathfrak S$.  Hence, if $p$ exceeds the Ramsey number $\mathrm{Ram}(\chi+1,N)$, we have by Ramsey's theorem that $\{U,g\cdot U,\ldots,g^{p-1}\cdot U\}$ contains a set of $N$ pairwise-transverse elements, so we can take $P=\mathrm{Ram}(\chi+1,N)-1$.\label{item:ramsey} 
 \item For each $C\geq0$ there exists $Q\in\naturals$ with the following property.  Let $x,y\in\cuco X$ and suppose $\{V_i\}_{i\in I}$ satisfies $\dist_{V_i}(x,y)>E$ for all $i$, and that $|I|\geq Q$.  Then there exists $V\in\mathfrak S$ so that $V_i\propnest V$ for some $i\in I$, and $\dist_V(x,y)>C$.  This is a slight strengthening of Lemma~\ref{lem:nest_progress}; this exact statement is~\cite[Lemma~1.8]{BehrstockHagenSisto:HHS_III}. \label{item:passing}
\end{enumerate}

Recall that $\chi$ denotes the complexity -- i.e. the maximum level -- in $\mathfrak S$, so that $S$ is the unique element of level $\chi$.  Since $\BIG(g)=\emptyset$ but there are arbitrarily large projections, by assumption, there exists a level $\ell<\chi$ and a constant $R<\infty$ so that:
\begin{itemize}
 \item $\diam_U(\pi_U(\langle g\rangle\cdot X))\leq R$ when $U$ has level greater than $\ell$;
 \item for each $D<\infty$, there exists $U\in\mathfrak S$, of level $\ell$, with $\diam_U(\pi_U(\langle g\rangle\cdot X))>D$.
\end{itemize}

Let $U\in\mathfrak S$ be chosen so that $\dist_U(X,g^n\cdot U)>\ddot R$, where $\ddot R$ is a constant to be determined.  We can and shall assume that our $U$ has been chosen at level $\ell$, and we emphasize that such a $U$ can be chosen so as to make $\ddot R$ arbitrarily large.

Let $Q=Q(R)$ be the constant provided by setting $C=R$ in fact~\eqref{item:passing} and let $P=\mathrm{Ram}(\chi+1,Q)$.  Fact~\eqref{item:ramsey} provides $U_1,\ldots,U_Q\in\{U,g\cdot U,\ldots,g^P\cdot U\}$ so that $U_i\transverse U_j$ when $i\neq j$.  Now, for $1\leq j\leq Q$, we have $\dist_{U_j}(X,g^n\cdot X)\geq\ddot R-100KEQ$.  So, provided $\ddot R$ --- which can be chosen \emph{independently} of $R$ and hence of $Q$ --- satisfies $\ddot R>100KEQ+10E$, fact~\eqref{item:passing} provides $T\in\mathfrak S$ so that $U_j\propnest T$ for some $j$ and so that $\dist_T(X,g^n\cdot X)>R$.  Now, since $U_j$ is a translate of $U$ and $\Aut(\mathfrak S)$ preserves the levels, the level of $U_j$ is $\ell$, and hence $T$ has level strictly greater than $\ell$, which is a contradiction since $\dist_T(X,g^n\cdot X)>R$.
\end{proof}

\begin{rem}\label{rem:back_to_the_future}
In the case where $\cuco X$ is proper, there is a quick proof of Proposition~\ref{prop:elliptic} relying on the more powerful tools from Section~\ref{sec:rank_rigidity}.
\end{rem}

\begin{lem}\label{lem:not_big}
Let $g\in\Aut(\mathfrak S)$.  Then there exists $D=D(g,E)$ so that $\diam_U(\pi_U(\langle g\rangle\cdot X))\leq D$ for all $U\in\mathfrak S-\BIG(g)$.  
\end{lem}

\begin{proof}
Let $\BIG(g)=\{U_i\}_{i\in I}$.  Note that it suffices to prove the lemma for some positive power of $g$, so by Lemma~\ref{lem:Big is fixed}, we may assume that $g\cdot U_i=U_i$ for all $i\in I$. 

If $\BIG(g)=\emptyset$, then $g$ is elliptic by Proposition~\ref{prop:elliptic}, from which the lemma follows immediately: for each $V\in\mathfrak S$, we have $\diam_V(\pi_V(\langle g\rangle\cdot X))\leq K\diam_{\cuco X}(\langle g\rangle\cdot X)$, which is bounded independently of $V$.

Next, suppose that $\BIG(g)\neq\emptyset$ and $S\not\in\BIG(g)$ (as usual, $S\in\mathfrak S$ is the unique $\nest$--maximal element).  Then, for each $i\in I$, the element $U_i$ is maximal in an HHS $(F_{U_i},\mathfrak S_{U_i})$ admitting a $g$--equivariant hieromorphism to $(\cuco X,\mathfrak S)$.  Since $U_i\neq S$, the complexity of $(F_{U_i},\mathfrak S_{U_i})$ is strictly lower than that of $(\cuco X,\mathfrak S)$, so it follows by induction that $\diam_V(\pi_V(\langle g\rangle\cdot X))$ is bounded independently of $V$ when $V\nest U_i$.  Indeed, in the base case, when the complexity is $1$, $\cuco X$ is itself a hyperbolic space and the lemma follows from the usual elliptic/parabolic/loxodromic classification of isometries of hyperbolic spaces~\cite{Gromov:essay}.

Now, let $\mathfrak T$ be the set of all $U\in\mathfrak S$ so that $U\nest U_i$ for some $i\in I$.  Observe that $\mathfrak T$ is $g$--invariant and downward-closed under nesting.  Then Proposition~2.4 of~\cite{BehrstockHagenSisto:HHS_III} provides an HHS $(\widehat{\cuco X}_{\mathfrak T},\mathfrak S-\mathfrak T)$ with the same associated nesting and orthogonality relations, hyperbolic spaces, and projections.  Since $\mathfrak T$ was $g$--invariant, $g$ descends to an automorphism of $(\widehat{\cuco X}_{\mathfrak T},\mathfrak S-\mathfrak T)$ so that the action of $g$ on $\mathfrak S-\mathfrak T$ is the restriction of the original action on $\mathfrak S$ and, for each $V\in\mathfrak S-\mathfrak T$, the isometry $\contact V\to\contact gV$ is the original one.  Now $g$ has $\BIG(g)= \emptyset$ with respect to $(\widehat{\cuco X}_{\mathfrak T},\mathfrak S-\mathfrak T)$ and hence we are done by the proof of Proposition \ref{prop:elliptic}.

The preceding two analyses prove the lemma except in the case where $S\in\BIG(g)$.  Hence, suppose $S\in\BIG(g)$, so that $g$ acts either loxodromically or parabolically on $\contact S$.  In this case, we cannot induct on complexity, so we argue directly using consistency, bounded geodesic image, and simple properties of isometries of hyperbolic spaces.

If $U\in\mathfrak S-\{S\}$, then $U\propnest S$, and $\rho^U_S\subset\contact S$ is a well-defined diameter--$\leq E$ subset.

First suppose that $g$ acts loxodromically on $\contact S$.  Then there exists $N=N(g)$ so that $\leq N$ elements of $\pi_S(\langle g\rangle\cdot X)$ lie in the $100E$--neighborhood of $\rho^U_S$.  Let $\{g^i\cdot X\}_{i=n}^{n'}$ be the points in $\langle g\rangle\cdot X \subset \cuco X$ projecting into $\neb^S_{100E}(\rho^U_S) \subset \fontact S$, so that $n'-n\leq N$.  Then for all $i,j\in\integers$, consistency and bounded geodesic image imply that 
\begin{eqnarray*}
\dist_U(g^i\cdot X,g^j\cdot X)&\leq& E+\max_{n\leq k,k\leq n'}\dist_S(g^k\cdot X,g^{k'}\cdot X)\\
&\leq&E+\max_{0\leq k,k'\leq N}K\dist_{\cuco X}(g^k\cdot X,g^{k'}\cdot X)+K,
\end{eqnarray*}
which is independent of $U$ (here $K$ is the coarse Lispchitz constant from Definition~\ref{defn:space_with_distance_formula}).

Next, suppose that $g$ acts parabolically on $\contact S$.  By definition, $\langle g\rangle\cdot X$ has a unique limit point in the Gromov boundary of $\contact S$, so there is an increasing function $f:\naturals\to\naturals$ so that $(g^n\cdot\pi_S(X)|g^m\cdot\pi_S(X))_{\pi_S(X)}>f(k)$ whenever $\min\{|m|,|n|\}\geq k$.  In particular, there exists $k$, independent of $U$, so that no $\contact S$--geodesic from $\pi_S(g^n\cdot X)$ to $\pi_S(g^m\cdot X)$ passes $100E$--close to $\rho^U_S$ provided $|m|\geq k$ and $|n|\geq k$.  We now argue exactly as in the loxodromic case to bound $\diam_U(\pi_U(\langle g\rangle\cdot X))$ independently of $U$.  This completes the proof.
\end{proof}

\begin{lem}\label{lem:Big is orth}
For any distinct $U, V \in \BIG(g)$, we have $U \orth V$.
\end{lem}

\begin{proof}
Lemma~\ref{lem:Big is fixed} shows that by passing to a uniformly bounded power, if necessary (which does not affect the big-set), we can assume that $gU=U$ and $gV=V$.  Hence $g$ acts as an isometry of both of the (not necessarily proper) hyperbolic spaces $\contact U,\contact V$.  Since $U,V\in\BIG(g)$, the isometry $g$ cannot be elliptic on either $\contact U$ or $\contact V$.  Hence, by e.g.~\cite[Section 8.1]{Gromov:essay}, $g$ is either parabolic or loxodromic on $\contact U$ and $\contact V$.

If $U\propnest V$ or $U\transverse V$, then $\rho^U_V$ is a uniformly bounded subset of $\contact V$, and, since $g^n\cdot\rho^U_V\asymp\rho^{g^nU}_{g^nV}=\rho^U_V$ for all $n\in\integers$, we have that $\langle g\rangle$--orbits in $\contact V$ are bounded, contradicting that $U\in\BIG(g)$.  
\end{proof}

\begin{defn}[Elliptic] \label{defn:elliptic}
An automorphism $g \in \Aut(\mathfrak S)$ is \emph{elliptic} if some (hence any) orbit of $\langle g\rangle$ in $\cuco X$ is bounded.
\end{defn}

\begin{defn}[Axial] \label{defn:axial}
An automorphism $g \in \Aut(\mathfrak S)$ is \emph{axial} if some (hence any) orbit of $\langle g\rangle$ in $\cuco X$ is quasiisometrically embedded.
\end{defn}

\begin{defn}[Distorted] \label{defn:distorted}
An element $g \in \Aut(\mathfrak S)$ is \emph{distorted} if it is not elliptic or axial.
\end{defn}

\begin{exmp}[Distorted automorphisms in familiar examples]\label{exmp:distorted}
Let $S$ be a surface of finite type and $\alpha$ a simple closed curve.  In $\MCG(S)$, the subgroup $\langle \tau_{\alpha} \rangle$ generated by the Dehn twist about $\alpha$ is quasiisometrically embedded \cite{FarbLubMinsky}, but in $(\TT(S), d_T)$, the orbit of $\tau_{\alpha}$ is distorted.  In fact, $\MCG(S)$ has no distorted automorphisms, as is the case for cube complexes with factor systems, since cubical automorphisms are combinatorially semisimple \cite{Haglund:semisimple}.  In Theorem \ref{thm:HHGs have no distorteds} below, we prove that HHGs have no distorted elements.  A simple example of an HHS with a distorted automorphism is obtained by gluing a combinatorial horoball to $\integers$; this encapsulates the difference between the HHS structures of $\MCG(S)$ and $(\TT(S),d_T)$, where annular curve graphs are replaced by horoballs over annular curve graphs.
\end{exmp}

\begin{prop}\label{prop:axial}
The automorphism $g\in\Aut(\mathfrak S)$ is axial if and only if there exists $U\in\BIG(g)$ such that $n\to g^n \cdot\pi_U(X)$ is a quasiisometric embedding $\integers\to\fontact U$ for any $X \in \cuco X$.
\end{prop}

\begin{proof}
Suppose that there exists $U\in\BIG(g)$ so that $n\to g^n\cdot\pi_U(X)$ is a quasi-isometric embedding.  Then the distance formula (Theorem~\ref{thm:distance_formula}) yields a lower bound on $\dist_{\cuco X}(g^m\cdot X,g^n\cdot X)$ which is (at least) linear in $|m-n|$, i.e. $g$ is axial.

Conversely, suppose that $g$ is axial.  Lemma \ref{lem:Big is orth} bounds the number of $U\in\BIG(g)$ by the complexity of $\mathfrak S$.  Lemma~\ref{lem:not_big} ensures that $\diam_V(\pi_V(\langle g\rangle\cdot X))$ is bounded independently of $V$ for $V\not\in\BIG(g)$.  Since $g$ acts axially on $\cuco X$, the distance formula (Theorem~\ref{thm:distance_formula}) now implies that there exists at least one $U \in \BIG(g)$ such that $g$ acts axially on $\fontact U$.
\end{proof}

The next proposition is an immediate consequence of Propositions \ref{prop:elliptic} and \ref{prop:axial}:

\begin{prop}\label{prop:distorted}
The automorphism $g\in\Aut(\mathfrak S)$ is distorted if and only if there exists $U\in\BIG(g)$ such that $\langle g\rangle\cdot\pi_U(X)$ is unbounded, but, for all $U\in\BIG(g)$, we have $$\dist_{\fontact U}(X,g^n \cdot X)=o(n).$$
\end{prop}

\begin{defn}[Reducible]\label{defn:reducible}
The automorphism $g\in\Aut(\mathfrak S)$ is \emph{irreducible} if $\BIG(g)=\{S\}$, where $S\in\mathfrak S$ is the unique $\nest$--maximal element.  Otherwise, $S\not\in\BIG(g)$ and $g$ is \emph{reducible}.  
\end{defn}

Finally, we have the following strong characterization of irreducible axials:

\begin{thm}\label{thm:acyl_morse}
Let $G\leq\Aut(\mathfrak S)$ act properly and coboundedly on the
hierarchically hyperbolic space $(\cuco X,\mathfrak S)$.  Suppose that
$g\in G$ is irreducible axial.  Then $g$ is Morse.
\end{thm}

\begin{proof}
By~\cite[Corollary~14.4]{BehrstockHagenSisto:HHS_I}, $G$ acts
acylindrically on $\fontact S$, where $S$ is $\nest$--maximal
in $\mathfrak S$, while $g$ acts hyperbolically on
$\fontact S$.  By~\cite[Proposition~3.8]{Sisto:contracting_random}, $g$ is \emph{weakly
contracting} for the \emph{path system} consisting of all geodesics in
$\fontact S$, so $g$ is Morse, by~\cite[Lemma~2.9]{Sisto:contracting_random}.
\end{proof}

\begin{rem}[Reducible Morse elements]\label{rem:morse_reducible}
The converse of Theorem~\ref{thm:acyl_morse} does not hold, as can be seen be examining a Morse element of an appropriately-chosen right-angled Artin group whose support does not include all generators.
\end{rem}

\subsection{Dynamics of action on the boundary}\label{subsec:dynamics}
In the remainder of this section, we impose the standing assumption that $\cuco X$ is proper.  We will analyze the action of an infinite-order automorphism $g$ on $\boundary(\cuco X,\mathfrak S)$, according to whether $g$ is irreducible or reducible and according to whether $g$ is axial or distorted.  

\subsubsection{Irreducible automorphisms}\label{subsubsec:irreducible}  

\begin{lem}\label{lem:irreducible limit support}
Let the irreducible $g \in \Aut(\mathfrak S)$ fix some $\lambda \in \partial \cuco X$.  Then $\supp(\lambda) = \{S\}$.
\end{lem}

\begin{proof}
Suppose $U \in \supp(\lambda)-\{S\}$.  Since $g$ is irreducible, its orbit in $\fontact S$ is unbounded.  In particular, this means that the orbit of $\rho^U_S$ is unbounded.  By definition, $g \cdot \rho^U_S \asymp \rho^{g\cdot U}_S$ and thus $U$ could not be fixed by $g$, completing the proof.
\end{proof}

\begin{prop}[Irreducible axials act with north-south dynamics] \label{prop:irreducible axial}
If $g \in \Aut(\mathfrak S)$ is irreducible axial, then $g$ has exactly two fixed points $\lambda_+, \lambda_- \in \partial \cuco X$.  Moreover, for any boundary neighborhoods $\lambda_+ \in U_{+}$ and $\lambda_- \in U_{-}$, there exists an $N>0$ such that $g^N(\partial \cuco X - U_-) \subset U_+$.
\end{prop}

\begin{proof}
Let $g \in \Aut(\mathfrak S)$ be irreducible axial.  For the rest of the proof, fix a basepoint $X \in \cuco X$.

\textbf{Existence of $\lambda_+, \lambda_- \in \partial \cuco X$:}  For any $n$, let $X_n = g^n \cdot X$.  We will show that $(X_n)$ converges to some point in $\partial \cuco X$; a similar argument will show that $(X_{-n})$ converges to some other point, and then we will prove they are distinct.
By compactness (Theorem \ref{thm:cpt}), there exists a subsequence $(X_{n_k}) \subset (X_n)$ which converges to some point $\lambda_+ \in \partial \cuco X$.  By irreducibility of $g$, we must have that $\lambda_+ \in \partial \fontact S \subset \partial \cuco X$.  By irreducibility and the definition of convergence, we have that $\pi_{\fontact S}(X_{n_k}) \rightarrow \lambda_+ \in \partial \fontact S$.  Axiality of $g$ then implies that, for any other subsequence $(X_{n_l}) \subset (X_n)$, the Gromov product $(X_{n_k}, X_{n_l})_X \rightarrow \infty$ in $\fontact S$ as $k, l \rightarrow \infty$.  This implies that $\pi_{\fontact S}(X_n) \rightarrow \lambda_+ \in \partial \fontact S$, which implies that $X_n \rightarrow \lambda_+ \in \partial \cuco X$.

Similarly, we define $X_{-n}\rightarrow \lambda_- \in \partial \cuco X$.  Observe that $\left(\pi_{\fontact S}(X_n), \pi_{\fontact S}(X_{-n})\right)_{\pi_{\fontact S}(X)}$ is uniformly bounded by Proposition \ref{prop:axial}, implying that $\lambda_+ \neq \lambda_-$.  Since $g$ stabilizes the orbit, it obviously fixes $\lambda_+$ and $\lambda_-$.  Note that $\lambda_+, \lambda_-$ are independent of our choice of $X \in \cuco X$.

\textbf{Uniqueness of $\lambda_+, \lambda_- \in \partial \cuco X$:} By Lemma \ref{lem:irreducible limit support}, any point $\lambda \in \partial \cuco X$ fixed by $g$ has $\supp(\lambda) = S$.  If $g$ fixes three points in $\partial \cuco X$, then it fixes three points in $\partial \fontact S$.  As such, $g$ coarsely fixes the coarse median of those points, producing a bounded orbit, a contradiction.

\textbf{North-south dynamics on $\partial \cuco X$:} Fix boundary neighborhoods $\lambda_+ \in U_+$ and $\lambda_- \in U_-$ with $U_+ \cap U_- = \emptyset$.  

\setcounter{claimnum}{0}
\begin{claim}\label{claim:north-south}
For any $p \in \partial \cuco X-\{\lambda_-\}$, $(g^n(p))$ does not converge to $\lambda_-$.
\end{claim}

\renewcommand{\qedsymbol}{$\blacksquare$}

\begin{proof}[Proof of Claim \ref{claim:north-south}]
If $\supp(p) \neq \{S\}$, then $(g^n(p))$ cannot converge to a point in $\partial \cuco X$ supported on $S$, as $g$ does not alter the coefficients of the pieces of $p$ supported on proper subdomains. In particular, since $\supp(\lambda_-)=\{S\}$, as shown above, $(g^n(p))$ cannot converge to $\lambda_-$.  Thus we may assume that $\supp(p) = \{S\}$.

Let $[X, p]$ be a hierarchy ray in $\overline{\cuco X}$.  Since $\supp(p)= \{S\}$, $[X,p]$ projects to a $D$-quasigeodesic, $[X, p]_S \subset \overline{\fontact S}$.  Let $[X, \lambda_-]$ be the orbit $(g^{-n}(X))$, which is a quasigeodesic with quality depending on $g$.

Consider $m\in \fontact S$, the coarse median of $(\lambda_-, p, X)$.  By hyperbolicity, there exist points $Y \in [X,p]_S$, $Z \in [X, \lambda_-]$ sufficiently far out along $[X, p]_S$ and $[X, \lambda_-]$ such that any geodesic between $Y, Z$, $[Y,Z]$, comes uniformly close to $m$, independent of $Y$ and $Z$; in particular, the coarse median of $(X, Y, Z)$ is uniformly close to $m$.  Moreover, there is a uniform constant $\delta'>0$ (depending on $D$, $g$, and the hyperbolicity constant, $\delta>0$) such that each of $[Y,Z], [X,Y],$ and $[X, Z]$ is $\delta'$-close to $m$. 

Let $m_{Y,Z} \in [Y,Z]$ and $m_{X,Z} \in [X, Z]$ be points $\delta'$-close to $m$.  Then there exists a uniform $\delta''>0$ such that $[m_{Y,Z}, Z]$ and $[m_{X,Z}, Z]$ must $\delta''$-fellow travel.  By axiality, there exists $N>0$ such that, for all $n>N$, $g^n(m_{X,Z})$ is between $X$ and $g^n(X)$ along the quasigeodesic axis of $g$ in $\fontact S$.  This implies that the coarse median of $(X, g^n(Y), g^n(Z))$ is uniformly close to $X$.  Thus $(g^n(p), \lambda_-)_{X}$ is uniformly bounded and $(g^n(p))$ cannot converge to $\lambda_-$ in $\partial \fontact S$ and thus not in $\partial \cuco X$ as well.\end{proof}
\renewcommand{\qedsymbol}{$\Box$}

Since the limit of $(g^n(p))$ is a fixed point, uniqueness of $\lambda_-, \lambda_+$ and Claim \ref{claim:north-south} imply that $g^n(p) \rightarrow \lambda_+$ for any $p \in \partial \cuco X - \{\lambda_-\}$.

Now consider the function $f: \partial \cuco X - U_- \rightarrow \naturals$, where $f(p)$ is the least power $N_p$ such that $g^{N_p}(p) \in U_+$.  Since $\lambda_+$ and $\lambda_-$ are the unique fixed points of $g$, such a power exists, otherwise the sequence $(g^n(p)) \subset \partial \cuco X$ would subconverge to another fixed point.  Since $\partial \cuco X$ is compact (Theorem \ref{thm:cpt}) the function $f$ attains a maximum, $N_f$.  By definition, $g^{N_f}(\partial \cuco X - U_-) \subset U_+$, completing the proof.
\end{proof}

We now treat the irreducible distorted case:

\begin{prop}[Irreducible distorteds act parabolically] \label{prop:irreducible distorted}
If $g \in \Aut(\mathfrak S)$ is irreducible distorted, then $g$ has exactly one fixed point $\lambda_g \in \partial \cuco X$, and  $g^n \cdot X, g^{-n} \cdot X \rightarrow \lambda_g$ for any $X \in \overline{\cuco X}$.
\end{prop}

\begin{proof}
Let $S\in\mathfrak S$ be the unique $\nest$--maximal element, so that $gS=S$ and $g:\contact S\to\contact S$ is an isometry.  By the definition of irreducibility, $\BIG(g)=\{S\}$, so $g$ has unbounded orbits in the $\delta$--hyperbolic space $\contact S$.  We now apply the classification of isometries of hyperbolic spaces, as summarised in~\cite[Section~3]{CCMR}, emphasizing that these results do \emph{not} rely on properness of the space in question.

First, by Proposition~3.2 of~\cite{CCMR} and the fact that $\langle g\rangle\cdot\pi_X(X)$ (which coarsely coincides with $\pi_S(\langle g\rangle\cdot X)$) is distorted -- i.e. not quasiconvex -- in $\contact S$, we have that the action of $\langle g\rangle$ on $\contact S$ is \emph{not} lineal or focal.  By Lemma~3.3, the action of $\langle g\rangle$ on $\contact S$ is not of \emph{general type}.  Hence the action is \emph{horocyclic}, i.e. the limit set of $\langle g\rangle$ on $\boundary\contact S$ consists of \textbf{exactly one point} $\lambda_g$ with $g\lambda_g=\lambda_g$.  Moreover, Proposition~3.1 of~\cite{CCMR} implies that every $\lambda\neq\lambda_g$ in $\boundary\contact S$ has infinite $\langle g\rangle$--orbit.  We also denote by $\lambda_g$ the image of this limit point under the usual ($\Aut(\mathfrak S)$--equivariant) embedding $\boundary\contact S\to\boundary\cuco X$.  We thus have a fixed point $\lambda_g\in\boundary\cuco X$ for $g$.  Now, suppose that $\lambda\in\boundary\cuco X$ is fixed by $g$.  By Lemma~\ref{lem:irreducible limit support}, $\lambda\in\boundary\contact S\subset\boundary\cuco X$.  If $\lambda\neq\lambda_g$, then (as a point of $\boundary\contact S$), $\lambda$ cannot be fixed by $g$, so $\lambda_g$ is the unique fixed point in $\boundary\cuco X$.

Finally, if $p \in \partial \cuco X - \lambda_g$, then $g^n \cdot p \rightarrow \lambda_g$, for it subconverges to some point by compactness of $\overline{\cuco X}$ (Theorem \ref{thm:cpt}), which is fixed by $g$ and thus must be $\lambda_g$ by uniqueness.
\end{proof}

\begin{prop}\label{prop:irreducible distorted dynamics}
Let $g \in \Aut(\mathfrak S)$ be irreducible distorted and fix $\lambda_g \in \partial \cuco X$.  For any neighborhood $U\subset\boundary\cuco X$ of $\lambda_g$, there exists $N>0$ such that if $p \in \partial \cuco X - U$, then $g^N \cdot p \in U$.
\end{prop}

\begin{proof}
Fix a neighborhood $\lambda_g \in U \subset \partial \cuco X$ and let $p \in \partial \cuco X - U$.  Let $F: \overline{\cuco X} \rightarrow \naturals$ be the map which takes each $p \in \overline{\cuco X}$ to the minimal $n \in \naturals$ such that $g^n \cdot p \in U$; note that $F$ is defined by Proposition \ref{prop:irreducible distorted}.  We prove that $F$ is bounded.

Assume not; then there exists a sequence $(p_i) \subset \partial \cuco X$ such that $F(p_i) = n_i \rightarrow \infty$ as $i \rightarrow \infty$.  By compactness of $\overline{\cuco X}$, the sequence $(p_i)$ accumulates on some point $\mu \in \partial \cuco X$.  If $N_{\mu} = F(\mu)$, then $g^{N_{\mu}}\cdot \mu \in U$.  Choose an open neighborhood $g^{N_{\mu}} \cdot \mu \in V \subset U$.

By passing to a subsequence if necessary, we may assume $p_i \rightarrow \mu$ and continuity of the action of $g$ on $\partial \cuco X$ implies that $g^{N_{\mu}} \cdot p_i \rightarrow g^{N_{\mu}} \cdot \mu$.  In particular, this implies that the sequence $(g^{N_{\mu}} \cdot p_i)$ eventually lies in $V \subset U$, a contradiction.
\end{proof}

\subsubsection{Reducible automorphisms}  We now turn to non-elliptic reducible automorphisms.  As before, we assume $\cuco X$ is proper, $g\in\Aut(\mathfrak S)$ has infinite order and is thus axial or distorted, and $\BIG(g)\neq\emptyset$ denotes the set of (pairwise--orthogonal) $U\in\mathfrak S$ where $\diam_{\fontact U}(\langle g\rangle\cdot X)=\infty$.

If $g$ is reducible, then $\BIG(g)=\{A_i\}\sqcup\{B_j\}$, where $g$ acts axially on $\fontact A_i$ and distortedly on $\fontact B_j$ for all $i,j$ and $A_i,B_j\neq S$ for all $i,j$.  Proposition~\ref{prop:axial} implies that $g$ is axial if and only if $\{A_i\}\neq\emptyset$; otherwise $g$ is distorted.

We must be careful with nontrivial finite orbits in $\mathfrak S$.  To that end, recall that by Lemma~\ref{lem:Big is fixed} there exists $M=M(\mathfrak S)>0$ such that $g^M$ fixes $\BIG(g)$ pointwise.  The proof of that lemma shows that $g^M$ in fact fixes $\{A_i\}$ and $\{B_i\}$ pointwise, since we cannot have $g\cdot A_i=B_j$ for any $i,j$.  Let $h=g^M$, and note that $\BIG(h)=\BIG(g)$.  Note that we can choose $M$ so that \emph{any} pairwise-orthogonal subset of $\mathfrak S$ stabilized by $h$ is fixed by $h$ pointwise.

\begin{lem}\label{lem:lower_fixed_points}
Let $V\in\mathfrak S$ and suppose that $V\nest U$ or $V\transverse U$, for some $U\in \BIG(g)$.  Suppose also that $p\in\boundary \cuco X$ is fixed by $g$.  Then $V\not\in\support(p)$.
\end{lem}

\begin{proof}
By hypothesis, $h\cdot p=p$.  Observe that $\langle h\rangle \cdot\rho^V_U$ is unbounded.  Since $U \in \BIG(g)$, we have that $h\cdot \rho^V_U=\rho^{h\cdot V}_U$ and $h\cdot U$ is infinite, implying $U\not\in\support(p)$, as required.
\end{proof}

We denote by $\mathbb S^k$ a $k$--sphere and by $\mathbb D^k$ a $k$--ball.  Given spaces $X,Y$, we denote by $X\star Y$ their join.  For each $i,j$, let $F_i=F_{A_i},F'_j=F_{B_j}$ be the standard factors associated to $A_i,B_j$, so that there is a quasiconvex hieromorphism $\prod_{i}F_i\times \prod_jF'_j\to\cuco X$, inducing an embedding $\bigstar_{i}\boundary F_i\to\bigstar_j\boundary F'_j\to\boundary\cuco X$ whose image is a closed $g$--invariant subset which we denote $\mathfrak E(g)$.  (Note: The image of $\prod_iF_i\times\prod_jF'_j$ need not be $g$--invariant, but since $g$ stabilizes each standard product region $F'_j \times E_{B_j}$, the subspaces $gF_i,F_i$ are parallel, and thus have the same boundary.)

For each $i$, the action of $h = g^M$ on $P_{F_i} \cong F_i \times E_{A_i}$ induces an action of $h$ on $F_i$ by applying the restriction homomorphism $\theta_{A_i}: \mathrm{Stab}_{\Aut(\mathfrak S)}(A_i) \rightarrow \Aut(\mathfrak S_{A_i})$.  For each $A_i$, let $h_i$ be the image of $h$ under this homomorphism, and let $h_j$ be the image of $h$ under the corresponding restriction homomorphism for $B_j$.   
 
The following proposition says that, up to taking a power, a reducible automorphism can be decomposed into irreducible automorphisms on subdomains:

\begin{prop}\label{prop:reducible big set}
If $g$ is non-elliptic reducible and $h=g^M$, then the following hold:
\begin{enumerate}
 \item For each $i$, $h_i$ is an irreducible axial automorphism of $F_i$ which fixes a unique pair of points $\lambda_{i,+},\lambda_{i,-}\in\boundary\fontact A_i$ and acts with north-south dynamics on $\boundary\fontact A_i$;
 \item For each $j$, $h_j$ is an irreducible distorted automorphism of $F'_j$ and fixes a unique point $\lambda_{h_j}\in\boundary\fontact B_j$.
\end{enumerate}
Hence, $g$ stabilizes (and $h$ fixes pointwise) a nonempty subspace $S(g)\star C(g)\subseteq\boundary\cuco X$, where $S(g)=\emptyset$ or $S(g)\cong \mathbb S^{|\{A_i\}|-1}$ and $C(g)=\emptyset$ or $C(g)\cong\mathbb D^{|\{B_j\}|}$.  Moreover, for all $n>0$, $g^n$ does not fix any point in $\mathfrak E(g)-S(g)\star C(g)$.
\end{prop}

\begin{proof}
For each $i$, $h_i$ acts on $\fontact A_i$ axially by the assumption on $g$ and irreducibly by construction.  Hence, Proposition \ref{prop:irreducible axial} implies that $h_i$ fixes two points $\lambda_{i,+}, \lambda_{i,-} \in \partial \fontact A_i$ and acts with north-south dynamics on $\partial \fontact A_i$.  Similarly, for each $j$, $h_j$ acts on $\fontact B_j$ distortedly by assumption and irreducibly by construction.  Proposition \ref{prop:irreducible distorted} then implies that $h_j$ fixes a unique point $\lambda_{h_j} \in \partial \fontact B_j$.

If $\{A_i\} \neq \emptyset$, then each $A_i$ contributes a pair of points $\lambda_{i,+}, \lambda_{i, -} \in \partial \fontact A_i$ fixed by $h$, which we can think of as a copy of $\mathbb{S}^0$, namely $\mathbb{S}_i^0$.  Moreover, $h$ clearly fixes the join of these spheres, $\bigstar_i \mathbb{S}_i^0 \cong \mathbb{S}^{|\{A_i\}| - 1} = S(g)$, as required.

Similarly, if $\{B_i\} \neq \emptyset$, then each $B_j$ contributes a point $\lambda_{h_j} \in \partial \fontact B_j$ fixed by $h$, and $h$ fixes the join of these points, $\bigstar_j \lambda_{h_j} \cong \mathbb{D}^{|\{B_j\}|} = C(g)$, as required.

Since $h$ fixes these $S(g)$ and $C(g)$, $h$ clearly fixes $S(g) \star C(g)$.  Now, if $g^n$ fixes a point $\lambda \in \mathfrak E(g)$, then $h^n = (g^n)^M$ fixes $\lambda$.  If $\lambda = \sum_i a_i p_i + \sum_j b_j q_j$, where $p_i \in \partial F_i$ and $q_i \in \partial F'_j$, then the uniqueness of the $\lambda_{i, +}, \lambda_{i, -}, \lambda_{h_j}$ implies that, for $a_i \neq 0, b_j \neq 0$, we must have $p_i = \lambda_{i, +}$ or $p_i = \lambda_{i,-}$ and $q_j = \lambda_{h_j}$.
\end{proof}

\begin{rem}
Set $\mathrm{Comp}(g) = \{p \in \partial \cuco X | \supp(p) \subset \{A_i, B_j\}^{\orth}_{i,j}\}$ and let $\mathrm{Fix}(h) \subset \partial \cuco X$ be the set of fixed points of $h$.  It is not difficult to show that $\mathrm{Fix}(h) \subseteq S(g) \star C(g) \star \mathrm{Comp}(g)$, but proper containment can happen.
\end{rem}

\begin{lem}\label{lem:bounded big papa}
Let $U \in \BIG(g)$ and $U \nest V$.  For all $p \in \partial \cuco X$ such that $g^n(p) = p$ for some $n>0$, we have $V \notin \supp(p)$.
\end{lem}

\begin{proof}
It suffices to prove the lemma for $h= g^M$.  Suppose for a contradiction that $V \in \supp(p)$.  Since $U \in \BIG(h)$, $\diam_V(\langle h \rangle \cdot \rho^U_V)$ is uniformly bounded.  Take any sequence $X_k \rightarrow p$ in $\overline{\cuco X}$; note that this implies $X_k \rightarrow p_V$ in $\overline{\fontact V}$.   Thus, there exists $K>0$ such that $d_V(X_k, \rho^U_V) > 100 E$ if $k\geq K$.

Since $h$ is unbounded on $\fontact U$, there exists $N>0$ depending only on $K$ such that $d_U(X_k, h^n(X_k))> 100E$ if $n \geq N$ and $k \geq K$.  If $\gamma$ is a hierarchy path between $X_k$ and $h^n(X_k)$ in $\cuco X$, then the bounded geodesic image axiom (Definition~\ref{defn:space_with_distance_formula}.\eqref{item:dfs:bounded_geodesic_image}) implies that $\pi_V(\gamma) \cap N_E(\rho^U_V) \neq \emptyset$.  In particular, this implies that $d_V(X_k, h^n(X_k)) > 100E$.  Thus, for any $n>N$, we have that $(X_k,h^n(X_k))_{\rho^U_V}$ is uniformly bounded as $k \rightarrow \infty$, which implies that no power of $h$ could fix $p$, a contradiction.
\end{proof}

\begin{prop}\label{prop:exterior}
Let $p \in \partial \cuco X$ be such that $g^M(p)= p$ for some $M>0$.  Then 
$$p \in S(g) \star C(g) \star \left(\bigcap_i \partial E_{A_i} \cap \bigcap_j \partial E_{B_j}\right).$$
\end{prop}

\begin{proof}
Lemmas \ref{lem:lower_fixed_points} and \ref{lem:bounded big papa} imply 
$$\supp(p) \subset \bigcup_{i,j} \left({\mathfrak S}_{A_i} \cup {\mathfrak S}_{B_j} \cup \left( \{A_i\}^{\orth} \cap \{B_j\}^{\orth}\right)\right),$$
which, together with Proposition \ref{prop:reducible big set} and $g$--invariance of $\BIG(g)$, gives the claim.
\end{proof}

\subsection{Dynamics on boundaries of HHG}  Fix a hierarchically hyperbolic group $(G,\mathfrak S)$.

\begin{defn}[Stable boundary points]
A point $p \in \partial G$ is a \emph{stable boundary point} if $p$ is a fixed point of some irreducible axial element of $\Aut(\mathfrak S)$.
\end{defn}

The next lemma states that irreducible axials have cobounded orbits.

\begin{lem}\label{lem:stable hp are cobounded}
Let $g \in G$ be an irreducible axial.  Then given any $X \in \cuco X$, there exists $N>0$ such that $\diam_{\fontact U}(\langle g \rangle \cdot X) < N$ for any $U \in \mathfrak S - \{S\}$.
\end{lem}

\begin{proof}
If not, then there exists a sequence of domains $U_n \in \mathfrak S$ such that $\diam_{\fontact U_n} (\langle g \rangle \cdot x) \geq n$ for each $n$.  Since $g$ is irreducible axial, $\langle g \rangle \cdot X$ projects to a uniform quasigeodesic in $\fontact S$.

By the bounded geodesic image axiom and hyperbolicity of $\fontact S$, for each $n> 100E$, there exists a sequence $(k_n) \subset \integers$ such that $\rho^{U_n}_S \in \neb_E([g^{k_n} \cdot X, g^{k_n + 1} \cdot X]) \subset \fontact S$, where $[g^{k_n} \cdot X, g^{k_n + 1} \cdot X]$ is any geodesic between $g^{k_n} \cdot X$ and $g^{k_n+1} \cdot X$ in $\fontact S$.  Moreover, since $\langle g \rangle \cdot X$ is a uniform quasigeodesic in $\fontact S$, it follows that $d_{U_n}(g^{k_n} \cdot X, g^{k_n + 1} \cdot X) \asymp \diam_{U_n}(\langle g \rangle \cdot X) \geq n$.

It follows that there exists a sequence of domains $U'_n = g^{-k_n} \cdot U_n \in \mathfrak S$ with $\rho^{U'_n}_S \in \neb_E([X, g \cdot X])$ and $d_{U'_n}(X, g \cdot X) \asymp \diam_{U'_n}(\langle g \rangle \cdot X) \geq n$, which is impossible by the distance formula.  This completes the proof.
\end{proof}

\begin{prop}\label{prop:stable dense}
If $G$ has an irreducible axial element, then the set of stable boundary points is dense in $\partial G$.
\end{prop}

\begin{proof}
Let $p \in \partial G$ be any point and let $\lambda \in \partial G$ be a stable boundary point for some irreducible axial $g \in G$.  Choose $X \in \cuco X$ and let $\gamma_n = [X, g^n \cdot X]$ be a $D$-hierarchy path between $X$ and $g^n \cdot X$.  Let $\gamma = [X, \lambda]$ be the limiting $D$-hierarchy ray as $n\rightarrow \infty$.  Since $\gamma_n \rightarrow \gamma$ uniformly on compact sets and $\langle g\rangle \cdot X$ is uniformly cobounded by Lemma \ref{lem:stable hp are cobounded}, it follows that $\gamma$ is uniformly cobounded.

By coboundedness of the action of $G$ and density of the interior (Proposition \ref{prop:properties}), there exists a sequence $(g_n) \subset G$ and $N>0$ such that $g_n(X) \rightarrow p$ and thus $g_n \cdot \lambda \rightarrow p$.  Since $G$ acts on itself by automorphisms, we have that $g_n \cdot [X, \lambda]$ projects to an infinite quasigeodesic in $\fontact S$, implying that $g_n \cdot \lambda \in \partial \fontact S \subset \partial G$, which completes the proof.\end{proof}

\begin{thm}[Topological transitivity of the $G$-action on $\partial G$]\label{thm:dense orbit}
Let $(G,\mathfrak S)$ be a hierarchically hyperbolic group with $G$ not virtually cyclic and containing an irreducible axial element.  For any $p\in\boundary G$, $G\cdot p$ is dense in $\boundary G$.
\end{thm}

\begin{proof}
Let $U\subseteq\boundary G$ be an open set.  By Proposition~\ref{prop:stable dense}, there exists an irreducible axial $g \in G$ with stable boundary points $\lambda_{g, +}, \lambda_{g, -} \in \partial G$, one of which is contained in $U$.  Suppose that $\lambda_{g,+}\in U$ and $\lambda_{g,-}\neq p$.  Then since $\boundary G$ is Hausdorff, it follows from Proposition~\ref{prop:irreducible axial} that some power of $g$ moves $p$ into $U$, as required.  Hence either we are done, or for every irreducible axial $g$ with $\lambda_{g,+}\in U$, we have $\lambda_{g,-}=p$.  

Now, suppose that there exists $q\in\boundary G-U\cup\{p\}$.  Then, by Proposition~\ref{prop:stable dense}, and the fact that $\boundary G$ is Hausdorff, we may argue as above, using Proposition~\ref{prop:irreducible axial}, that some irreducible axial element takes $p$ arbitrarily close to $q$, and thus that some power of $g$ takes a translate of $p$ into $U$, as required, unless $p$ is a stable point for \emph{every} irreducible axial element of $G$.  But then $G$ does not contain two independent irreducible axial elements whence, since $G$ acts acylindrically on $\fontact S$ by~\cite[Theorem 14.3]{BehrstockHagenSisto:HHS_I}, a theorem of Osin (see Theorem~\ref{thm:osin} below) implies that $G$ is virtually cyclic.
\end{proof}

\begin{cor}\label{cor:dense_boundary_curve_graph}
If $(G,\mathfrak S)$ is an HHG with an irreducible axial, then $\boundary\fontact S$ is dense in $\boundary G$.
\end{cor}

\begin{rem}\label{rem:rank_one}
In Section~\ref{sec:rank_rigidity}, we investigate the question of when groups of HHS automorphisms contain irreducible axial elements.  In that section, we consider a more general class, so-called ``rank-one'' elements, of which irreducible axial elements are the main examples.
\end{rem}

\section{Coarse semisimplicity in hierarchically hyperbolic groups}\label{sec:semisimple}

\begin{thm}\label{thm:HHGs have no distorteds}
If $(G, \mathfrak S)$ is a hierarchically hyperbolic group, then each $g \in G$ is either elliptic or axial, and $\pi_U(\langle g\rangle)$ is a quasiisometrically embedded copy of $\integers$ for each $U\in\BIG(g)$.
\end{thm}

\begin{proof}[Proof of Theorem~\ref{thm:HHGs have no distorteds}]
This follows from Lemma~\ref{lem:HHG irreducibles are undistorted} and Lemma~\ref{lem:HHG reducibles are undistorted} below.
\end{proof}

Our main tool here is the following result of Bowditch:

\begin{lem}[Lemma 2.2 of \cite{Bowditch:tight}]\label{lem:bowditch_acyl}
If $G$ acts acylindrically by isometries on a hyperbolic space $M$, then each element of $G$ acts either elliptically or loxodomically on $M$. 
\end{lem}

Lemma~\ref{lem:bowditch_acyl} and~\cite[Theorem 14.3]{BehrstockHagenSisto:HHS_I} combine to yield:

\begin{lem}\label{lem:HHG irreducibles are undistorted}
If $g \in G$ is irreducible, then $g$ is either elliptic or axial.
\end{lem}

Recall that for any reducible $g \in G$, we have $\BIG(g) = \{A_i\} \cup \{B_j\}$, where $g$ acts axially on each $\fontact A_i$ and distortedly on each $\fontact B_j$.  It remains to prove:

\begin{lem} \label{lem:HHG reducibles are undistorted}
If $g \in G$ is reducible, then $\{B_j\} = \emptyset$.
\end{lem}

For each $U\in\mathfrak S$, let $G_U=\mathcal A_U\cap G$ be the subgroup of $G$ fixing $U\in\mathfrak S$ and let $\widebar{G}_U=\theta_U(G_U)$, where $\mathcal A_U=\stabilizer_{\Aut(\mathfrak S)}(U)$ and $\theta_U:\mathcal A_U\to\Aut(\mathfrak S_U)$ is the restriction homomorphism.

\begin{lem}\label{lem:stabilizes_FU}
Let $U\in\mathfrak S$.  Then $\widebar{G}_U$ acts acylindrically on $\fontact U$.
\end{lem}

\begin{proof}[Proof of Lemma \ref{lem:stabilizes_FU}]
By definition, $\widebar{G}_U$ acts by automorphisms on the hierarchically hyperbolic space $(F_U,\mathfrak S_U)$.  We first establish:  

\setcounter{claimnum}{0}

\begin{claim}\label{claim:acyl_enable}
For each $R\geq0$, there exists $K=K(R)$ such that any $R$--ball $B\subseteq F_U$ intersects $gB$ for at most $K$ elements $g\in\widebar{G}_U$.
\end{claim}

\renewcommand{\qedsymbol}{$\blacksquare$}

\begin{proof}[Proof of Claim~\ref{claim:acyl_enable}]
Since the inclusion hieromorphism $(F_U,\mathfrak S_U)\to(G,\mathfrak S)$ is a quasiisometric embedding (with constants independent of $U$), it suffices to bound the number of cosets $g(\kernel\theta_U)$ in $G_U$ so that $g(\kernel\theta_U)\cdot (B'\times E_U)=(\bar gB')\times E_U$ intersects $B'\times E_U$, where $B'$ is a ball in $F_U\subset P_U\subset\cuco X$ of radius depending on $R$ and the quasiisometry constants.  Such a bound exists because $G$ acts on itself geometrically.
\end{proof}

\renewcommand{\qedsymbol}{$\Box$}

We now follow the proof of Theorem~14.3 of~\cite{BehrstockHagenSisto:HHS_I}.  Let $\epsilon>0$ be given and let $R\geq1000\epsilon$.  Consider the set $\mathfrak H$ of $g\in \widebar{G}_U$ so that $\dist_U(x,gx),\dist_U(y,gy)<\epsilon$, where $x,y\in F_U$.  Choose $s_0$ as in the distance formula for $(F_U,\mathfrak S_U)$ and, for each $r\geq0$, consider the set $\mathfrak L(r)$ of $\nest$--maximal $V\in\mathfrak S_U-\{U\}$ so that $\dist_V(x,y)>s_0$ and $|\dist_U(x,\rho^V_U)-\frac{R}{2}|<r\epsilon$.  Arguing exactly as in the proof of Theorem~14.3 of~\cite{BehrstockHagenSisto:HHS_I} yields a uniform bound on $|\mathfrak L(11)|$.  We then divide into two cases.  

First, if $\mathfrak L(10)\neq\emptyset$, then we again argue as in the proof of~\cite[Theorem 14.3]{BehrstockHagenSisto:HHS_I}, reaching the conclusion that, if $V\in\mathfrak L(10)$ and $g\in\mathfrak H$, then $\gate_{P_V}(x)$ coarsely coincides with $g\cdot \gate_{P_V}(x)$, from which it follows from Claim~\ref{claim:acyl_enable} that $\mathfrak H$ has uniformly bounded cardinality.  The argument in~\cite{BehrstockHagenSisto:HHS_I} uses only the $\widebar{G}_U$--equivariance of the gate construction and Definition~\ref{defn:space_with_distance_formula} and thus goes through.

Similarly, if $\mathfrak L(10)=\emptyset$, then the argument in~\cite{BehrstockHagenSisto:HHS_I} uses only the existence of hierarchy paths, large links, bounded geodesic image, the distance formula, and a bound on the cardinalities of stabilizers of balls in $F_U$.  The latter comes from Claim~\ref{claim:acyl_enable}, and thus the argument works verbatim in the present context.
\end{proof}

\begin{proof}[Proof of Lemma~\ref{lem:HHG reducibles are undistorted}]
Let $U \in \BIG(g)$.  Let $M>0$ be as in Lemma \ref{lem:Big is fixed} and set $h = g^M$; note that $h \cdot U = U$, i.e. $h\in\mathcal A_U$. Let $h_U=\theta_U(h)\in \widebar{G}_U$.  By Lemma~\ref{lem:stabilizes_FU}, $\widebar{G}_U$ acts acylindrically on $\fontact U$, so by Lemma~\ref{lem:bowditch_acyl}, $h_U$ is either elliptic or loxodromic on $\fontact U$.  Since $U\in\BIG(h)$, it must be the case that $h_U$ is loxodromic on $\fontact U$.  Since $h$ acts like $h_U$ on $\fontact U$, the claim follows.
\end{proof}

\section{Essential structures, essential actions, and product HHS}\label{sec:essential}

\subsection{Product HHS}\label{subsec:product}
It is shown in~\cite{BehrstockHagenSisto:HHS_II} that, if $\cuco X_0,\cuco X_1$ admit hierarchically hyperbolic structures, then $\cuco X_0\times\cuco X_1$ admits a hierarchically hyperbolic structure making the inclusions $\cuco X_i\to\cuco X_0\times\cuco X_1$ into hieromorphisms with hierarchically quasiconvex image.  Rather than recall the construction, we now give a more streamlined (equivalent) definition.

\begin{defn}\label{defn:product_HHS}
Let $(\cuco X,\mathfrak S)$ be a hierarchically hyperbolic space.  Then $(\cuco X,\mathfrak S)$ is a \emph{product HHS} if there exists $K<\infty$ and $U\in\mathfrak S$ such that for all $V\in\mathfrak S$, either $V\nest U$, or $V\orth U$, or $\diam(\fontact V)\leq K$.  If, in addition, for each $n\in\naturals$ there exist $V,W\in\mathfrak S$ with $V\nest U,W\orth U$ and $\diam(\pi_V(\cuco X)),\diam(\pi_W(\cuco X))>n$, then $(\cuco X,\mathfrak S)$ is a \emph{product region with unbounded factors}.  Observe that $(\cuco X,\mathfrak S)$ is a product HHS if and only if there exists $U\in\mathfrak S$ so that $P_U\to\cuco X$ is coarsely surjective, and that $(\cuco X,\mathfrak S)$ is a product region with unbounded factors if in addition $F_U,E_U$ are both unbounded. 
\end{defn}

\subsection{Essential structures and cores}\label{subsec:essential_core}

\begin{defn}[Essential HH sructures]\label{defn:essential}
Let $(\cuco X,\mathfrak S)$ be an HHS and let $G\leq\Aut(\mathfrak S)$.  Then $(\cuco X,\mathfrak S)$ is \emph{$G$--essential} if, for any $G$--invariant hierarchically quasiconvex $\cuco Y\subset\cuco X$, all of $\cuco X$ is contained in some regular neighborhood of $\cuco Y$. 
\end{defn}

\begin{rem}
Compare Definition~\ref{defn:essential} to the definition of a $G$--essential cube complex from~\cite{CapraceSageev:rank_rigidity}, which requires that the cube complex be the cubical convex hull of a $G$--orbit (but actually requires something stronger).
\end{rem}

\begin{prop}[Essential core] \label{prop:essential core}
Let $(\cuco X, \mathfrak S)$ be an HHS and let $G \leq \Aut(\mathfrak S)$ be a subgroup.  Suppose that one of the following holds:

\begin{enumerate}
\item \label{item:ess_1}$G$ acts properly and cocompactly on $\cuco X$ and with finitely many orbits on $\mathfrak S$, i.e. $(G, \mathfrak S)$ is an HHG;
\item \label{item:ess_2}$G$ acts on $\cuco X$ with unbounded orbits and with no fixed point in $\partial \cuco X$.
\end{enumerate}

Then there exists a $G$-invariant, $G$-essential, hierarchically quasiconvex subspace $\cuco Y \subset \cuco X$ so that whichever of~\eqref{item:ess_1} or~\eqref{item:ess_2} held for $G\curvearrowright \cuco X$ holds for the action of $G$ on $\cuco Y$.
\end{prop}

\begin{proof}
If $(\cuco X, \mathfrak S)$ is an HHG, the claim follows immediately with $\cuco Y = \cuco X$.  In the second case, we will build $\cuco Y \subset \cuco X$ so that $\cuco Y$ is hierarchically quasiconvex and $G$-invariant, with the property that if $\cuco Y' \subset \cuco X$ is hierarchically quasiconvex and $G$-invariant, then there exists an $R>0$ such that $\cuco Y \subset \cuco N_R(\cuco Y')$.  Given such a $\cuco Y$, the fact that $G$ does not fix a point in $\partial\cuco Y$ follows from Proposition \ref{prop:hierarchically_quasiconvex_set} and the hypothesis that $G$ does not fix a point in $\boundary\cuco X$.

To construct $\cuco Y$, for each $U\in\mathfrak S$, let $H_U\subseteq\fontact U$ be the union of all geodesics starting and ending in $\pi_U(G\cdot x)$ for some fixed basepoint $x\in\cuco X$.  A thin quadrilateral argument shows that $H_U$ is uniformly quasiconvex.  Let $\cuco Y$ consist of all realization points $y$ with $\pi_U(y)\in H_U$ for all $U\in\mathfrak S$; this subspace is easily seen to have the required properties.
\end{proof}

Recall that, by hierarchical quasiconvexity, $(\cuco Y,\mathfrak S)$ is normalized: for each $U\in\mathfrak S$, the associated hyperbolic space is uniformly quasiisometric to $\pi_U(\cuco Y)\subseteq\fontact U$.

\section{Coarse rank-rigidity and its consequences}\label{sec:rank_rigidity}
Throughout this section, $(\cuco X,\mathfrak S)$ is a hierarchically hyperbolic space with $\cuco X$ proper and $\mathfrak S$ countable; we always let $S$ denote the $\nest$--maximal element of $\mathfrak S$.  
In this section, we consider countable subgroups $G\leq\Aut(\mathfrak S)$ (so that, by the distance formula, $G$ acts discretely on $\cuco X$).  These standing hypotheses cover the case where $(G,\mathfrak S)$ is an HHG.  We emphasize our standing assumption that all HHS are normalized.

\begin{defn}[Rank-one automorphism]\label{defn:rank_one_automorphism}
The automorphism $g\in\Aut(\mathfrak S)$ is \emph{rank-one (on $(\cuco X,\mathfrak S)$)} if:
\begin{itemize}
 \item $g$ is axial;
 \item $|\BIG(g)|=1$;
 \item if $U\in\mathfrak S$ is orthogonal to the domain in $\BIG(g)$, then $\diam(\pi_U(\cuco X))<\infty$.
\end{itemize}
Irreducible axial elements are rank-one.
\end{defn}

Our first goal is to show that, under the above hypotheses, either $G$ contains an irreducible axial element or the $G$--essential core of $\cuco X$ is a product HHS (not necessarily with unbounded factors).  This is done in Section~\ref{subsec:irred_axial_or_fixed_domain}, using tools from Sections~\ref{subsec:finite_orbits},\ref{subsec:finding_product_structures},\ref{subsec:finding_irreducible_axial}.  In Section~\ref{subsec:actual_rank_rigidity}, we apply results of Section~\ref{subsec:irred_axial_or_fixed_domain}.

\subsection{Irreducible axials or fixed domains}\label{subsec:irred_axial_or_fixed_domain}
We now prove the following two parallel propositions (one covering the non-parabolic case, and one covering the HHG case):

\begin{prop}\label{prop:alternative_non_parabolic_version}
Let $(\cuco X,\mathfrak S)$ be an HHS with $\cuco X$ proper and $\mathfrak S$ countable.  Let the countable group $G\leq\Aut(\mathfrak S)$ act with unbounded orbits in $\cuco X$ and without a global fixed point in $\boundary\fontact S$.  Then either $G$ contains an irreducible axial element, or there exists $U\in\mathfrak S-\{S\}$ so that $|G\cdot U|<\infty$.  Moreover, any $G$--essential hierarchically quasiconvex subspace $\cuco Y\subset\cuco X$ coarsely coincides with the standard product region $P_U\cap\cuco Y$.
\end{prop}

\begin{proof}
By Proposition~\ref{prop:essential core}, there exists a $G$--invariant hierarchically quasiconvex subspace $\cuco Y$, with a hierarchically hyperbolic structure $(\cuco Y,\mathfrak S)$ admitting a $G$--equivariant hieromorphism $(\cuco Y,\mathfrak S)\to(\cuco X,\mathfrak S)$ that is the inclusion on $\cuco Y$ and the identity on $\mathfrak S$, and so that $(\cuco Y,\mathfrak S)$ is $G$--essential.  Moreover, $G$ continues to act without a global fixed point in $\boundary\fontact S$.  Hence, since $\cuco Y$ is proper and $\mathfrak S$ is countable, Proposition~\ref{prop:finding_irreducible_axial} provides an irreducible axial isometry of $(\cuco Y,\mathfrak S)$ (hence of $(\cuco X,\mathfrak S)$) unless $\diam(\pi_S(\cuco Y))<\infty$.  If $\diam(\pi_S(\cuco Y))<\infty$, then Proposition~\ref{prop:finding_product_structures} completes the proof.
\end{proof}

The HHG version requires the following theorem of Osin, which we also use elsewhere:

\begin{thm}[Theorem 1.1 of \cite{Osin:acylindrically_hyperbolic}] \label{thm:osin}
Let $G$ be a group acting acylindrically on a hyperbolic space.  Then exactly one of the following holds:
\begin{enumerate}
\item $G$ has bounded orbits;
\item $G$ is virtually infinite cyclic and contains a loxodromic element;\label{item:cyclic}
\item $G$ contains infinitely many independent loxodromic elements.
\end{enumerate}
\end{thm}

\begin{prop}\label{prop:alternative_HHG_version}
Let $(G,\mathfrak S)$ be an HHG.  Then either $G$ contains an irreducible axial element or there exists $U\in\mathfrak S$ such that $|G\cdot U|<\infty$ and $G$ coarsely coincides with $P_U$.
\end{prop}

\begin{proof}
The $G$--action on $(G,\mathfrak S)$ is essential.  If $\diam(\fontact S)=\infty$, then, since $G$ acts acylindrically on $\fontact S$, as proved in~\cite[Section~14]{BehrstockHagenSisto:HHS_I}, Theorem~\ref{thm:osin} implies that $G$ contains an irreducible axial element.  Hence we can assume that $\diam(\fontact S)<\infty$, and in particular that $G$ has no fixed point in $\boundary\fontact S=\emptyset$.  The claim now follows from Proposition~\ref{prop:finding_product_structures}.
\end{proof}

\subsection{Finding finite orbits in $\mathfrak S$}\label{subsec:finite_orbits}
Let $\mu$ be a probability measure on $G$, whose support generates $G$.  All spaces are equipped with their Borel $\sigma$--algebra, so every subset of  $G$ is measurable, while the measurable subsets of $\overline{\cuco X}$ are determined by Definition~\ref{defn:boundary_topology}.

\begin{lem}[Stationary measure on $\overline{\cuco X}$]\label{lem:stationary_measure_exists}
There exists a $\mu$--stationary probability measure $\nu$ on $\overline{\cuco X}$, i.e. for all $\nu$--measurable $E\subseteq\overline{\cuco X}$, $$\nu(E)=\sum_{g\in G}\mu(g)\nu(g^{-1}E)=\mu*\nu(E).$$
\end{lem}

\begin{proof}
This is a standard fact, relying on compactness of $\overline{\cuco X}$, i.e. Theorem~\ref{thm:cpt}.  See e.g.~\cite[Lemma~1.2]{Furstenberg}.
\end{proof}

\begin{rem}[Sampling $\cuco X$]\label{rem:sampling}
Since our aim in this section is to establish that, after passing if necessary to a $G$--essential core, $G$ contains an irreducible axial element or $\cuco X$ is a product HHS, and these properties are insensitive to modifications of $\cuco X$ within its quasiisometry type, we now ``discretize'' $\cuco X$, for convenience in the proof of Lemma~\ref{lem:fixed_point}.

Let $\cuco D=G\backslash\cuco X$, and let $\bar\dist$ be the quotient pseudometric, so $(\cuco D,\bar\dist)$ is proper since $\cuco X$ is proper.  Hence there exists $\epsilon>0$ and a countable set $\{\bar x_n\}_{n\geq0}$ in $\cuco D$ such that $\neb_{\epsilon}^{\cuco D}(\{\bar x_n\})=\cuco D$.  Thus $\cuco X$ contains a countable, $G$--invariant set $\{x_n\}_{n\geq0}$ for which the inclusion $\{x_n\}\hookrightarrow\cuco X$ is a quasiisometry, and we replace $\cuco X$ with $\{x_n\}$.  We can thus assume that $\cuco X$ is countable.
\end{rem}

\begin{lem}\label{lem:support_measurable}
For each $\mathcal U\subset\mathfrak S$, the set $\{p\in\boundary\cuco X:\support(p)=\mathcal U\}$ is $\nu$--measurable.
\end{lem}

\begin{proof}
Either $\{p\in\boundary\cuco X:\support(p)=\mathcal U\}=\emptyset$, in which case we're done, or $\mathcal U=\{U_i\}$ is a set of pairwise-orthogonal domains.  Let $\mathcal X_0$ be the set of points $q\in\boundary\cuco X$ so that, for all $V\in\support(q)$, there exists $U\in\mathcal U$ with $V\nest U$.  Note that $\mathcal Y=\{p\in\boundary\cuco X:\support(p)=\mathcal U\}\subseteq\mathcal X_0$.  Let $\mathcal X_1$ be the subset of $\mathcal X_0$ consisting of those $q\in\mathcal X_0$ such that for some $V\in\support(q)$, we have $V\not\in\mathcal U$ (so $V$ is properly nested in some $U\in\mathcal U$ and orthogonal to the remaining elements).  

\textbf{$\mathcal X_0$ is closed in $\overline{\cuco X}$:}  We will check that for any sequence $\{q_n\}$ with each $q_n\in\mathcal X_0$, if $q_n\to q$, then $q\in\mathcal X_0$.  Suppose not, i.e. suppose that there exists $V\in\support(q)$ so that $V\not\nest U$ for all $U\in\mathcal U$.  Consider a basic neighborhood $\neb=\neb_{\epsilon,\{N_T\}}(q)$ of $q$.  There are two cases.

\emph{First case:} This is the case where there exists $U\in\mathcal U$ so that $U\transverse V$ or $U\propnest V$ and, for infinitely many $n$, there exists $W\in\support(q_n)$ so that $W\nest U$ and $W\north V$.  Let $\mathcal I$ be the set of such $n$.

First, suppose that $q_n$ is remote with respect to $q$.  Suppose that the basic neighborhood $\neb$ has been chosen so that $N_V$ does not meet the $10^9E$--neighborhood of $\rho^U_V$.  Then for arbitrarily large $n\in\mathcal I$, the subsets $\rho^U_V,\rho^W_V$ coarsely coincide, and hence $(\boundary\pi_{\support(q)}(q_n))_V=\rho^W_V$ does not lie in $N_V$.  It follows that for arbitrarily large $n\in\mathcal I$, we have $q_n\not\in\neb$, by the definition of the remote part of a basic set.  This is a contradiction.

Second, suppose that $q_n$ is non-remote with respect to $q$, where $n\in\mathcal I$.  Exactly as before, suppose that $N_V$ does not meet the $10^9E$--neighborhood of $\rho^U_V$ (which is still defined by assumption).  We still have that $\rho^W_V$ is defined and coarsely coincides with $\rho^U_V$, for some $W\in\support(q_n)$, by assumption. Hence, again, we have that $(\boundary\pi_{\support(q)}(q_n))_V=\rho^W_V$ does not lie in $N_V$.  From the final condition in the definition of the non-remote part of a basic set, it follows that $q_n\not\in\neb$, which is again a contradiction.

\emph{Second case:} In this case, for all but finitely many $n$, we have $V\orth W$ for all $W\in\support(q_n)$.  
The point $q_n$ is non-remote with respect to $q$.  Indeed, there exists $V\in\support(q)$ which is orthogonal to every element of $\support(q_n)$.  In particular $V\in\support(q)-\support(q_n)\cap\support(q)$.  Now, $\sum_{T\in\support(q_n)-\support(q)}a^{q_n}_T<\epsilon$, so $$\sum_{T\in\support(q_n)\cap\support(q)}a^{q_n}_T>1-\epsilon,$$ while $|a^q_T-a^{q_n}_T|<\epsilon$ whenever $T\in\support(q_n)\cap\support(q)$.  Hence $$\sum_{T\in\support(q)\cap\support(q_n)}a^q_T>1-\epsilon\left(|\support(q_n)\cap\support(q)|\right),$$ which is impossible when $\epsilon$ is sufficiently small compared to $a^q_V$, since $V\not\in\support(q_n)$.  Hence $q_n\not\in\neb$, a contradiction.

\textbf{Conclusion:}  Let $\mathfrak T$ be the set of support sets $\mathcal V\neq\mathcal U$ such that for each $V\in\mathcal V$, there exists $U\in\mathcal U$ with $V\nest U$.  Then $\mathfrak T$ is countable, being a set of finite subsets of the countable set $\mathfrak S$.  Now, $\mathcal X_1$ is the union over all $\mathcal V\in\mathfrak T$ of the set $\mathcal X_0(\mathcal V)$ of $q\in\boundary\cuco X$ such that for each $W\in\support(q)$, there exists $V\in\mathcal V$ with $W\nest V$.  Hence, by the previous part of the proof, $\mathcal X_1$ is a countable union of closed sets.  Thus $\mathcal Y=\mathcal X_0-\mathcal X_1$ is Borel, and hence $\nu$--measurable.  
\end{proof}

\begin{lem}\label{lem:fixed_point}
If $G$ has no finite orbit in $(\mathfrak S-\{S\})\cup\boundary\fontact S$, then $\nu$ is supported on $\boundary\fontact S\subset\overline{\cuco X}$.
\end{lem}

\begin{proof}
Let $D$ be the set of finite subsets of $\mathfrak S$, so that $D$ is countable and $G$ acts on $D$ in the obvious way.  By construction, $\{S\}$ and $\emptyset$ are the only elements of $D$ whose $G$--orbits are finite.  
We first define a map $\mathcal O:\overline{\cuco X}\to D$.  Note that if $\mathfrak S=\{S\}$, then $\boundary\cuco X=\boundary\fontact S$, and the claim follows, so we assume that there exists $U\propnest S$.

\textbf{Defining $\mathcal O$ on boundary points:} For each $p\in\boundary\cuco X$, let $\mathcal O(p)=\supp(p)$.  Observe that this assignment is $G$--equivariant and that $\mathcal O(p)=\{S\}$ if and only if $p\in\boundary\fontact S$.  

\textbf{Defining $\mathcal O$ on interior points:}  Let $\cuco B\subset\cuco X$ contain exactly one point from each $G$--orbit, and choose $F\in D-\{\{S\},\emptyset\}$.  For each $x\in\cuco B$, let $\mathcal O(x)=F$.  Then, for any $x\in\cuco B$ and $g\in G$, let $\mathcal O(gx)=gF$.  Then $\mathcal O$ is $G$--equivariant and, for all $x\in\cuco X$, the nonempty finite set $\mathcal O(x)$ differs from $\{S\}$.  For any $F'\in D$, either $\mathcal O^{-1}(F')=\emptyset$ or $F'=gF$ for some $g\in G$.  Hence, for any subset $D'$ of $D$, we can write $\mathcal O^{-1}(D')=\bigcup_{gF\in D'}g\cuco B$.  It follows that $\mathcal O^{-1}(D')$ is a countable union of translates of $\cuco B$, which is a countable union of closed sets (singletons) by Remark~\ref{rem:sampling}, and thus $\mathcal O^{-1}(D')$ is Borel.

\textbf{Measurability of $\overline{\cuco X}-\boundary\fontact S$:}  Since $\boundary\fontact S=\{p\in\boundary\cuco X:\support(p)=\{S\}\}$, it follows from Lemma~\ref{lem:support_measurable} that $\overline{\cuco X}-\boundary\fontact S$ is measurable.

\textbf{Measurability of $\mathcal O$:}  There is a probability measure $\tilde\nu$ on $D$ given by $\tilde \nu(A)=\nu(\mathcal O^{-1}(A))$, for each $A\subseteq D$.  A set $\mathcal O^{-1}(A)$ decomposes as:
$$\left\{x\in\cuco X:\mathcal O(x)\in A\right\}\cup\left\{p\in\boundary\cuco X:\support(p)\in A\right\}.$$
The set $\left\{p\in\boundary\cuco X:\support(p)\in A\right\}=\bigcup_{\mathcal U\in A}\left\{p:\support(p)=\mathcal U\right\}$, which is $\nu$-measurable by Lemma~\ref{lem:support_measurable}.  Since $A\subseteq D$ is countable, it suffices to show that $\mathcal O^{-1}(F)\cap\cuco X$ is Borel for each $F\in D$, but this was established above. 

\textbf{Conclusion:}  We have that $\mathcal O:\overline{\cuco X}\to D$ is a measurable $G$--equivariant map.   Since $G$ preserves $\boundary\fontact S$, it follows that $\overline{\cuco X}-\boundary\fontact S$ is a $G$--invariant $\nu$--measurable set.

Suppose that $F'\in D$ has the property that $G\cdot F'$ is finite.  Then $G \cdot U$ is a finite $G$--invariant subset of $\mathfrak S$ for each $U\in F'$ and, by our hypothesis that there is no finite $G$--orbit in $\mathfrak S-\{S\}$, we have that $F'=\{S\}$.  Since $\mathcal O(e)\neq\{S\}$ for all $e\in \overline{\cuco X}-\boundary\fontact S$, it follows that $\mathcal O( \overline{\cuco X}-\boundary\fontact S)$ does not contain a finite $G$--orbit.  As shown in e.g~\cite{Ballmann:dirichlet},\cite[Lemma~2.2.2]{KaimanovichMasur:poisson},\cite[Lemma~3.4]{Woess:boundaries}\cite[Lemma~3.3]{Horbez:handelmosher}, we must have $\nu(\overline{\cuco X}-\boundary\fontact S)=0$.
\end{proof}

\begin{cor}\label{cor:finite_curve_graph_stabilized}
If $\diam(\fontact S)<\infty$, then $G$ stabilizes a finite subset of $\mathfrak S-\{S\}$.
\end{cor}

\begin{proof}
By hypothesis, $\boundary\fontact S=\emptyset$, so $\nu$ cannot be supported on $\boundary\fontact S$.  Hence $G$ has a finite orbit in $\mathfrak S\cup\boundary\fontact S$ by Lemma~\ref{lem:fixed_point} and thus $G$ must have a finite orbit in $\mathfrak S-\{S\}$.
\end{proof}

\subsection{Finding product structures when $\diam(\fontact S)<\infty$}\label{subsec:finding_product_structures}

\begin{prop}\label{prop:finding_product_structures}
Suppose $G\leq\Aut(\mathfrak S)$ is a countable subgroup, where $\diam(\fontact S)<\infty$.  Then there exists $U\in\mathfrak S-\{S\}$ and a finite-index subgroup $G'$ such that $G'\cdot U=U$ and $\cuco X$ coarsely coincides with $P_U$.  Hence either $(\cuco X,\mathfrak S)$ is a product HHS with unbounded factors or $\cuco X$ coarsely coincides with $F_U$ or $E_U$.
\end{prop}

\begin{proof}
By Corollary~\ref{cor:finite_curve_graph_stabilized}, there exists $U\in\mathfrak S-\{S\}$ and a finite-index subgroup $G'\leq G$ so that $G'\cdot U=U$.  Note that $G'$ continues to act essentially on $(\cuco X,\mathfrak S)$, coarsely stabilizing $P_U$.  Since $P_U$ is hierarchically quasiconvex, $\cuco X$ coarsely equals $P_U$ by essentiality.  The last assertion is immediate.
\end{proof}

\subsection{Finding irreducible axial elements when $\diam(\fontact S)=\infty$}\label{subsec:finding_irreducible_axial}

\begin{prop}\label{prop:finding_irreducible_axial}
Let $(\cuco X,\mathfrak S)$ be a hierarchically hyperbolic space.  Let $G\leq\Aut(\mathfrak S)$ act essentially and suppose that $G$ acts on $\cuco X$ with no global fixed point in $\boundary\fontact S$ and that $\fontact S$ is unbounded.  Then $G$ contains an irreducible axial automorphism of $(\cuco X,\mathfrak S)$.
\end{prop}

\begin{proof}
Suppose that every orbit of $G$ in $\fontact S$ is bounded, so that, fixing $x_0\in\cuco X$, there exist $Q,R<\infty$ so that $\diam_S(G\cdot\pi_S(x_0))\leq R$ and $G\cdot\pi_S(x_0)$ is $Q$--quasiconvex.  Consider the set of all $E$--consistent tuples $(b_U)_{U\in\mathfrak S}$ such that $b_S\in G\cdot \pi_S(x_0)$.  Let $\cuco Y$ be the set of realization points in $\cuco X$ corresponding to such tuples, provided by Theorem~\ref{thm:realization}, and note that $G$ acts on $\cuco Y$.  By definition, $\cuco Y$ is hierarchically quasiconvex in $\cuco X$ provided $\pi_U(\cuco Y)$ is uniformly quasiconvex in $\fontact U$ for each $U\in\mathfrak S$, which we now verify. 

If $\tup b$ is such a tuple, with $\dist_S(b_S,\rho^U_S)\leq E$, then consistency puts no constraint on the $U$--coordinate of $\tup b$, i.e. for any such $U$, the map $\pi_U:\cuco Y\to\fontact U$ is uniformly coarsely surjective, and in particular $\pi_U(\cuco Y)$ is uniformly quasiconvex in $\fontact U$.  On the other hand, if $\dist_S(\rho^U_S,G\cdot\pi_S(x_0))>E$, then consistency and bounded geodesic image imply that $\pi_U(\cuco Y)$ is uniformly bounded, and hence uniformly quasiconvex.

The existence of $\cuco Y$ contradicts $G$--essentiality of $\cuco X$.  Hence $G$ has an unbounded orbit in $\fontact S$, so either there exists $g\in G$ acting loxodromically on $\fontact S$, so $g$ is irreducible axial, or there exists a unique fixed point $p\in\boundary\fontact S$, which is impossible.
\end{proof}

\subsection{Coarse rank-rigidity}\label{subsec:actual_rank_rigidity}
Recall that a metric space $\cuco X$ is \emph{wide} if no asymptotic cone of $\cuco X$ has a cut-point.  The following lemma is well-known and elementary:

\begin{lem}\label{lem:wide}
Let $\cuco X$ be a metric space quasiisometric to the product $\cuco X_0\times\cuco X_1$, where each $\cuco X_i$ is unbounded.  Then $\cuco X$ is wide, i.e. no asymptotic cone of $\cuco X$ has a cut-point.
\end{lem}

We now prove the main theorems of this section.  Much of the work was done in proving Propositions~\ref{prop:alternative_non_parabolic_version} and~\ref{prop:alternative_HHG_version}; the remaining work is largely in sorting out technical issues that arise when attempting to induct on complexity; these issues mainly stem from the fact that, given $U\in\mathfrak S$, the induced HHS structure on $E_U$ does not have a uniquely-determined $\nest$--maximal element.

\begin{thm}[Coarse rank-rigidity for non-parabolic actions]\label{thm:rank_rigidity_non_parabolic}
Let $(\cuco X,\mathfrak S)$ be an HHS with $\cuco X$ proper and $\mathfrak S$ countable.  Let the countable group $G\leq\Aut(\mathfrak S)$ act essentially with unbounded orbits in $\cuco X$ and without a fixed point in $\boundary(\cuco X,\mathfrak S)$.  Then one of the following holds:
\begin{enumerate}
 \item \label{item:rr_non_par_product}$\cuco X$ is a product HHS with unbounded factors; specifically, $\cuco X$ is coarsely equal to $P_U$ for some $U\in\mathfrak S$ with $|GU|<\infty$;
 \item \label{item:rr_non_par_irr_ax}there exists $g\in G$ such that $g$ is rank-one.
\end{enumerate}
If conclusion~\eqref{item:rr_non_par_product} holds, then $\cuco X$ is wide.
\end{thm}

\begin{proof}
By Proposition~\ref{prop:alternative_non_parabolic_version}, either $G$ contains an irreducible axial element, which is rank-one by definition, so conclusion~\ref{item:rr_non_par_irr_ax} holds, or there is a finite-index subgroup $G'\leq G$ fixing some $U\in\mathfrak S-\{S\}$, so that by essentiality, $\cuco X$ coarsely coincides with the standard product region $P_U$.  This implies that $\cuco X$ is a product HHS.  Choose $U$ of minimal level with this property, i.e. no domain of lower level has a finite $G$--orbit in $\mathfrak S$.

Since $G$ has unbounded orbits in $\cuco X$, at least one of $E_U,F_U$ is unbounded.  If $F_U,E_U$ are both unbounded, then conclusion~\ref{item:rr_non_par_product} holds, and we are done.  The statement about wideness follows from Lemma~\ref{lem:wide}.

If $F_U$ is unbounded and $E_U$ is bounded, then $(F_U,\mathfrak S_U)$ is a HHS with $F_U$ proper and $\mathfrak S_U$ countable, on which $G'$ acts by HHS automorphisms with no fixed point in $\boundary\fontact U$ (for otherwise $G$ would have a fixed point in $\boundary\cuco X$).  By minimality, $G'$ has no finite orbit in $\mathfrak S_U-\{U\}$, so Proposition~\ref{prop:finding_irreducible_axial} provides $g\in G'$ acting as an irreducible axial element of $\Aut(\mathfrak S_U)$.  As an element of $\Aut(\mathfrak S)$, we see that $g$ is rank-one, for otherwise there would be some $V\orth U$ with $\diam(\fontact V)=\infty$, contradicting that $E_U$ is bounded.

Finally, suppose that $E_U$ is unbounded and $F_U$ is bounded.  Let $\mathfrak C$ be a minimal $G'$--invariant set of $\nest$--minimal elements $C$ of $\mathfrak S-\{S\}$ such that $W\nest C$ whenever $W\orth U$.  

Suppose that there exists $C\in\mathfrak C$ with $C\orth U$.  Then $g\cdot C\orth g\cdot U=U$ for all $g\in G'$, so $g\cdot C\nest C$, from which it follows that (passing if necessary to a further finite-index subgroup if necessary) $G'\cdot C=C$.  Then $(E_U,\mathfrak S_C)$ is an HHS satisfying the hypotheses of the theorem, and $G'\leq\Aut(\mathfrak S_C)$ acts without a fixed point in $\boundary E_U$ (since it stabilizes $\boundary E_U\subset\boundary\cuco X$).  In this case, the claim follows by induction on complexity.  Indeed, in the base case, $|\mathfrak S|=1$ and the theorem is obvious.  Otherwise, induction shows that either conclusion~\eqref{item:rr_non_par_product} holds, or there exists $g\in G$ that acts as a rank-one element of $\Aut(\mathfrak S_C)$.  Since $G'$ preserves $P_U$ and $P_U$ coarsely equals $\cuco X$, this implies that $g$ is rank-one on $(\cuco X,\mathfrak S)$, as required.

The definition of $\mathfrak C$, and Definition~\ref{defn:space_with_distance_formula}.\eqref{item:dfs_orthogonal}, imply that $C\not\nest U$ and $U\not\nest C$ for all $C\in\mathfrak C$.  Hence it remains to consider the case where each $C\in\mathfrak C$ satisfies $C\transverse U$; fix such a $C$.  Since $G'$ stabilizes $U$, it coarsely stabilizes the image $\overline P_U$ of $P_U=F_U\times E_U\to\cuco X$.  In other words, for any basepoint $x\in\cuco X$, the orbit $G'\cdot x$ lies in a neighborhood of $\overline P_U$.  Now, since $C\transverse U$, the definition of $P_U$ implies that $\pi_C(gx)$ uniformly coarsely coincides with $\rho^U_C$ for all $g\in G'$, whence $\diam(\pi_C(G'\cdot x))<\infty$ so, by essentiality, $\diam(\pi_C(\cuco X))<\infty$.
%

In this case, form a new index set $\mathfrak S_U^\orth$ by appending to the set of domains orthogonal to $U$ a new domain $\mathbf C$.  In $\mathfrak S_U^\orth\cap\mathfrak S$, the associated hyperbolic spaces, projections from $E_U$, and relative projections are defined as in $\mathfrak S$.  The hyperbolic space $\fontact \mathbf C$ is a single point, so the projections $\pi_{\mathbf C}:\cuco X\to\fontact\mathbf C$ and $\rho^V_{\mathbf C}$, for $V\orth U$, are defined in an obvious way.  We thus have an HHS structure $(E_U,\mathfrak S_U^\orth)$ with $G'\leq\Aut(\mathfrak S_U^\orth)$, of complexity less than that of $\mathfrak S$, and we can argue as above by induction.  Observe that, if $g\in\Aut(\mathfrak S_U^\orth)$ is rank-one on $E_U$, then $\BIG(g)$ consists of some element of $\mathfrak S_U^\orth\cap\mathfrak S$, and since $\pi_C(\cuco X)$ is bounded for all $C \in \mathfrak C$, and we can argue as above that $g$ is rank-one on $(\cuco X,\mathfrak S)$.
\end{proof}

\begin{thm}[Coarse rank-rigidity for HHG]\label{thm:rank_rigidity_HHG}
Let $(G,\mathfrak S)$ be an infinite hierarchically hyperbolic group.  Then exactly one of the following holds:
\begin{enumerate}
 \item \label{item:rr_HHG_product}$(G,\mathfrak S)$ is a product HHS with unbounded factors, and $G$ is wide;
 \item \label{item:rr_HHG_irr_ax}$G$ contains a rank-one element, and is thus not wide.
\end{enumerate}
Moreover, conclusion~\eqref{item:rr_HHG_product} holds if and only if $\diam(\fontact S)<\infty$.
\end{thm}

\begin{proof}
By Proposition~\ref{prop:alternative_HHG_version}, either $G$ contains an irreducible axial element, which is rank-one, or there exists $U\in\mathfrak S-\{S\}$ with $G'\cdot U=U$ for some finite-index $G'\leq G$, and $G$ coarsely coincides with $P_U$.  In the latter case, we argue as in the proof of Theorem~\ref{thm:rank_rigidity_non_parabolic}, by induction on complexity, using the following observation: if $V\in\mathfrak S-\{S\}$ and a finite-index subgroup $G'$ fixes $V$, then the action of $G'$ on $F_V$ is proper and cobounded.  Moreover, $G'$ acts with finitely many orbits on $\mathfrak S_V$, so $(G',\mathfrak S_V)$ is an HHG structure on $G'$, enabling induction.
\end{proof}

\subsection{Tits alternative for HHGs}\label{subsec:free_subgroups}

The goal of this subsection is the following theorem:

\begin{thm}[Tits alternative for HHGs]\label{thm:tits}
Let $(G, \mathfrak S)$ be an HHG and let $H\leq G$.  Then $H$ either contains a nonabelian free group or is virtually abelian.
\end{thm}

Before we proceed with the proof, we need some supporting results:

\begin{prop}\label{prop:irred_axial_schottky}
Let $(G,\mathfrak S)$ be a hierarchically hyperbolic group.  Then any $H\leq G$ containing an irreducible axial element is virtually $\integers$ or contains a nonabelian free group.
\end{prop}

\begin{proof}
Since $G$ acts on $\fontact S$ acylindrically~\cite{BehrstockHagenSisto:HHS_I}, and hence $H\leq G$ does, Theorem~\ref{thm:osin} implies that either $H$ is virtually cyclic or $H$ contains irreducible axial elements $g,h$ so that $\{h^\pm\}\cap\{g^\pm\}=\emptyset$. Proposition~\ref{prop:irreducible axial} and Proposition~\ref{prop:properties}.\eqref{item:hausdorff} enable an application of the ping-pong lemma, showing that $g^N,h^N$ freely generate a free subgroup $F$, for some $N>0$.  Or, one can apply Corollary~14.6 of~\cite{BehrstockHagenSisto:HHS_I}, which uses Proposition~2.4 of~\cite{Fujiwara:acyl}.  
\end{proof}

\begin{lem}\label{lem:bounded_orbits_fixed_point}
Let $(G,\mathfrak S)$ be an HHG with $S\in\mathfrak S$ $\nest$--maximal.  Suppose that $H\leq G$ has bounded orbits in $\fontact S$ and fixes some $p\in\boundary\fontact S$. Then $|H|<\infty$.
\end{lem}

\begin{proof}
By Theorem~14.3 of~\cite{BehrstockHagenSisto:HHS_I}, $G$ acts acylindrically on $\fontact S$, i.e. for each $\epsilon>0$, there exists $R\geq0$ and $N\in\naturals$ so that whenever $s,s'\in\fontact S$ satisfy $\dist_S(s,s')\geq R$, there are at most $N$ elements $g\in G$ for which $\dist_S(s,g\cdot s),\dist_S(s',g\cdot s')\leq\epsilon$.  

Fix $s\in\fontact S$ and let $\epsilon_1$ bound the diameter of the orbit $H\cdot s$.  Let $\gamma$ be a $(1,20\delta)$--quasigeodesic ray with endpoint $p$ and initial point $s$, where $\fontact S$ is $\delta$--hyperbolic.  Then, for all $h\in H$, the ray $h\cdot \gamma$ emanates from $h\cdot s$ and has endpoint $h\cdot p=p$.  This fact, together with a thin quadrilateral argument, shows that there exists $k=k(\delta)$ and $R_0$ such that for all $h\in H$, we have $\dist_S(t,h\cdot t)\leq k\delta$ whenever $t\in\gamma$ satisfies $\dist_S(s,t)\geq R_0$.  Let $\epsilon=\max\{\epsilon_1,k\delta\}$ and let $R,N$ be the associated constants coming from acylindricity.  Then we can choose $t\in\gamma$ so that $\dist_S(s,t)>R$ while $\dist_S(s,h\cdot s),\dist_S(t,h\cdot t)\leq\epsilon$ for all $h\in H$, and hence $|H|\leq N$.
\end{proof}

\begin{proof}[Proof of Theorem \ref{thm:tits}]
Note that $H$ is a countable subgroup of $\Aut(\mathfrak S)$, since $G$ is finitely generated.  
We divide into cases, according to whether $H$ fixes some $p\in\boundary G$.

\textbf{$H$ fixes $p\in\boundary\fontact S$:}  In this case, by Proposition~\ref{prop:irred_axial_schottky}, $H$ is either virtually cyclic, contains a nonabelian free group, or, by Theorem~\ref{thm:osin}, $H$ has a bounded orbit in $\fontact S$.  Lemma~\ref{lem:bounded_orbits_fixed_point} implies that $H$ is finite in the latter case.

\textbf{$H$ has no fixed boundary point:}  Suppose there is an irreducible axial $g\in H$. Then either $H$ contains a nonabelian free-group or $H$ is virtually $\integers$, by Proposition~\ref{prop:irred_axial_schottky}.

Otherwise, Proposition~\ref{prop:alternative_non_parabolic_version} provides $U\in\mathfrak S-\{S\}$ such that $H\cdot U$ is finite and the $H$--essential core $\cuco Y$ of in $G$ coarsely coincides with $P_U\cap\cuco Y$.  By replacing $H$ with a finite-index subgroup if necessary, we can assume that $H\cdot U=U$.  

Thus we have an $H$--essential product HHS $(\cuco X_0\times\cuco X_1,\mathfrak S^\times)$ with $H\leq\Aut(\mathfrak S^\times)$ acting on $\cuco X_0\times\cuco X_1$.  Here $\mathfrak S^\times$ consists of two disjoint subsets $\mathfrak S_0,\mathfrak S_1$, together with various domains whose associated spaces are uniformly bounded, with the property that $U_0\orth U_1$ for all $U_0\in\mathfrak S_0,U_1\in\mathfrak S_1$ and each $\mathfrak S_i$ gives $\cuco X_i$ an HHS structure (for more on product decompositions, see~\cite{BehrstockHagenSisto:HHS_II}).   Let $H_i\leq H$ be the stabilizer of some (hence any) parallel copy of $\cuco X_i$.

Observe that $H_i\leq\Aut(\mathfrak S_i)$ is an action on an HHS of strictly lower complexity, for $i\in\{0,1\}$, namely $(\cuco X_i,\mathfrak S_i)$.  If $H_i$ contains no irreducible axial element, then $\cuco X_i$ decomposes as a product HHS, by Theorem~\ref{thm:rank_rigidity_non_parabolic}.  Otherwise, applying Lemma~\ref{lem:stabilizes_FU} and Theorem~\ref{thm:osin}, we see that either $H_0$ or $H_1$ (hence $H$) contains a nonabelian free group, or $H_i$ is virtually $\integers$ for $i\in\{0,1\}$. Hence, either $H$ contains a nonabelian free subgroup, or by induction on complexity, we have a product HHS $(\prod_jL_j^i,\mathfrak S_i)$ such that $H_i\leq\Aut(\mathfrak S_i)$ and each $L_j^i\cong_{q.i.}\reals$.  In the latter case, we conclude that $H$ virtually acts geometrically by HHS automorphisms on $(\prod_{ij}L^i_j,\mathfrak S^\times)$.  Hence, for some $n$, a finite-index subgroup of $H$ acts by uniform quasi-isometries on $\reals^n$, so $H$ is virtually abelian.

\textbf{$H$ fixes $p\in\boundary G-\boundary\fontact S$:}  In this case, $H$ has a finite-index subgroup fixing some $U\in\support(p)$ (so $U\propnest S$).  We now argue by induction on complexity as above.
\end{proof}

\subsection{The ``omnibus subgroup theorem''}\label{subsec:omnibus}
Our next result generalizes the Handel-Mosher ``omnibus subgroup theorem'' from ~\cite{HandelMosher:omnibus}.  Theorem \ref{thm:ost hhg} below implies the omnibus subgroup theorem in the case where $\cuco X$ is the mapping class group of a connected, oriented surface of finite type.  In order to state the theorem, we need to restrict the class of HHS we consider, and give some definitions.

\begin{defn}[Hierarchical acylindricity]
Given an HHS $(\cuco X, \mathfrak S)$, we say $G \leq \Aut(\mathfrak S)$ is \emph{hierarchically acylindrical} if, for each $U \in \mathfrak S$, the image of $G\cap\mathcal A_U$ under the restriction homomorphism $\theta_U:\mathcal A_U \rightarrow \Aut(\mathfrak{S}_U)$ acts acylindrically on $\fontact U$.
\end{defn}

Lemma \ref{lem:stabilizes_FU} implies that every group of automorphisms of an HHG is hierarchically acylindrical.  Moreover, hierarchical acylindricity passes to subgroups.  For the rest of this subsection, fix $G \leq \Aut(\mathfrak S)$ to be hierarchically acylindrical.

\begin{defn}[Active domains]\label{defn:active domain}
Let $G \leq \Aut(\mathfrak S)$ be a group of HHS automorphisms.  We say $U \in \mathfrak S$ is an \emph{active domain} for $G$ if $\diam_U(\pi_U(G \cdot x))$ is unbounded for some (hence any) $x \in \cuco X$.  Let $\Active(G)$ be the set of $\nest$-maximal active domains for $G$.  Note that if $G = \langle g \rangle$, then $\Active(G)= \BIG(g)$.
\end{defn}

\begin{thm}[Omnibus Subgroup Theorem]\label{thm:ost hhg}
Let $(\cuco X,\mathfrak S)$ be a hierarchically hyperbolic space with $\mathfrak S$ countable and $\cuco X$ proper.  Let $G \leq \Aut(\mathfrak S)$ be a countable hierarchically acylindrical subgroup.  Then there exists an element $g \in G$ with $\Active(G)= \BIG(g)$.  Moreover, for any $g' \in G$ and each $U \in \BIG(g')$, there exists $V \in \BIG(g)$ with $U \nest V$.
\end{thm}

Before we prove Theorem \ref{thm:ost hhg}, we prove a lemma related to fixed boundary points of $G$.  Throughout, $\xi(\mathfrak S)$ denotes the complexity of $(\cuco X,\mathfrak S)$, i.e. the length of a longest $\nest$--chain.  

\begin{defn}[Fixed-point set]\label{defn:fix}
Given an arbitrary HHS $(\cuco X,\mathfrak S)$ and $G\leq\Aut(\mathfrak S)$, let $\mathrm{Fix}(G)=\{p\in\boundary(\cuco X,\mathfrak S)\mid G\cdot p=p\}$.
\end{defn}

Given $p\in \Fix(G)$, let $G'\leq_{f.i.} G$ be a finite index subgroup of $G$ which fixes each $U\in\support(p)$.  Let $U \in \support(p)$ and suppose that $G$ is hierarchically acylindrical.  Since $G'$ fixes $U$, the restriction homomorphism $\theta_U$ gives a group $G'_U$ which (coarsely) acts on $F_U$ and acts acylindrically on $\fontact U$.  The next lemma relates supports of fixed points to active domains.

\begin{lem}\label{lem:fixed to big}
If $p \in \mathrm{Fix}(G)$, $U \in \support(p)$, and $V \in \Active(G)$, then either $U \orth V$ or $U = V$.  Moreover, in the latter case, there exists $g'_U \in G'_U$ such that $U \in \BIG(g'_U)$ and $\langle g'_U \rangle \leq_{f.i.} G'_U$.
\end{lem}

\begin{proof}
We separately analyze two cases.

\textbf{The case $U\transverse V$ or $U\propnest V$:}  Suppose that $U\transverse V$ or $U\propnest V$, i.e. $\rho^U_V$ is a well-defined coarse point.  Since $G'\cdot U=U$, we have that $G'$ coarsely stabilizes the image of $P_U=F_U\times E_U\to\cuco X$, which we denote $\overline P_U$.  In other words, $G' \cdot x_0$ is uniformly close to $\overline{P}_U$ for all $x_0 \in \overline{P}_U$.

By definition of the standard embedding, if $V \pitchfork U$ or $U \nest V$, then $\pi_V(\overline{P}_U) \asymp \rho^U_V \in \fontact V$ (see Subsection \ref{sec:product_regions}).  Thus for any $x_0 \in \overline{P}_U$ and $V \in \mathfrak S$ with $U \pitchfork V$ or $U \propnest V$, we have
$$\diam_V(G' \cdot x_0)\asymp 1$$
which implies that any orbit of $G'$ projects to a bounded subset of $\fontact V$.  Hence $V\not\in\Active(G)$, a contradiction.  Thus either $V\nest U$ or $V\orth U$.

\textbf{The case $V\nest U$:}  Now suppose $V \nest U$. Since $U \in \support(p)$, it follows that $G'_U$ fixes a point $p_U \in \partial F_U$, where $p_U \in \partial \fontact U$.  Since $G$ is hierarchically acylindrical, $G'_U$ acts acylindrically on $\fontact U$.  By Theorem~\ref{thm:osin} and the fact that $G'_U$ fixes a point in $\partial\fontact U$, one of the following holds:
\begin{enumerate}
 \item \label{item:osin_consequence_1} $G'_U$ has bounded orbits in $\fontact U$;
 \item \label{item:osin_consequence_2} $G'_U$ contains an element $g'_U$ which acts axially on $\fontact U$, and $\langle g'_U \rangle \leq_{f.i.} G'_U$.
\end{enumerate}

If item~\eqref{item:osin_consequence_1} holds, then, since $G'_U$ fixes a point of $\boundary\fontact U$, Lemma~\ref{lem:bounded_orbits_fixed_point} implies that $|G'_U|<\infty$.  In this case, since $V\nest U$, we have $\pi_V(G'\cdot x)=\pi_V(G'_U\cdot x)$ is finite, so $V\not\in\Active(G)$, a contradiction.

If item~\eqref{item:osin_consequence_2} holds, then we have found the desired element $g'_U$.  Moreover, the existence of this element shows that $U$ is nested into some element of $\Active(G)$.  On the other hand, $V\nest U$ and $V\in\Active(G)$, so $U=V$ by maximality of $V$.

Thus the only possibilities are that either $V\orth U$ or $U=V$ and the desired $g'_U$ exists.
\end{proof}

We are now ready for the proof of Theorem \ref{thm:ost hhg}:

\begin{proof}[Proof of Theorem \ref{thm:ost hhg}]
The ``moreover'' part of the statement follows automatically from the first assertion and the definition of $\Active(G)$, for if $g' \in G$ and $U \in \BIG(g')$, then $U$ is an active domain for $G$ and thus $U$ must nest into some domain in $\Active(G)=\BIG(g)$.

We now prove the main part of the statement.  By Proposition~\ref{prop:essential core}, we can assume that $G$ acts essentially on $\cuco X$.  Let $S \in \mathfrak S$ to be the unique $\nest$-maximal domain in $\mathfrak S$.  Note that if $G$ contains an irreducible axial element or has finite order, then we are done.  Moreover, by acylindricity of the action of $G$ on $\fontact S$, either $G$ contains an irreducible axial or has bounded orbits in $\fontact S$ (so $S\not\in\Active(G)$).

In particular, if $G$ fixes a point of $\boundary\fontact S$, then Lemma~\ref{lem:bounded_orbits_fixed_point} implies that $|G|<\infty$, and we are done.  We may therefore assume that $G$ does not fix a point in $\boundary\fontact S$ and $S \notin \Active(G)$.

We now argue by induction on complexity of $\mathfrak S$.  Suppose that $\xi(\mathfrak S) = 1$.  Then either there is an irreducible axial element, and we are done, or $G$ acts with bounded orbits on $\fontact S$, in which case $\Active(G)=\emptyset$ since $\mathfrak S=\{S\}$, and we are done.

Now assume that the statement holds for any group of automorphisms of an HHS that satisfies the hypotheses of the theorem and has complexity less than $\xi(\mathfrak S)$.  

There are two main cases, depending on whether or not $G$ has a fixed point in $\boundary \cuco X$.

First consider the case where $G$ fixes no point of $\boundary \cuco X$.  Proposition~\ref{prop:alternative_non_parabolic_version} implies that either $G$ contains an irreducible axial, in which case we are done, or there exists $U \in \mathfrak S-\{S\}$ such that $|G\cdot U|< \infty$ and $\cuco X$ is coarsely equal to $P_U \subset \cuco X$.  In the latter case, after passing to a finite-index subgroup if necessary, we have $G$ acting by automorphisms on the HHS $(P_U,\mathfrak S)$ (with complexity $\xi(\mathfrak S)$).

The remaining possibility is that $G$ fixes some $p\in\boundary\cuco X-\boundary\fontact S$.  In this case, after passing if necessary to a finite-index subgroup, we again find $U\in\mathfrak S-\{S\}$ with $GU=U$ and $G$ acting by automorphisms on the HHS $(P_U,\mathfrak S)$ (with complexity $\xi(\mathfrak S)$).

In either case, let $P_U=F_U\times E_U$, so that $\mathfrak S$ contains orthogonal subsets $\mathfrak S_U,\mathfrak S_U^\orth$ such that $(F_U,\mathfrak S_U)$ and $(E_U,\mathfrak S_U^\orth)$ are HHSes of complexity at most $\xi(\mathfrak S)-1$.  By replacing $G$ with an index-$2$ subgroup if necessary, we can assume that $G$ stabilizes $\mathfrak S_U$.  Moreover, $G$ stabilizes $\mathfrak S_U^{\orth_o}:=\{V\in\mathfrak S:V\orth U\}$, i.e. $\mathfrak S_U^{\orth_o}$ is obtained from $\mathfrak S_U^\orth$ by removing $W$ if $W\notorth U$, where $W\propnest S$ is the (arbitrarily-chosen) $\nest$-minimal ``container'' domain containing everything orthogonal to $U$.

Recall that $\mathfrak S_U^{\orth}$ consists of all domains $V \in \mathfrak S$ with $V \orth U$ along with a $\nest$-minimal domain $W \in \mathfrak S$ such that $V \nest W$ for all $V \orth U$.  If $W$ is the unique such domain, then $G \cdot W = W$, and thus $G$ admits a natural restriction homomorphism to $\Aut(\mathfrak S_U^{\orth})$.

Otherwise, $W \notin \Active(G)$.  Since $\diam_W(\pi_W(P_U)) \asymp 1$, we may replace $W$ with single point $W^*$ so that $\fontact W^* =  \{*\}$.  From this we obtain a new HHS structure on $(E_U, \mathfrak S_U^{\orth_o})$, where $\mathfrak S_U^{\orth_o} = \mathfrak S_U^{\orth} - W \cup \{W^*\}$, by making the obvious alterations to the projection and domain maps associated to $W$.
  
In either case, let $G_U$ be the image of $G$ under the usual restriction homomorphism $\mathcal A_U\to\Aut(\mathfrak S_U)$.  Let $G_U^\orth$ be the image of $G$ under the restriction map $\psi:\mathcal A_U\to\Aut( \mathfrak S_U^{\orth})$ or, if $W$ is not unique, we take $G_U^{\orth}$ be the image of $\psi:\mathcal A_U \rightarrow \Aut(\mathfrak S_U^{\orth_o})$ defined as follows: for all $g\in\mathcal A_U$, the map $\psi(g)$ acts like $g$ on $\mathfrak S_U^{\orth_o}$ and acts as the identity on $\fontact W^*$.

Hence we have HHS $(F_U,\mathfrak S_U),(E_U,\mathfrak S_U^\orth)$, of complexity at most $\xi(\mathfrak S)-1$, and groups $G_U\leq\Aut(\mathfrak S_U)$ and $G_U^\orth\leq\Aut(\mathfrak S_U^\orth)$ or $\Aut( \mathfrak S_U^{\orth_o})$, that satisfy the hypotheses of the theorem.

We now show that $\Active(G)=\Active(G_U)\sqcup\Active(G_U^\orth)$.  The inclusions $\Active(G_U),\Active(G_U^\orth)\to\Active(G)$ are obvious.  Conversely, suppose that $V\in\Active(G)$.  If $U\in\support(p)$ for some $p\in\Fix(G)$ (as we can assume is the case whenever $\Fix(G)\neq\emptyset$), then Lemma~\ref{lem:fixed to big} implies that $V=U$ or $V\orth U$, i.e. $V\in\mathfrak S_U\sqcup\mathfrak S_U^\orth$ (and, if $V=W$, then $W$ is the unique container and hence $G$--invariant).  Otherwise, the proof of Lemma~\ref{lem:fixed to big} shows that $V\orth U$ or $V\nest U$.  Hence $V\in\Active(G_U)\sqcup\Active(G_U^\orth)$.  

By induction on complexity, either $\Active(G_U)=\emptyset$, or there exists $\bar h\in G_U$ with $\BIG(\bar h)=\Active(G_U)$. Likewise, either $\Active(G_U^\orth)=\emptyset$, or there exists $\bar h^\orth\in G_U^\orth$ with $\BIG(\bar h^\orth)=\Active(G_U^\orth)$.  If $\Active(G_U)=\emptyset$ (repsectively, $\Active(G_U^\orth)=\emptyset$), we take $\bar h=1$ (respectively, $\bar h^\orth=1$).  Since $\Active(G)=\Active(G_U)\sqcup\Active(G_U^\orth)$, we must use $\bar h,\bar h^\orth$ to find $g\in G$ with $\BIG(g)=\Active(G_U)\sqcup\Active(G_U^\orth)$.

Choose $h,h^\orth\in G$ stabilizing $\mathfrak S_U$ and $\mathfrak S_U^\orth$ and mapping to $\bar h\in G_U,\bar h^\orth\in G_U^\orth$, respectively, under the above restriction maps.  Let $k$ be the image of $h$ in $G_U^\orth$ and let $k^\orth$ be the image of $h^\orth$ in $G_U$, so we are considering the action of $h,k^\orth$ on $\mathfrak S_U$ and $h^\orth,k$ on $\mathfrak S_U^\orth$.    

Let $\{U_1,\ldots,U_\ell\}=\BIG(\bar h)\subset\mathfrak S_U$ and let $\{V_1,\ldots,V_k\}=\BIG(\bar h^\orth)\subset\mathfrak S_U^\orth$.  By passing to powers, we can assume that $hU_i=U_i$ and $h^\orth V_j=V_j$ for all $i,j$.  Since the action of $G_U$ on $\mathfrak S_U$ preserves $\Active(G_U)$, and the action of $G_U^\orth$ on $\mathfrak S_U^\orth$ preserves $\Active(G_U^\orth)$, we can, by passing to powers, assume that $k^\orth$ preserves each $U_i$ and $k$ preserves each $V_j$.

Let $N\gg0$ and consider $F=\langle h^N,(h^\orth)^{10N}\rangle \leq G$.  The image of $F$ in $G_U$ is $\bar F=\langle \bar h^N,(k^\orth)^{10N}\rangle$, and the image of $F$ in $G_U^\orth$ is $\bar F^\orth=\langle k^N,(\bar h^\orth)^{10N}\rangle$.  The above discussion shows that $\bar F$ acts acylindrically on each $\contact U_i$ and $\bar F^\orth$ acts acylindrically on each $\contact V_j$.  Examining the various cases that arise according to how $k$ acts on the $\fontact V_i$ and how $k^\orth$ acts on the $\fontact U_i$ shows that, in each case, there exists $g\in F$ whose image in $\bar F$ is loxodromic on each $\contact U_i$ and whose image in $\bar F^\orth$ is loxodromic on each $V_j$.  Hence $\BIG(g)=\Active(G_U)\sqcup\Active(G_U^\orth)$, as required.
\end{proof}

The following is an immediate but useful corollary of Theorem \ref{thm:ost hhg}:
\begin{cor}
If $G \leq \Aut(\mathfrak S)$ is hierarchically acylindrical, then $\Active(G)$ is pairwise orthogonal.
\end{cor}

\subsection{Rank-rigidity for some CAT(0) cube complexes}\label{subsec:rr_cube}
We now use Theorems~\ref{thm:rank_rigidity_HHG} and \ref{thm:rank_rigidity_non_parabolic} to reprove the rank-rigidity theorem of Caprace and Sageev~\cite{CapraceSageev:rank_rigidity}, in the case where the cube complex in question contains a \emph{factor system}.  See Section~\ref{sec:cube_complex} for a discussion of the definition, and the definition of the \emph{simplicial boundary} $\simp\cuco X$ of the cube complex $\cuco X$.

\begin{cor}[Rank-rigidity for cube complexes with factor-systems]\label{cor:rr_cubes}
Let $\cuco X$ be an unbounded CAT(0) cube complex with a factor-system $\mathfrak F$.  Let $G$ act on $\cuco X$ and suppose that one of the following holds:
\begin{enumerate}
 \item $G$ acts on $\cuco X$ properly and cocompactly;
 \item $G$ acts on $\cuco X$ with no fixed point in $\cuco X\cup\simp\cuco X$.
\end{enumerate}
Then $\cuco X$ contains a $G$--invariant convex subcomplex $\cuco Y$ such that either $G$ contains a rank-one isometry of $\cuco Y$ or $\cuco Y=\cuco A\times\cuco B$, where $\cuco A$ and $\cuco B$ are unbounded convex subcomplexes. 
\end{cor}

We remark that in view of~\cite[Remark 5.3]{Hagen:boundary}, we could have stated the corollary in terms of fixed points in the CAT(0) boundary rather than the simplicial boundary, but we have opted for the latter because of the close relationship between the simplicial and HHS boundaries discussed in Section~\ref{sec:cube_complex}.

\begin{proof}[Proof of Corollary~\ref{cor:rr_cubes}]
First suppose that $G$ acts on $\cuco X$ essentially, in the sense that every halfspace contains points of some $G$--orbit arbitrarily far from the associated hyperplane (in particular, $\cuco X$ does not contain a $G$--invariant proper convex subcomplex).  Recall from~\cite{BehrstockHagenSisto:HHS_I} that $\cuco X$ is equipped with a hierarchically hyperbolic structure $(\cuco X,\mathfrak S)$, where $\mathfrak S$ is the set of \emph{factored contact graphs} of elements of $\mathfrak F$, and that $G\leq\Aut(\mathfrak S)$.  If $G$ acts on $\cuco X$ properly and cocompactly, then $(G,\mathfrak S)$ is an HHG; if $G$ acts on $\cuco X$ with no fixed point in $\simp\cuco X$, then $G$ does not fix a point in $\boundary(\cuco X,\mathfrak S)$, by Theorem~\ref{thm:simplicial_HHS} below.  

Depending on which hypothesis we invoke, one of Theorem~\ref{thm:rank_rigidity_HHG} or Theorem~\ref{thm:rank_rigidity_non_parabolic} implies that either there exists $g\in G$ which is rank-one (in the HHS sense) or there exists $U\in\mathfrak S$ so that $\cuco X$ coarsely coincides with $P_U$, which has unbounded factors, and $G'U=U$ for some finite-index $G'\leq G$.  In the former case, elements that are rank-one in the HHS sense (with respect to this particular HHS structure on $\cuco X$) are rank-one isometries of $\cuco X$ in the usual sense, by~\cite[Proposition~5.1]{Hagen:boundary} and the definition of a factor system~\cite[Section~8]{BehrstockHagenSisto:HHS_I}.  

In the latter case, $P_U=F_U\times E_U$ is a genuine convex product subcomplex with unbounded factors (see~\cite{BehrstockHagenSisto:HHS_I}).  Let $g\in G$ and suppose that $H$ is a hyperplane intersecting $P_U$ but not $gP_U$.  Since $P_U$ is coarsely equal to $\cuco X$ and $\cuco X$ is essential, the halfspace of $P_U$ separated from $gP_U$ by $H$ contains points arbitrarily far from $H$, whence $P_U$ and $gP_U$ cannot lie at finite Hausdorff distance.  This contradicts that $P_U$ is invariant under a finite-index subgroup of $G$.  Hence $P_U$ and $gP_U$ are \emph{parallel} for all $g\in G$, i.e. they are crossed by exactly the same hyperplanes.  Thus $\cuco X=P_U\times Y$ for some compact cube complex $Y$, whence $Y$ is a single point, by essentiality.  It follows that $P_U$ is $G$--invariant, so $\cuco X=P_U$ by essentiality.  Hence $\cuco X$ decomposes as a product with unbounded factors. In general, we first replace $\cuco X$ by its $G$--essential core in either preceding argument, using Proposition~3.5 of~\cite{CapraceSageev:rank_rigidity}.
\end{proof}

\begin{rem}\label{rem:factor_system}
Question~A of~\cite{BehrstockHagenSisto:HHS_II} asks whether the existence of a proper cocompact action of $G$ on the CAT(0) cube complex $\cuco X$ ensures that $\cuco X$ contains a factor system.  By a result in~\cite{BehrstockHagenSisto:HHS_I}, the answer is affirmative provided $\cuco X$ embeds as a convex subcomplex in the universal cover of the Salvetti complex of some right-angled Artin group.  Although it is a strong condition, we believe that such embeddings always exist (although there is in general no algebraic relationship between $G$ and the RAAG).
\end{rem}

\subsubsection{The Poisson boundary of an HHG}\label{subsubsec:poisson}
Results in~\cite{BehrstockHagenSisto:HHS_I} show that, if $G$ is an HHG with $\diam\fontact S=\infty$, then, given a nonelementary probability measure $\mu$ on $G$, the boundary $\boundary \fontact S$ admits a $\mu$--stationary measure making it the Poisson boundary.  As a topological model of the Poisson boundary, $\boundary\fontact S$ is unsatisfactory since it need not be compact.  However:

\begin{thm}[The HHS boundary is the Poisson boundary]\label{thm:poisson}
Let $(G, \mathfrak S)$ be an HHG with $\diam \fontact S  = \infty$, $\mu$ be a nonelementary probability measure on $G$ with finite entropy and finite first logarithmic moment, and $\nu$ the resulting $\mu$-stationary measure on $\partial G$.  Then $(\partial G, \nu)$ is the Poisson boundary for $(G, \mu)$.
\end{thm}

We use acylindricity of the action of $G$ on $\fontact S$ and a result of Maher-Tiozzo \cite{maher2014random}:

\begin{thm}[Theorem 1.5 in \cite{maher2014random}]\label{thm:maher-tiozzo}
Let $G$ be a countable group which acts acylindrically on a separable Gromov hyperbolic space $X$.  If $\mu$ is a nonelementary probability measure on $G$ with finite entropy and finite first logarithmic moment with corresponding stationary measure $\nu$, then $(\partial X, \nu)$ is the Poisson boundary for $(G, \mu)$.
\end{thm}

\begin{proof}[Proof of Theorem \ref{thm:poisson}]
Let $\mu$ be a nonelementary probability measure on $G$ with finite entropy and finite first logarithmic moment.  Since $G$ acts on $\fontact S$ acylindrically \cite{BehrstockHagenSisto:HHS_I}[Theorem 14.3], Theorem \ref{thm:maher-tiozzo} implies that there exists a $\mu$-stationary measure $\nu'$ on $\partial \fontact S$ such that $(\partial \fontact S, \nu')$ is the Poisson boundary for $(G, \mu)$.

Let $f:\boundary\fontact S \hookrightarrow \partial G$ be the embedding from Proposition \ref{prop:babies continuously embed}.  By Lemma \ref{lem:support_measurable}, $f(\partial \fontact S)$ is Borel, so for any Borel subset $V \subset \partial G$, the set $V \cap f(\partial \fontact S)$ is Borel.  Define a new measure $\nu$ on $\partial G$ by $\nu(V) = \nu'\left(f^{-1}(V \cap f(\partial \fontact S))\right).$

Since $f$ is $G$-equivariant, it follows that $\nu$ is $\mu$-stationary.  By definition, $f(\partial \fontact S)$ has full $\nu$-measure.  Moreover, $(\partial G, \nu)$ is a $\mu$-boundary by measurability of $f$ and it is maximal since $(\partial \fontact S, \nu')$ is maximal.  Thus $(\partial G, \nu)$ models the Poisson boundary for $(G, \mu)$.
\end{proof}

\section{Case study: CAT(0) cube complexes}\label{sec:cube_complex}
Throughout this section, $\cuco X$ is a locally finite CAT(0) cube complex in which each collection of pairwise--intersecting hyperplanes is (not necessarily uniformly) finite. In~\cite{BehrstockHagenSisto:HHS_I}, it is shown that CAT(0) cube complexes can often be given HH structures using certain collections of convex subcomplexes called \emph{factor systems}. We recall the definition in Subsection \ref{subsec:factor_systems}. When $\mathfrak F$ is a factor system for $\cuco X$, denote the resulting HH structure by $(\cuco X,\overline{\mathfrak F})$.

The \emph{simplicial boundary} of $\cuco X$ was introduced in~\cite{Hagen:boundary}; we recall the definition below.  The simplicial boundary and the HH structure are closely related by the following theorem:

\begin{thm}[Simplicial and HHS boundaries]\label{thm:simplicial_HHS}
Let $\cuco X$ be a CAT(0) cube complex with a factor system $\mathfrak F$.  There is a topology $\mathcal T$ on the simplicial boundary $\simp\cuco X$ so that:
\begin{enumerate}
 \item There is a homeomorphism $b:(\simp\cuco X,\mathcal T)\to\boundary(\cuco X,\overline{\mathfrak F})$,
 \item for each component $C$ of the simplicial complex $\simp\cuco X$, the inclusion $C\hookrightarrow(\simp\cuco X,\mathcal T)$ is an embedding.
\end{enumerate}
In particular, if $\mathfrak F,\mathfrak F'$ are factor systems on $\cuco X$, then $\boundary(\cuco X,\overline{\mathfrak F})$ is homeomorphic to $\boundary(\cuco X,\overline{\mathfrak F'})$. 
\end{thm}

We prove Theorem~\ref{thm:simplicial_HHS} in Subsection~\ref{subsec:simplicial_vs_HHS_boundary}.

\begin{rem}\label{rem:tits_visual}
Proposition~3.37 of~\cite{Hagen:boundary} relates $\simp\cuco X$ to its Tits boundary $\partial_T\cuco X$.  There is an analogous relationship between the HHS boundary and the visual boundary when the former is defined (i.e. when $\cuco X$ has a factor system).  Specifically, one can show that there is a commutative diagram \begin{center}
$
\begin{diagram}
\node{\simp\cuco X}\arrow{e,t}{I}\arrow{s,l}{b}\node{\partial_T\cuco X}\arrow{s,r}{\id}\\
\node{\boundary(\cuco X,\overline{\mathfrak F})}\arrow[1]{e,t}{J}\node{\visual\cuco X}
\end{diagram}
$
\end{center}
where $b$ is the bijection from Theorem~\ref{thm:simplicial_HHS}, $I$ and $J$ are embeddings, $J$ is $\pi/2$--quasi-surjective, and $\boundary(\cuco X,\overline{\mathfrak F})$ is a deformation retract of $\visual\cuco X$.  The CAT(0) metric on $\cuco X$ is far afield from our present discussion, since the HHS structure depends only on the combinatorics of $\cuco X$ and is insensitive to changes in the CAT(0) metric (unlike the visual boundary~\cite{CrokeKleiner}), so we will not give a detailed proof of the above.  The top part of the diagram comes from~\cite[Proposition 3.37]{Hagen:boundary}; the missing ingredient is to shown that $J$ is an embedding, which is a tedious exercise in the definition of the topology on $\boundary(\cuco X,\overline{\mathfrak F})$.
\end{rem}

\subsection{The simplicial boundary}\label{subsec:simplicial_boundary}
We first recall the necessary definitions from~\cite{Hagen:boundary}.

\begin{defn}[UBS, boundary equivalence, minimal UBS]\label{defn:ubs}
A set $\mathcal U$ of hyperplanes in $\cuco X$ is a \emph{unidirectional boundary set (UBS)} if each of the following holds:
\begin{itemize}
 \item \label{item:infinite} $\mathcal U$ is infinite;
 \item \label{item:inseparable}  if $U,U'\in\mathcal U$ and a hyperplane $V$ separates $U,U'$, then $V\in\mathcal U$;
 \item \label{item:no_ft}  if $U,U',U''\in\mathcal U$ are pairwise disjoint, then one of them separates the other two;
 \item \label{item:unidirectional}  for all hyperplanes $W$, at least one component of $\cuco X-W$ contains at most finitely many elements of $\mathcal U$.
\end{itemize}
Given UBSes $\mathcal U,\mathcal V$, let $\mathcal U\preceq\mathcal V$ if all but finitely many elements of $\mathcal U$ lie in $\mathcal V$.  The UBSes $\mathcal U,\mathcal V$ are \emph{boundary equivalent} if $\mathcal U\preceq\mathcal V$ and $\mathcal V\preceq\mathcal U$, and $\mathcal U$ is \emph{minimal} if $\mathcal U$ and $\mathcal V$ are boundary equivalent for all UBSes $\mathcal V$ with $\mathcal V\preceq\mathcal U$. 
\end{defn}

\begin{rem}
Any infinite set of hyperplanes which is closed under separation contains a minimal UBS~\cite[Lemma~3.7]{Hagen:boundary}.
\end{rem}

Proposition~3.10 of~\cite{Hagen:boundary} shows that each UBS $\mathcal U$ is boundary equivalent to a UBS of the form $\bigsqcup_{i=0}^k\mathcal U_i$, where each $\mathcal U_i$ is a minimal UBS, and this decomposition is unique up to boundary equivalence.  Up to reordering, for $0\leq i< j\leq k$, for all but finitely many $U\in\mathcal U_j$, the hyperplane $U$ intersects all but finitely many elements of $\mathcal U_i$.  In this situation, $\mathcal U_j$ \emph{dominates} $\mathcal U_i$.  The number $k$ is the \emph{dimension} of $\mathcal U$.

\begin{defn}[Simplicial boundary]\label{defn:crossing}
A \emph{$k$--simplex at infinity} is a boundary equivalence class of $k$--dimensional UBSes.  If $v,v'$ are simplices at infinity, represented by boundary sets $\mathcal V,\mathcal V'$, then $\mathcal V\cap\mathcal V'$ is, if infinite, a boundary set representing the simplex $v\cap v'$.  The \emph{simplicial boundary} $\simp\cuco X$ of $\cuco X$ is the simplicial complex with a closed $k$--simplex for each $k$--dimensional simplex at infinity; the simplex $u$ represented by the UBS $\mathcal U$ is a face of the simplex $v$, represented by $\mathcal V$, if $\mathcal U\preceq\mathcal V$.      
\end{defn}

\begin{rem}[Boundaries of convex subcomplexes]\label{rem:subcomplex_boundary}
It is shown in~\cite{Hagen:boundary} that if $\cuco Y\subseteq\cuco X$ is a convex subcomplex, then $\simp\cuco Y\subset\simp\cuco X$ in a natural way: each simplex at infinity in $\simp\cuco Y$ corresponds to a UBS in $\cuco X$ consisting of hyperplanes that intersect $\cuco Y$, and these hyperplanes intersect in $\cuco X$ exactly when they intersect in $\cuco Y$, by convexity. 
\end{rem}

\subsubsection{Visibility}\label{subsubsec:visibility}

\begin{defn}[Visible simplex]\label{defn:visibility}
The simplex $u$ at infinity is \emph{visible} if there exists a combinatorial geodesic ray $\gamma$ in $\cuco X^{(1)}$ such that the set $\mathcal U$ of hyperplanes intersecting $\gamma$ represents the boundary--equivalence class $u$.  Otherwise, the simplex $u$ at infinity is \emph{invisible}.  If every simplex at infinity is visible, then $\cuco X$ is \emph{fully visible}.
\end{defn}

\noindent Theorem~3.19 of~\cite{Hagen:boundary} states that every maximal simplex of $\simp\cuco X$ is visible.  Visibility is also related to a subtlety in the definition of $\simp\cuco X$:  

\begin{rem}[Visibility and proper faces]\label{rem:defn_subtlety}
Let $\bigsqcup_{i=0}^k\mathcal U_i$ be a UBS, with each $\mathcal U_i$ a minimal UBS, numbered so that for $0\leq i<j\leq k$ and all $U\in\mathcal U_j$, we have that $U\cap V\neq\emptyset$ for all but finitely many $V\in\mathcal U_i$. If, up to modifying each $\mathcal U_i$ in its boundary equivalence class, $U\cap V\neq\emptyset$ whenever $U\in\mathcal U_i,V\in\mathcal V_j$, and $i\neq j$, then the simplex $u$ represented by $\bigsqcup_{i=0}^k\mathcal U_i$ is visible. In this case, $\cuco X$ contains an isometrically embedded (on the $1$--skeleton) cubical orthant, the boundary of whose convex hull is $u$.  Conversely, if we know that each $\mathcal U_i$ represents a visible $0$--simplex, then $\bigsqcup_{i\in K}\mathcal U_i$ represents a visible simplex at infinity for any $K\subset\{0,\ldots,k\}$, as is proved in~\cite{Hagen:boundary}.  If this does not occur, then there may be subsets $K\subset\{0,\ldots,k\}$ so that $\bigsqcup_{i\in K}\mathcal U_i$ represents an invisible simplex at infinity, or is not even a UBS (by virtue of failing to satisfy the condition on separation).  In other words, when $\cuco X$ is not fully visible, simplices at infinity may have \emph{proper faces} that are not genuine simplices at infinity represented by UBSes.\end{rem}

A visible simplex $v\subseteq\simp\cuco X$ is \emph{represented} by the combinatorial geodesic ray $\gamma\subseteq\cuco X^{(1)}$ if the UBS of hyperplanes intersecting $\gamma$ represents the boundary equivalence class $v$.

\begin{rem}[Factor systems and visibility]\label{rem:FS_vis}
Conjecture~2.8 of~\cite{BehrstockHagen:thick} states that if $\cuco X$ is a CAT(0) cube complex on which some group acts geometrically, then $\cuco X$ is fully visible.  Also, the proof of Theorem~\ref{thm:simplicial_HHS} shows that, if $\cuco X$ contains a factor system (see Definition~\ref{defn:factor_system}), then every simplex of $\simp\cuco X$ is visible.  This is another reason for interest in Question~A of~\cite{BehrstockHagenSisto:HHS_II}, which asks whether every CAT(0) cube complex on which some group acts geometrically contains a factor system.
\end{rem}

\subsection{Factor systems: hierarchical hyperbolicity of cube complexes}\label{subsec:factor_systems}
\renewcommand{\fontact}{\widehat{\mathcal C}}
We now summarize results from~\cite{BehrstockHagenSisto:HHS_I} yielding hierarchically hyperbolic structures on $\cuco X$.  We refer the reader to Section~2 of~\cite{BehrstockHagenSisto:HHS_I} for discussion of convex subcomplexes and the gate map $\gate_F:\cuco X\rightarrow F$ from $\cuco X$ to any convex subcomplex $F$.  

Recall that each hyperplane $H$ of $\cuco X$ lies in a \emph{carrier}, $\neb(H)$, which is the union of closed cubes intersecting $H$.  For all $H$, there is a cubical isomorphism $\neb(H)\cong H\times[-\frac{1}{2},\frac{1}{2}]$; a subcomplex of $\cuco X$ which is the image under the inclusion $\neb(H)\to\cuco X$ of either of the subcomplexes $H\times\{\frac{1}{2}\}$ or $H\times\{-\frac{1}{2}\}$ is a \emph{combinatorial hyperplane}. We say that two convex subcomplexes $F,F'$ of $\cuco X$ are \emph{parallel} if for any hyperplane $H$ of $\cuco X$, we have $H\cap F\neq \emptyset$ if and only if $H\cap F'\neq \emptyset$. We let $\overline{ \mathfrak F}$ denote a choice of representatives for each parallelism class of elements of $\mathfrak F$.

\begin{defn}\label{defn:factor_system}
A \emph{factor system} $\mathfrak F$ is a set of convex subcomplexes such that:
\begin{enumerate}
 \item\label{item:FS_hyperplane}  Each nontrivial combinatorial hyperplane of $\cuco X$ belongs to $\mathfrak F$, as does each convex subcomplex parallel to a nontrivial combinatorial hyperplane,
 \item\label{item:FS_master} $\cuco X\in\mathfrak F$,
 \item\label{item:projection} there exists $\xi>0$ such that for all $F,F'\in\mathfrak F$, either $\gate_F(F')\in\mathfrak F$ or $\diam(\gate_F(F'))\leq\xi$,
 \item\label{item:multiplicity} there exists $\Delta\geq1$ such that each point in $\cuco X$ belongs to at most $\Delta$ elements of $\mathfrak F$.
\end{enumerate}
We require that elements of $\mathfrak F$ are not single points.  (This condition is only imposed to ensure that nesting and orthogonality are mutually exclusive: if $F$ is a single point and $F'\in\mathfrak F$, then $F\orth F'$ and $F\nest F'$, so we exclude this situation.)
\end{defn}

The \emph{contact graph} $\contact\cuco X$ of $\cuco X$ (see~\cite{Hagen:quasi_arb}) has a vertex for each hyperplane, with two hyperplanes joined by an edge if no third hyperplane separates them.  If $F\subseteq\cuco X$ is a convex subcomplex, then $F$ is a CAT(0) cube complex whose hyperplanes have the form $H\cap F$, where $H$ is a hyperplane of $\cuco X$, and, by convexity of $F$, this yields an embedding $\contact F\hookrightarrow\contact\cuco X$ of $\cuco F$ as a full subgraph.

Given a factor system $\mathfrak F$ on $\cuco X$, we define the \emph{factored contact graph} $\fontact F$ of each $F\in\mathfrak F$ as follows. Begin with $\contact F$.  For each parallelism class of subcomplexes $F'\in\mathfrak F$, parallel to a proper subcomplex of $F$ that is not a single $0$--cube, we have $\contact F'\subsetneq\contact F$, and we cone off $\contact F'$ by adding a vertex $v_{F'}$ to $\contact F$ and joining each vertex of $\contact F'\subset\contact F$ to $v_{F'}$.  The resulting factored contact graph $\fontact F$ is uniformly quasiisometric to a tree~\cite[Proposition~8.24]{BehrstockHagenSisto:HHS_I}.   

Let us now define the maps $\pi_F:\cuco X\to 2^{\fontact F}$. For each $F\in\mathfrak F$, given $x\in\cuco X$, let $\gate_F(x)\in F$ be its gate.  There is a nonempty finite set of hyperplanes $H$ of $F$ that are not separated from $x$ by any other hyperplane; these form a nonempty clique in $\contact F$, to which we send $x$.  We then compose with $2^{\contact F}\hookrightarrow 2^{\fontact F}$ to yield $\pi_F:\cuco X\to 2^{\fontact F}$ sending each point to a clique.  

Let $F\nest F'$ if $F$ is parallel to a subcomplex of $F'$, and $F\orth F'$ if there is a cubical isometric embedding $F\times F'\to\cuco X$ (after possibly varying $F,F'$ in their parallelism classes).  Otherwise, $F,F'$ are transverse.  With these definitions, it is shown in~\cite{BehrstockHagenSisto:HHS_I,BehrstockHagenSisto:HHS_II} that $(\cuco X,\overline{ \mathfrak F})$ is a hierarchically hyperbolic space.

\subsection{Relating the simplicial and HHS boundaries}\label{subsec:simplicial_vs_HHS_boundary}
Fix $\cuco X$ with a factor system $\mathfrak F$; necessarily, $\cuco X$ is uniformly locally finite.

\begin{proof}[Proof of Theorem~\ref{thm:simplicial_HHS}]
We will first exhibit a bijection $b:\simp\cuco X\to\boundary(\cuco X,\overline{\mathfrak F})$.  We then define $\mathcal T=\{b^{-1}(\mathcal O)\}$, where $\mathcal O$ varies over all open sets in $\boundary(\cuco X,\overline{\mathfrak F})$, so as to make $b$ a homeomorphism.  It then suffices to verify that this topology agrees with the simplicial topology on each component of $\simp\cuco X$; the ``in particular'' statement then follows immediately.

\textbf{Reduction to the single-simplex case:}  Let $m$ be a maximal simplex of $\simp\cuco X$.  By the definition of the simplicial boundary, $m$ is a simplex at infinity, i.e. it is represented by some UBS $\mathcal M$.  Moreover, by~\cite[Theorem~3.19]{Hagen:boundary}, we can take $\mathcal M$ to be the set of hyperplanes intersecting some combinatorial geodesic ray $\gamma_m$ emanating from the (fixed) basepoint $x_0$.  Let $\mathcal Y_m$ be the convex hull of $\gamma_m$.  

By~\cite[Lemma~8.4]{BehrstockHagenSisto:HHS_I}, $\mathfrak F_m=\{F\cap\cuco Y_m:F\in\mathfrak F\}$ is a factor system.  (We emphasize that $\mathfrak F_m$ is a set, not a multiset: if $F,F'\in\mathfrak F$ satisfy $F\cap\cuco Y_m=F'\cap\cuco Y_m$, we count this subcomplex once.)  We adopt the following convention: for each $F\cap\cuco Y_m\in\mathfrak F_m$, we assume that $F$ has been chosen so that $F$ is $\nest$--minimal among all $F'\in\mathfrak F$ with $F'\cap\cuco Y_m=F\cap\cuco Y_m$.  (Note that there is a unique such minimal $F$: if $F\cap\cuco Y_m=F'\cap\cuco Y_m$, then $F\cap\cuco Y_m=F\cap F'\cap\cuco Y_m$, and $F\cap F'\nest F,F'$.)

Also, if $F\nest F'$, then $F\cap\cuco Y_m\nest F'\cap\cuco Y_m$, obviously.  Conversely, suppose that $F\cap\cuco Y_m\nest F'\cap\cuco Y_m$.  Let $F''=\gate_{F}(F')$, so $F''\nest F'$ and $F''\nest F$.  Then $F''\cap\cuco Y_m=F\cap\cuco Y_m$, so $F''=F$ by minimality, whence $F\nest F'$.  

If $F\orth F'$, then convexity of $\cuco Y_m$ implies $(F\times\orth F')\cap\cuco Y_m=(F\cap\cuco Y_m)\times(F'\cap\cuco Y_m)$, so $(F\cap\cuco Y_m)\orth( F'\cap\cuco Y_m)$.  Conversely, suppose that $(F\cap\cuco Y_m)\orth(F'\cap\cuco Y_m)$.  For brevity, let $A=F\cap\cuco Y_m$ and $B=F'\cap\cuco Y_m$, so that $\cuco X$ contains $A\times B$.  By Lemma~\ref{lem:product_extend}, there exist $F_A,F_B\in\mathfrak F$ so that $A\subset F_A,B\subset F_B$ and $F_A\orth F_B$.  Let $F_A'=F\cap F_A$ and $F'_B=F'\cap F_B$.  Then $F_A'\cap\cuco Y_m=F\cap\cuco Y_m$ and $F_A'\nest F$, so minimality of $F$ implies $F_A'=F$; similarly $F_B'=F'$.  But since $F_A\orth F_B$ and $F'_A\nest F_A,F'_B\nest F_B$, we have $F\orth F'$.

It follows that there is a hieromorphism $(\cuco Y_m,\overline{\mathfrak F_m})\rightarrow(\cuco X,\overline{\mathfrak F})$ defined as follows: the map $\cuco Y_m\rightarrow\cuco X$ is the inclusion; the map $\mathfrak F_m\to\mathfrak F$ is given by $F\cap\cuco Y_m\mapsto F$ for each $F\cap\cuco Y_m\in\mathfrak F_m$ (where $F$ is $\nest$--minimal in $\mathfrak F$ with the given intersection with $\cuco Y$), and for each $F\cap\cuco Y_m$, the map $\fontact(F\cap\cuco Y_m)\to\fontact F$ is the inclusion on contact graphs and sends cone vertices to cone vertices in the obvious way.  


We will see below that $\cuco Y_m=\prod_{i=0}^k\cuco Y_{m_i}$, where each $\cuco Y_{m_i}$ has the property that $\boundary\fontact (F\cap\cuco Y_{m_i})=\emptyset$ for all $F\in\mathfrak F$ except for a unique $\ddot F_i\in\mathfrak F$ for which $\boundary\fontact(\ddot F_i\cap\cuco Y_{m_i})$ consists of a single point $p_i$.  Moreover, $\ddot F_i\orth \ddot F_j$ for $i\neq j$.  Lemma~\ref{lem:discs} shows that for each $F\cap\cuco Y_m$, the map $\fontact(F\cap\cuco Y_m)\hookrightarrow\fontact F$ is a uniform quasiisometric embedding, inducing a boundary map, i.e. $p_i$ may be regarded as a point in $\boundary\fontact \ddot F_i$ for each $i$.  We thus obtain an injective map $b_m:\boundary(\cuco Y_m,\overline{\mathfrak F_m})\to\boundary(\cuco X,\overline{\mathfrak F})$ given by $$b_m\left(\sum_{i=0}^ka_im_i\right)=\sum_{i=0}^ka_ip_i.$$

\textbf{Constructing $b$:}  We will observe below that if $m,m'$ are maximal simplices, then the associated collections $\{p_i\}_{i=0}^k$ and $\{p'_i\}_{i=0}^{k'}$ intersect in a set corresponding precisely to the set of $0$--simplices of $m\cap m'$.  It follows that the maps constructed above are \emph{compatible}, i.e. $b_m|_{_{\cuco Y_{m\cap m'}}}=b_{m'}|_{_{\cuco Y_{m\cap m'}}}$ and that, if $m,m'$ are disjoint maximal simplices of $\simp\cuco X$, then $b_m$ and $b_{m'}$ have disjoint images.  Pasting together the $b_m$ thus yields an injection $b:\simp\cuco X\to\boundary({\cuco X},\overline{\mathfrak F})$.  

\textbf{Surjectivity of $b$:}  Let $\{\ddot F_i\}_{i=1}^k$ be a support set in $\overline{\mathfrak F}$, choose for each $i$ a point $p_i\in\boundary\fontact \ddot F_i$, and let $p=\sum_{i=1}^ka_ip_i$.  For each $i$, let $\sigma_i$ be a geodesic ray in the quasi--tree $\fontact \ddot F_i$ joining $\pi_{\fontact \ddot F_i}(x_0)$ to $p_i$.  Let $\{H_n^i\}$ be a sequence of hyperplanes of $\cuco X$, each crossing $\ddot F_i$, corresponding to vertices of $\sigma_i$, ordered so that $H_n^i$ separates $H_{n+1}^i$ from $x_0$.  Any $P\in\mathfrak F$ that crosses infinitely many of these hyperplanes satisfies $\ddot F_i\nest P$, or else some element of $\mathfrak F$ nested into $\ddot F_i$ would ``kill'' the $p_i$ direction in $\boundary\fontact\ddot F_i$.  Every simplex of $\simp(\prod_{j=0}^k\ddot F_j)\subset\simp\cuco X$ is visible, from which it is easy to check that there is a unique (up to boundary--equivalence) minimal UBS $\mathcal M_i$ containing $\{H_n^i\}$ and representing a $0$--simplex $m_i$ of $\simp\cuco X$ such that $\{m_0,\ldots,m_k\}$ span a simplex $m$.  By definition, $b_m(\sum_ia_im_i)=p$.  

\textbf{Analysis of components:}  To prove that each component $C$ of $\simp\cuco X$, with the simplicial topology, is embedded in $(\simp\cuco X,\mathcal T)$, we must show that $b\circ\id:\simp\cuco X\to\boundary(\cuco X,\overline{\mathfrak F})$ restricts to an embedding on $C$, where $\id:\simp\cuco X\to(\simp\cuco X,\mathcal T)$ is the identity.  Let $m$ be a maximal simplex of $\simp\cuco X$.  Let $p=\sum_ia_ip_i\in b\circ\id(M)$ and let $\neb=\neb_{\{U_i\},\epsilon}(p)\cap\boundary(\cuco Y_m,\overline{\mathfrak F_m})$ be a basic neighborhood of $p$, as defined in Section~\ref{subsec:definition}.  Observe that $\neb$ is completely non-remote, whence it is clear from the definition that $b_m^{-1}(\neb)$ is basic in the simplicial topology on $\simp\cuco Y_m=m$, so $b_m$ is continuous.  It follows that $b\circ\id$ is continuous.  A similar argument shows that the restriction of $b\circ\id$ to $C$ is an open map.  To complete the proof, it now suffices to produce the $F_i$ and analyze their factored contact graphs, which we do in the next several steps.

\textbf{Visibility of faces of $m$:}  Let $m$ be a maximal simplex of $\simp\cuco X$ and observe that $\simp\cuco Y_m$ is exactly the simplex $m$.  We now verify that each face of $m$ is a visible simplex at infinity.  Let $m_0,\ldots,m_k$ be the $0$--simplices of $m$; represent $m_i$ by a minimal UBS $\mathcal M_i$ so that $\mathcal M_j$ dominates $\mathcal M_i$ when $i<j$ and $\mathcal M=\bigsqcup_{i=0}^k\mathcal M_i$.  Recall from Remark~\ref{rem:defn_subtlety} that if $\mathcal M_i$ dominates $\mathcal M_j$ for all $i,j$, then each sub-simplex of $m$ is visible.  

By projecting $\gamma_m$ to a combinatorial hyperplane on the carrier of some element of $\mathcal M_k$, we see that $\mathcal M-\mathcal M_k$ represents a visible codimension--$1$ face $m'$ of $m$, represented by a ray $\gamma_{m'}$. The convex hull $\cuco Y_{m'}\subset\cuco Y_m$ of $\gamma_{m'}$ inherits a factor system from $\cuco Y_m$ as above.  Hence, by induction, for $i<k$, the $0$--simplex represented by $\mathcal M_i$ is visible.  Thus it suffices to show that the $0$--simplex $m_k$ represented by $\mathcal M_k$ is visible.  (In the base case, $m$ is a maximal $0$--simplex, and is visible by maximality.)  Suppose, for a contradiction, that $m_k$ is not visible, so there exists $i<k$ such that $\mathcal M_i$ fails to dominate $\mathcal M_k$.  In particular, $k\geq1$.

The UBS $\mathcal M_k$ contains a sequence $\{M_n\}_{n\geq0}$ of pairwise disjoint hyperplanes such that $M_n$ separates $M_{n\pm1}$ for all $n\geq 1$.  For each $n$, let $M^+_n$ be the combinatorial hyperplane in $\neb(M_n)$ in the same component of $\cuco X-M_n$ as $M_{n+1}$.  For each $n$, let $P_n=\gate_{M_0^+}(M_n^+)$ be the projection of $M^+_n$ on $M_0^+$.  The set of hyperplanes crossed by both $M_0$ and $M_n$ contains all but finitely many elements of $\mathcal M_{i}$; hence each $P_n$ is unbounded and thus belongs to the factor system $\mathfrak F_m$.  Moreover, for all $N\geq0$, the intersection $\bigcap_{n=0}^NP_n\neq\emptyset$.  Hence, since $\mathfrak P_m$ has multiplicity $\Delta<\infty$, it must be the case that there exists $N$ such that $P_n=P_N$ for all $N\geq n$.  Thus, when $n,n'\geq N$, the set of elements of $\mathcal M_j$ crossed by $M_n$ coincides with the set crossed by $M_{n'}$, for all $j\leq k-1$.  Hence each $\mathcal M_j$ dominates $\mathcal M_k$, whence $m_k$ is visible.

\textbf{Structure of $\cuco Y_m$:}  By~\cite[Theorem~3.23]{Hagen:boundary} and visibility of the $m_i$ established above, after moving $x_0$ if necessary, $\cuco Y_m=\prod_{i=0}^k\cuco Y_{m_i}$, where $\cuco Y_{m_i}$ is the convex hull in $\cuco X$ of a combinatorial geodesic ray $\gamma^i$ at the basepoint $x_0$ representing a $0$--simplex $m_i$ of $m$.  Each point of $m=\simp\cuco Y_m$ can be uniquely written as $\sum_{i=0}a_im_i$, where $a_i\geq 0$ and $\sum_{i=1}^ka_i=1$.  

For each $i$, let $\{H^i_n\}_{n\geq0}$ be the set of hyperplanes crossing $\gamma^i$; this is a minimal UBS and is numbered according to the order in which $\gamma^i$ crosses the $H^i_n$.  Thus, if $n>m$, the hyperplane $H_n^i$ does not separate $H_m^i$ from $x_0$ (in fact, either $H_n^i\cap H_m^i\neq\emptyset$ or $H_m^i$ separates $H_n^i$ from $x_0$).  Choose $F_i\in\overline{\mathfrak F_m}$ to be $\nest$--minimal so that all but finitely many $H_n^i$ cross $F_i$.  Observe that $F_i\orth F_j$ for all $i\neq j$, and that $F_i\subseteq\cuco Y_{m_i}$.  

Suppose that $m'$ is some other maximal simplex and $\cuco Y_{m'}=\prod_{i=0}^{k'}\cuco Y_{m_i'}$.  For each $i$, let $F'_i\in\overline{\mathfrak F_{m'}}$ be $\nest$--minimal among those factors crossing all but finitely many of the elements crossing $\cuco Y_{m_i}$.  Suppose that $\boundary \fontact F_i=\boundary\fontact F'_j$ for some $i\leq k,j\leq k'$.  Then the set of hyperplanes crossing $\cuco Y_{m_i}$, which is boundary--equivalent to that crossing $F_i$, is boundary-equivalent to that crossing $F_j'$ and hence that crossing $\cuco Y_{m'_j}$, i.e. $m_i=m'_j$.  

\textbf{Orthogonality:}  Each $F_i$ has the form $F_i=\widehat F_i\cap\cuco Y_m$, where $\widehat F_i\in\overline{\mathfrak F}$.  While orthogonality of elements of $\overline{\mathfrak F}$ implies orthogonality of the corresponding elements of $\overline{\mathfrak F_m}$, the converse need not hold, but we will require that $\widehat F_i\orth\widehat F_j$ for all $i\neq j$, in order to construct points of $\boundary(\cuco X,\overline{\mathfrak F})$.  However, finitely many applications of Lemma~\ref{lem:product_extend} below show that for each $i$, there exists $\ddot F_i\in\overline{\mathfrak F}$ such that $F_i\subseteq\ddot F_i\subseteq\widehat F_i$ and such that $\ddot F_i\orth\ddot F_j$ for all $i\neq j$.  

\textbf{Factored contact graphs in $\overline{\mathfrak F_m}$:}  For any $F\in\overline{\mathfrak F_m}$, we have, by convexity and~\cite[Proposition~2.5]{CapraceSageev:rank_rigidity}, that $F=\prod_{i=0}^k\gate_{\cuco Y_{m_i}}(F)$, whence $\contact F$ decomposes as a join, so $\fontact F$ is obtained from a join by coning off certain subgraphs.  Thus $\fontact F$ is bounded (and $\boundary\fontact F=\emptyset$) unless $F$ is parallel to a subcomplex of some $\cuco Y_{m_i}$.  We claim that $\boundary\fontact F_i$ consists of exactly one point $p_i$ for each $i$, and that, for all other $F\in\overline{\mathfrak F_m}$, we have $\boundary\fontact F=\emptyset$.  

Observe that $\contact F_i$ coarsely coincides with $\contact\cuco Y_i$, the $\{H_n^i\}$ are partially ordered by the order in which $\gamma_i$ crosses them, and that $\contact F_i$ is coarsely equal to a maximal chain in this partial order (i.e. a combinatorial ray $\sigma$ in $\contact F_i$).  By Theorem~2.4 of~\cite{Hagen:boundary}, $\sigma$ is unbounded in $\contact F_i$, since $F_i$ is $\nest$--minimal, and thus determines a point $p_i\in\boundary\contact F_i$.  Moreover, $p_i$ is unique, since $\fontact F_i$ lies in the $1$---neighborhood of $\sigma$ ($\fontact F_i$ is obtained from $\sigma$ by adding edges reflecting intersections of elements of the $\{H_n^i\}$).  

Hence, if $\sigma\subset\fontact F_i$ is unbounded, then $\boundary\fontact F_i=\{p_i\}$.  By $\nest$--minimality of $F_i$, no hyperplane of $F_i$ crosses infinitely many $\{H_n^i\}$, so hyperplanes of $F_i$ are compact.  By minimality of the UBS $\{H_n^i\}$, any element of $\overline{\mathfrak F_m}$ corresponding to a cone-vertex in $\fontact F_i$ crosses finitely many hyperplanes.  It follows that for all $n\geq 0$, there exists $N\geq n$ such that $H_n^i$ and $H_m^i$ cannot be adjacent to the same cone-vertex of $\fontact F_i$ when $m\geq N$.  Hence  $\boundary\fontact F_i=\{p_i\}$. 

We have shown that if $F\in\overline{\mathfrak F_m}$ has unbounded factored contact graph, then $F$ is (up to parallelism) contained in some $\cuco Y_{m_i}$.  If $F$ intersects only finitely many elements of $\{H_i\}$, then $F$ is compact and thus $\fontact F$ is bounded.  If $F$ intersects infinitely many, then it intersects all but finitely many, whence either $F$ is parallel to $F_i$ or $\fontact F$ contains a subgraph, containing all but finitely many hyperplane-vertices, whose vertices are all adjacent to the cone-point corresponding to $\gate_F(F_i)$; thus $\fontact F$ is bounded.  This completes the description of the boundaries of the factored contact graphs of the elements of $\overline{\mathfrak F_m}$.
\end{proof}

\begin{lem}\label{lem:discs}
Let $\mathfrak F$ be a factor system in $\cuco X$, let $\cuco Y\subseteq\cuco X$ be a convex subcomplex, and let $\mathfrak F'$ be the factor system in $\cuco Y$ consisting of all subcomplexes of the form $F'\cap\cuco Y$, where $F'\in\mathfrak F$.  Let $F\cap\cuco Y\in\mathfrak F'$, and suppose that if $F'\in\mathfrak F$ satisfies $F'\cap\cuco Y=F\cap\cuco Y$, then $F\nest F'$.  

Then the following map $\phi:\fontact(F\cap\cuco Y)\to\fontact F$ is a $(3,0)$--quasiisometric embedding: $\phi$ is the inclusion on contact graphs; for each $F'\cap\cuco Y\in\mathfrak F'$ properly nested in $F\cap\cuco Y$ (with $F'$ minimal with this intersection with $\cuco Y$), the cone-point in $\fontact(F\cap\cuco Y)$ corresponding to $F'\cap\cuco Y$ is sent to the cone-point of $\fontact\cuco X$ corresponding to $F'$.
\end{lem}

\begin{rem}
Recall from the discussion in the proof of Theorem~\ref{thm:simplicial_HHS} of the hieromorphism $(\cuco Y_m,\overline{\mathfrak F_m})\to(\cuco X,\overline{\mathfrak F})$ that if $F'\cap\cuco Y\nest F\cap\cuco Y$ and $F,F'$ are each $\nest$--minimal with the given intersections with $\cuco Y$, then $F\nest F'$.
\end{rem}

\begin{proof}[Proof of Lemma~\ref{lem:discs}]
Let $v,v'$ be vertices of $\fontact(F\cap\cuco Y_m)$.  Let $v=v_0,v_1,\ldots,v_n=v'$ be a geodesic sequence in $\fontact F$ from $v$ to $v'$.  If $v_i$ is a hyperplane vertex, let $H_i$ be the corresponding hyperplane of $F$ (so $H$ crosses $F\cap\cuco Y$).  If $v_i$ is a cone-vertex, let $H_i$ be a subcomplex in $\mathfrak F$, properly contained in $F$, that represents the parallelism class corresponding to the cone-vertex $v_i$.  (For $i\in\{0,n\}$, if $H_i$ is a hyperplane, then it crosses $\cuco Y$.  Otherwise, $H_i\in\mathfrak F$ is $\nest$--minimal among all $U\in\mathfrak F_F$ with $U\cap\cuco Y=H_i\cap\cuco Y$.)

If $H_i$ is a cone-vertex, then $H_{i\pm1}$ are hyperplanes crossing $H_i$.  This gives a sequence $H_0,H_1,\ldots,H_n$ of hyperplanes or factor-system elements in $F$ such that $\neb(H_i)\cap \neb(H_{i+1})\neq\emptyset$ when $H_i,H_{i+1}$ are hyperplanes, and $H_i\cap H_{i+1}\neq\emptyset$ when $H_{i+1}$ is a subcomplex in $\mathfrak F$.

For each $i$ such that $H_i\in\mathfrak F$, we have $H_i\propnest F$.  In particular, our minimality assumption on $F$ ensures that if $H_i\cap\cuco Y\neq\emptyset$, then $H_i\cap\cuco Y\propnest F\cap\cuco Y$.  Otherwise, we would have $H_i\cap \cuco Y=F\cap\cuco Y$ while $H_i\propnest F$, contradicting minimality of $F$.  

For each $i$ with $H_i$ a hyperplane, choose a combinatorial geodesic $\gamma_i\to\neb(H_i)$ joining the terminal point of $\gamma_{i=1}$ to a closest point on $H_{i+1}$ (or $\neb(H_{i+1})$ if $v_{i+1}$ is a hyperplane vertex).  Similarly, choose $\gamma_i\to H_i$ when $v_i$ is a cone-vertex.  The geodesic $\gamma_1\to H_1$ joins $H_1\cap\cuco Y$ (or $\neb(H_1)\cap\cuco Y$ to $H_1\cap H_2$ (or $\neb(H_1)\cap H_2$ etc.), and $\gamma_n\to H_n$ (or $\neb(H_n)$) is similarly chosen to end in $\cuco Y$.  Let $D\to F$ be a minimal-area disc diagram bounded by  $\gamma_1\cdot\gamma_2\cdots\gamma_n$ and a geodesic of $\cuco Y$ joining its endpoints.  Moreover, suppose that each of the geodesics, and indeed the sequence $v_0,\ldots,v_n$ and the representative subspaces, are chosen so as to minimize the area of $D$ among all possible such choices.

Then, arguing exactly as in the proof of Proposition~3.1 of~\cite{BehrstockHagenSisto:HHS_I}, we see that $\gamma_1\cdots\gamma_n$ can be chosen to be a geodesic since a minimal $D$ cannot contain a dual curve traveling from $\gamma_i$ to $\gamma_j$ for any $i,j$.  It follows that $\gamma_1\cdots\gamma_n$ lies in $\cuco Y$, so each $H_i$ that is a hyperplane either crosses $\cuco Y$ or contributes a combinatorial hyperplane to $\mathfrak F'$, while each $H_i$ that is a subcomplex contributes an element to $\mathfrak F'$; as explained above, for each such $H_i$, we have $H_i\cap\cuco Y\propnest F\cap\cuco Y$, so $H_i\cap\cuco Y$ corresponds to a cone-point in $\fontact(F\cap\cuco Y)$.  We thus have a sequence $H_1,\ldots,H_n$ of (non-$\nest$--maximal) elements of $\mathfrak F'$ and hyperplanes crossing $\cuco Y$, which determines a path of length between $n-1$ and $3(n-1)$ in $\fontact(F\cap\cuco Y)$.
\end{proof}

\begin{lem}\label{lem:product_extend}
Let $\cuco X$ be a CAT(0) cube complex with a factor system $\mathfrak F$.  Suppose that $A,B$ are unbounded convex subcomplexes of $\cuco X$ such that there is a cubical isometric embedding $A\times B\to\cuco X$ extending $A,B\hookrightarrow\cuco X$.  Then there exist $P_A,P_B\in\mathfrak F$ with $P_A\orth P_B$ and $A\subseteq P_A,B\subseteq P_B$.
\end{lem}

\begin{proof}
Let $x=A\cap B$.  Then $A,B$ are contained in combinatorial hyperplanes $H_A,H_B$, respectively.  Indeed, every hyperplane crossing $A$ (including the one whose carrier contains $H_B$) crosses every hyperplane crossing $B$ (including the one whose carrier contains $H_A$).  For each hyperplane $V'$ crossing $H_B$, let $V$ be one of the two associated combinatorial hyperplanes and consider $\gate_{H_A}(V)$.  Observe that $\gate_{H_A}(V)\in\mathfrak F$ since it contains $A$ and is thus unbounded.  Since $\mathfrak F$ has finite multiplicity, there are only finitely many distinct subcomplexes $\gate_{H_A}(V)$, as $V$ varies over all hyperplanes whose projection to $H_A$ contains $A$; let $P_A\in\mathfrak F$ be their intersection.  Define $P_B$ analogously.  Then $P_A,P_B$ have the desired properties.  (Indeed, a hyperplane $H$ crosses $P_A$ if and only if $H$ crosses every hyperplane $V$ whose projection to $H_A$ contains $A$; the projection of $H$ to $H_B$ thus contains $B$, so very hyperplane crossing $P_B$ crosses $H$, whence $P_A\times P_B\subset\cuco X$.)
\end{proof}

\renewcommand{\fontact}{\mathcal C}

\bibliographystyle{alpha}
\bibliography{hier_hyp}
\end{document}